\documentclass[reqno,10pt]{amsart}

\usepackage{a4wide}
\usepackage{amssymb,mathrsfs,color}
\usepackage[colorlinks=false]{hyperref}

\usepackage{cleveref} 
\crefformat{footnote}{#2\footnotemark[#1]#3} 

\title[Hyperboloidal evolution and global dynamics for the cubic wave equation]{Hyperboloidal evolution and global dynamics\\ for the focusing cubic wave equation}


\author[A.~Burtscher]{Annegret Y.\ Burtscher}
\address{Mathematical Institute and Hausdorff Center for Mathematics, University of Bonn, Endenicher Allee 60, 53115 Bonn, Germany}
\email{\href{mailto:burtscher@math.uni-bonn.de}{burtscher@math.uni-bonn.de}}
\urladdr{\href{http://annegretburtscher.wordpress.com}{http://annegretburtscher.wordpress.com}}

\author[R.~Donninger]{Roland Donninger}
\address{Mathematical Institute and Hausdorff Center for Mathematics, University of Bonn, Endenicher Allee 60, 53115 Bonn, Germany}
\address{Faculty of Mathematics, University of Vienna, Oskar-Morgenstern-Platz 1, 1090 Vienna, Austria}
\email{donninge@math.uni-bonn.de, roland.donninger@univie.ac.at}

\date{\today}

\numberwithin{equation}{section}

\theoremstyle{plain}
\newtheorem{theorem}{Theorem}[section]
\newtheorem{corollary}[theorem]{Corollary}
\newtheorem{proposition}[theorem]{Proposition}
\newtheorem{lemma}[theorem]{Lemma}

\theoremstyle{definition}
\newtheorem{definition}[theorem]{Definition}

\theoremstyle{remark}
\newtheorem{remark}[theorem]{Remark}

\usepackage{graphicx}
\usepackage{calc} 

\newcommand{\vol}{\operatorname{vol}}
\newcommand{\eps}{{\tilde \varepsilon}}

\def\L{\mathbf{L}} 
\def\wL{\widetilde \L} 
\def\hL{\widehat \L} 
\def\N{\mathbf{N}} 

\def\S{\mathbf{S}}
\def\T{\mathbf{T}}
\def\R{\mathbf{R}} 
\def\P{\mathbf{Q}} 
\def\Q{\mathbf{P}} 
\def\p{q} 
\def\q{p} 
\def\wP{\widetilde \Q} 
\def\I{\mathbf{I}} 
\def\K{\mathbf{K}} 
\def\wK{\widetilde \K} 
\def\C{\mathbf{C}} 
\def\wC{\widetilde \C}
\def\G{\mathbf{G}} 
\def\F{\mathbf{F}} 

\newcommand\Z[2]{\mathbf{Z}(#1,#2)}

\def\RR{\mathbb{R}}

\def\CC{\mathbb{C}}
\def\SS{\mathbb{S}} 
\def\BB{\mathbb{B}} 

\def\A{{\mathcal{A}^\varepsilon}} 
\def\Ad{{\mathcal{A}^\varepsilon_\delta}}
\def\cC{\mathcal{C}} 
\def\H{\mathcal{H}}
\def\cB{\mathcal{B}} 
\def\D{\mathcal{D}}
\def\X{{\mathcal{X}^\varepsilon}} 
\def\Xd{{\mathcal{X}^\varepsilon_\delta}}
\def\M{\mathcal{M}} 

\def\graph{\mathrm{graph}}
\def\ii{i} 
\def\rg{\operatorname{rg}} 
\def\eetilde{ - \frac{1}{2} + \eps} 
\def\beeptilde{ \left( \frac{1}{2} - \eps \right)} 
\def\eep{\frac{1}{2}-\varepsilon} 
\def\beep{\left(\frac{1}{2}-\varepsilon\right)} 
\def\dbeep{(1-2\varepsilon)} 

\begin{document}


\begin{abstract}
The focusing cubic wave equation in three spatial dimensions has the explicit solution $\sqrt{2}/t$. We study the stability of the blowup described by this solution as $t \to 0$ without symmetry restrictions on the data. Via the conformal invariance of the equation we obtain a companion result for the stability of slow decay in the framework of a hyperboloidal initial value formulation. More precisely, we identify a codimension-1 Lipschitz manifold of initial data leading to solutions which converge to Lorentz boosts of $\sqrt{2}/t$ as $t\to\infty$. These global solutions thus exhibit a slow nondispersive decay, in contrast to small data evolutions.
\end{abstract}

\maketitle


\section{Introduction}
\label{intro}

\noindent We consider the cubic wave equation
\begin{align}\label{cwe}
 \Box v(t,x) + v(t,x)^3 = 0, \quad x \in \RR^3,
\end{align}
 with wave operator $\Box := - \partial_t^2 + \Delta_x$ and conserved energy 
\[
 \frac{1}{2} \int_{\RR^3} \left (|\partial_t v(t,x)|^2 + |\nabla_x v(t,x)|^2 \right ) dx 
 - \frac{1}{4}\int_{\RR^3} |v(t,x)|^{4} \, dx.
\]
 It is well known that solutions with small $\dot H^1(\RR^3) \times L^2(\RR^3)$-norm exist globally and scatter to zero \cite{P}, whereas for large initial data nonlinear effects dominate and generically lead to finite-time blowup \cite{BCT,DS12,G,L,MZ}. The explicit solution $v_0(t,x) = \sqrt{2} / t$ of \eqref{cwe}, on the other hand, exists globally for $t\geq 1$ but is nondispersive. Since it has infinite energy, its role in the dynamics is a priori unclear. Nevertheless, solutions with such an exceptionally slow decay have been observed in numerical simulations \cite{BZ} and are conjectured to sit on the threshold between dispersion and finite-time blowup. The main purpose of this paper is to develop a rigorous understanding of this phenomenon.

 To this end, we follow \cite{BZ, DZ} and utilize the conformal invariance of the cubic wave equation~\eqref{cwe} to reduce the question of stability of $v_0(t,x)$ as $t\to\infty$ to the stability analysis of the blowup described by $v_0(t,x)$ as $t\to 0$. In fact, the conformal transformation leads to a natural \emph{hyperboloidal initial value formulation} in which the stability of $v_0(t,x)$ as $t\to\infty$ can be studied. In what follows, we carry out this transformation, and then state and review our Theorem~\ref{maintheorem-blowup} for the blowup stability. We then return to the decay picture. After introducing the hyperboloidal initial value problem of interest, we state two decay results for global solution of \eqref{cwe}. Our main Theorem~\ref{maintheorem} verifies the codimension-1 stability of slow decay.

\subsection*{The blowup result}\label{ssec1.1}
 The cubic wave equation~\eqref{cwe} in $\RR^3$ is conformally invariant. This invariance is expressed in terms of the (time-reversed) \emph{Kelvin transform} $\kappa$, which reads
\begin{align}\label{Kt}
 (T,X) &\mapsto \kappa_T (X) = \kappa(T,X)=\left( - \frac{T}{T^2 - |X|^2}, \frac{X}{T^2 - |X|^2} \right).
\end{align}
 Note that the inverse transform $\kappa^{-1}$ from $(t,x)$ to $(T,X)$ is of the same form, i.e.,
\begin{align}\label{coord1}
 \kappa^{-1}(t,x) = \kappa_t(x) = \left( - \frac{t}{t^2 - |x|^2}, \frac{x}{t^2 - |x|^2} \right),
\end{align}
 and we moreover have the identity
\begin{align*}
 t^2 - |x|^2 = (T^2 - |X|^2)^{-1}.
\end{align*}
 The coordinate transformation $\kappa^{-1}$ maps the future light cone $\triangledown_0 := \{ (t,x) \, | \, t>0, |x| < t \}$ to the past light cone $\vartriangle_0 := \{ (T,X) \, | \, T < 0 , |X| < - T \}$.
A straightforward computation now shows the following invariance property.

\begin{lemma} \label{lem:transform1}
 The function
\[
 u(T,X) = \frac{v \circ \kappa_T(X)}{T^2 - |X|^2} = \frac{1}{T^2 - |X|^2} v \left( - \frac{T}{T^2 - |X|^2}, \frac{X}{T^2 - |X|^2} \right)
\]
 solves the cubic wave equation
\begin{align}\label{cwe-u}
 \Box u(T,X) + u(T,X)^3 = 0
\end{align}
 on the past light cone $\vartriangle_0$ if and only if $v$ solves the same equation \eqref{cwe} on the future light cone $\triangledown_0$.
\end{lemma} 

 As a consequence of this conformal invariance of \eqref{cwe}, the stability analysis of
 $v_0(t,x)$ as $t \to \infty$ translates into the question of stability of the blowup solution
\[
 u_0 (T,X) = \frac{v_0 \circ \kappa_T (X)}{T^2 - |X|^2} = \frac{\sqrt{2}}{-T}
\]
 of \eqref{cwe-u} in the backward light cone $\vartriangle_0$. In this \emph{blowup picture} we establish the following codimension-1 stability result for the blowup described by $u_0$ and its Lorentz boosts $u_a$ with rapidity $a \in \RR^3$ (see Section~\ref{ssec:lorentz} for the exact definition). We denote by $B_{r}$ the open ball of radius $r$ in $\RR^3$ and $\BB:=B_1$.

\begin{theorem}[Blowup stability]\label{maintheorem-blowup}
 There exists a codimension-1 Lipschitz manifold $\M$ of initial data in $H^1(\BB)\times L^2(\BB)$, with $(0,0) \in \M$, such that the Cauchy problem
 \begin{align*}
  \Box u(T,X) + u(T,X)^3 &= 0, \nonumber \\
  u(-1,.) &= u_0(-1,.) + F, \\
  \partial_T u (-1,.) &=  \partial_T u_0 (-1,.) + G, \nonumber
 \end{align*}
 with $(F,G) \in \M$ has a unique solution $u$ (in the Duhamel sense) on the truncated lightcone $\vartriangle_0 \cap \{(T,X): T\in [-1,0)\}$. For a unique rapidity $a \in \RR^3$ and the corresponding Lorentz-boosted $u_0$, denoted by $u_a$, we have
\begin{align*}
 |T|^{\frac{1}{2}}\| u(T,.) - u_{a}(T,.) \|_{\dot H^1(B_{|T|})}&\lesssim |T|^{\frac12-}, \\
 |T|^{-\frac12} \| u(T,.)-u_a(T,.) \|_{L^2(B_{|T|})}&\lesssim |T|^{\frac12-}, \\
 |T|^\frac12 \| \partial_T u(T,.) - \partial_T u_{a}(T,.) \|_{L^2(B_{|T|})}&\lesssim |T|^{\frac{1}{2}-},
\end{align*}
 for all $T \in [-1,0)$.
\end{theorem}

\begin{remark}
Observe that
\begin{align*}
 \|u_0(T,.)\|_{L^2(B_{|T|})}&\simeq |T|^\frac12, \\
 \|\partial_T u_0(T,.)\|_{L^2(B_{|T|})}&\simeq |T|^{-\frac{1}{2}},
\end{align*}
and thus, the normalization factors on the left hand sides are natural.
In particular, the solution $u$ converges to the Lorentz-boosted $u_0$ as the blowup time is approached.
\end{remark}

\begin{remark}
The instability of the blowup comes from the fact that general perturbations of $u_0$ will change the blowup time. Consequently, since the manifold has codimension one, it follows that the blowup profile is stable up to time translation (and Lorentz boosts). In this sense, the instability is not ``real''. In fact, one could include a Lyapunov--Perron-type argument as in \cite{DS} to get rid of the codimension-1 condition and vary the blowup time instead. However, for the decay picture it turns out that the result in its present form is more useful.
\end{remark}

\begin{remark}
Theorem \ref{maintheorem-blowup} is closely related to the seminal work by Merle and Zaag \cite{MZ, MZ1} which established the universality of the blowup \emph{speed} for the conformal wave equation. In the subconformal case they also proved the stability of the blowup \emph{profile} \cite{MZ2, MZ3}.
A similar stability result for superconformal equations was obtained in \cite{DS}, but in a stronger topology.
\end{remark}

\subsection*{The decay result}\label{ssec1.2}
 Geometrically, the conformal invariance naturally leads to a hyperboloidal initial value problem for equation \eqref{cwe}. For $T \in (-\infty,0)$ we consider the spacelike slices $\Sigma_T$ defined by
\begin{align*}
 \Sigma_T :=  \left\{ \kappa_{T}(X) \, \Big| \, X \in B_{|T|} \right\} =
\left\{ (t,x) \in (0,\infty) \times \RR^3 \, \Big| \, \frac{t^2-|x|^2}{t} = \frac{1}{(-T)} \right\}.
\end{align*}
 These slices provide a foliation of the future light cone $\triangledown_0$. The hyperboloidal slices $\Sigma_T$ in the decay picture are the pre-images of constant time $T<0$ slices in the blowup picture, see Figure~\ref{fig2a}. We emphasize that the slices $\Sigma_T$ are asymptotic to \emph{different} light cones and hence ``foliate'' future null infinity. As a consequence, energy can escape to infinity and this provides the crucial stabilizing mechanism. 

\begin{figure}[h]
\centering
\def\svgwidth{.45\columnwidth}
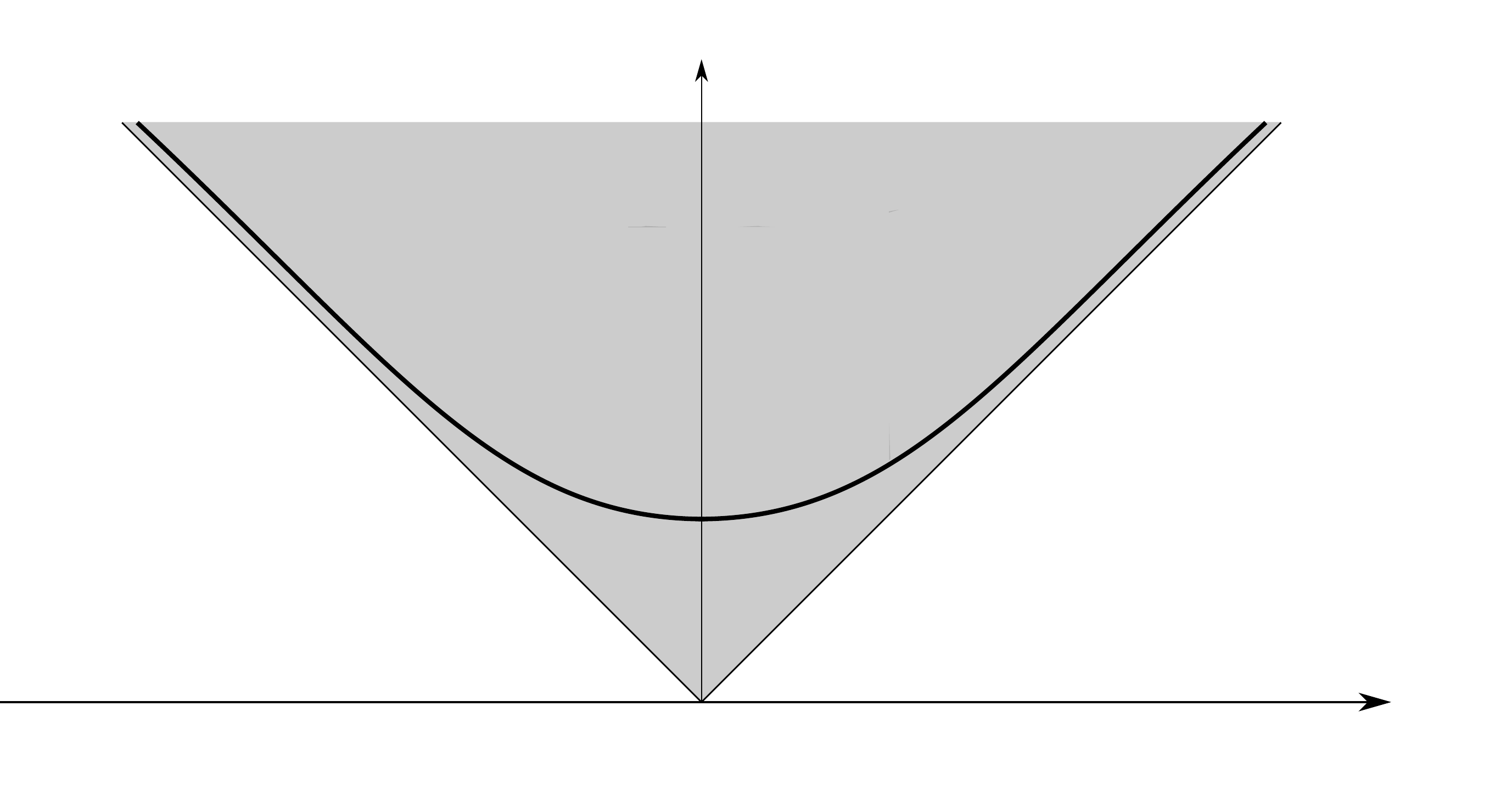
\def\svgwidth{.45\columnwidth}
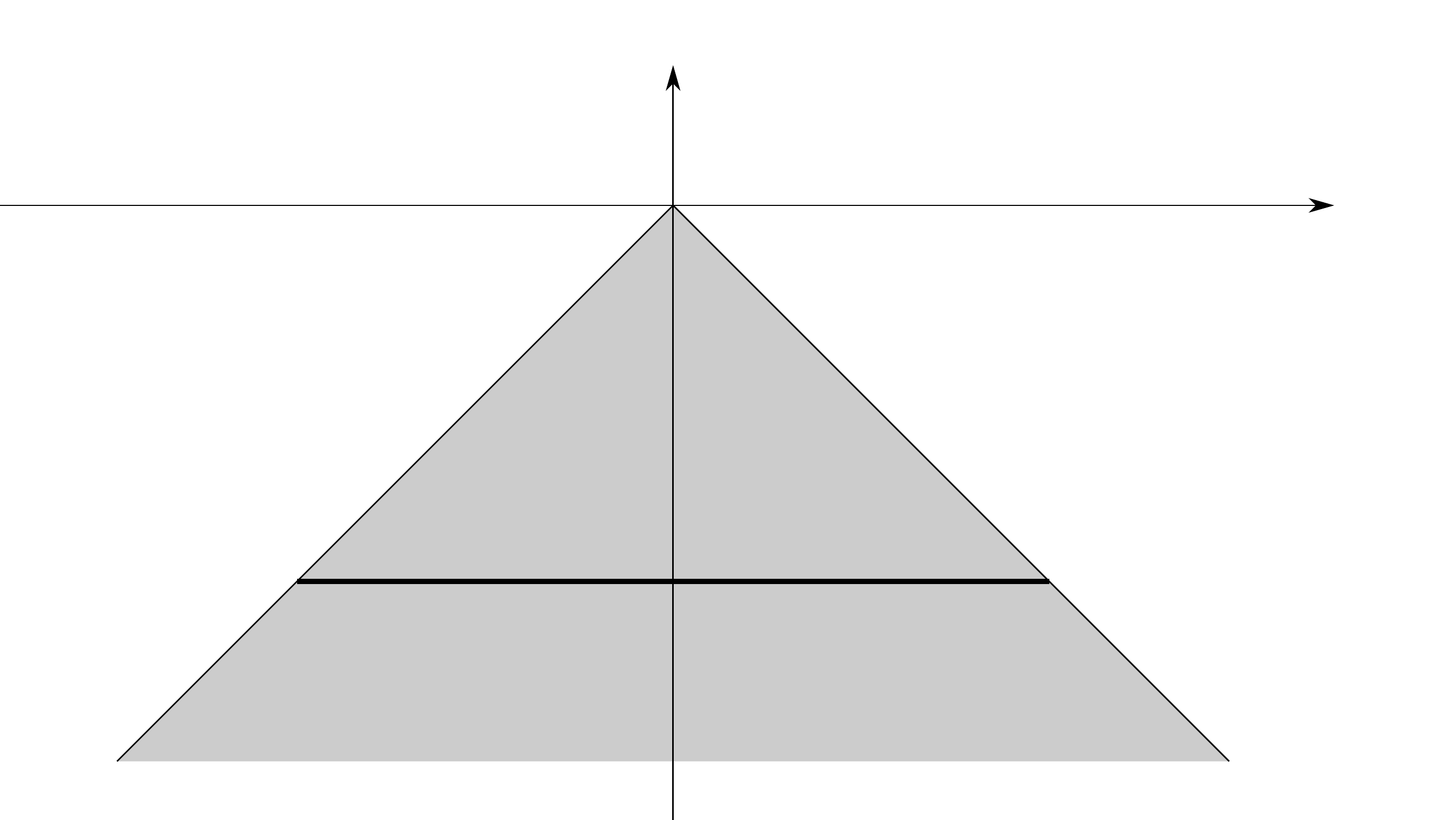
\caption{Domain of dependence in Cartesian coordinates $(t,x)$ and hyperboloidal coordinates $(T,X)$.}
\label{fig2a}
\end{figure}

\begin{figure}[h]
\centering
\def\svgwidth{.8\columnwidth}
\hspace*{1cm} 
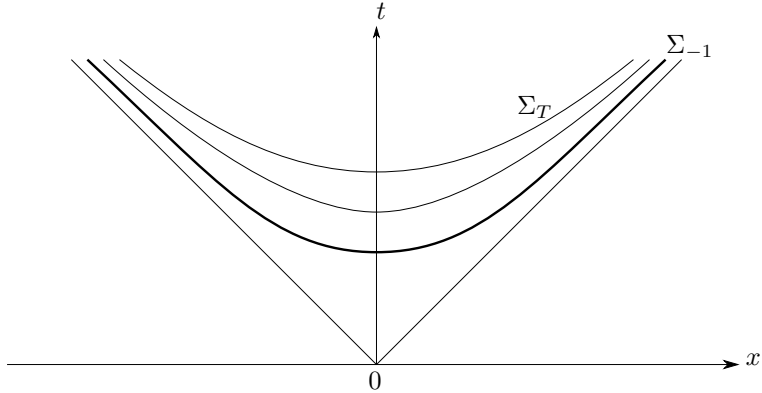
\caption{The initial value problem for the cubic wave equation is studied using a hyperboloidal foliation $(\Sigma_T)_T$, $T \in (-\infty,0)$, of the future light cone emanating from the origin. Initial data are prescribed on the spacelike hyperboloid $\Sigma_{-1}$.}
\label{fig1}
\end{figure}

 In the hyperboloidal initial value formulation the data are prescribed on the spacelike hyperboloid $\Sigma_{-1}$, see Figure~\ref{fig1}. These initial data consist of a function $v$ in $H^1(\Sigma_{-1})$ and a derivative $\nabla_n v$ normal to $\Sigma_{-1}$ in $L^2(\Sigma_{-1})$. The function spaces $H^1(\Sigma_{-1})$ and $L^2(\Sigma_{-1})$, as well as the derivative $\nabla_n$, are naturally transferred from the blowup picture and defined as follows.

\begin{definition}[Function spaces on hyperboloids]  \label{def:norms}
 On each hyperboloidal slice $\Sigma_T$ we define the norms
\begin{align*}
 \| v \|^2_{L^2(\Sigma_T)} &:= \int_{B_{|T|}} \left| \frac{v \circ \kappa_T(X)}{T^2 - |X|^2} \right|^2 dX, &
 \| v \|^2_{\dot H^1(\Sigma_T)} &:= \int_{B_{|T|}} \left| \nabla_X \frac{v \circ \kappa_T(X)}{T^2-|X|^2} \right|^2 dX,
\end{align*}
 and
\begin{align*}
 \| v \|^2_{H^1(\Sigma_T)} := \| v \|^2_{\dot H^1(\Sigma_T)} + |T|^{-2} \| v \|^2_{L^2(\Sigma_T)}.
\end{align*}
\end{definition}

\begin{definition}[Normal derivative] \label{def:nabla_n}
 The differential operator $\nabla_n$ is defined by the relation
 \begin{align*}
  \frac{\nabla_n v \circ \kappa_T (X)}{T^2 - |X|^2} = \partial_T \frac{(v \circ \kappa_T)(X)}{T^2-|X|^2},
 \end{align*}
 and explicitly given by
 \[
  \nabla_n v(t,x) = (t^2 + |x|^2)\partial_t v(t,x) + 2 t x^j \partial_j v(t,x) + 2tv(t,x).
 \]
\end{definition}

\begin{remark}
 The principal term $(t^2 + |x|^2)\partial_t v + 2 t x^j \partial_j v$ of $\nabla_n v$ is the pullback of the vector field $\partial_T$ along the inverse Kelvin transform $\kappa^{-1}$, and well-known as the Morawetz multiplier $K_0$. The zeroth order term $2tv$ appears due to the weight $(T^2-|X|^2)^{-1}$ in the transformation.
\end{remark}

 The result in the decay picture reads as follows. For the domain appearing in the Strichartz norm we refer to Figure~\ref{fig2}.
\begin{theorem}[Stability of slow decay] \label{maintheorem} 
 There exists a codimension-1 Lipschitz manifold $\M$ of initial data in $H^1(\Sigma_{-1}) \times L^2 (\Sigma_{-1})$, with $(0,0) \in \M$, such that the hyperboloidal initial value problem
 \begin{align*}
  \Box v(t,x) + v(t,x)^3 &= 0, \nonumber \\
  v |_{\Sigma_{-1}} &= v_0 |_{\Sigma_{-1}} + f \\
  \nabla_n v |_{\Sigma_{-1}} &=  \nabla_n v_0 |_{\Sigma_{-1}} + g, \nonumber
 \end{align*}
 with $(f,g) \in \M$ and $v_0(t,x) = \sqrt{2} /t$, has a unique solution $v$ (in the Duhamel sense) defined on the future domain of dependence $D^+ (\Sigma_{-1})$. For a unique rapidity $a \in \RR^3$ and the corresponding Lorentz-boosted $v_0$, denoted by $v_a$, we have the decay
\[
  |T|^\frac{1}{2} \left( \| v - v_a \|_{H^1(\Sigma_T)} + \| \nabla_n v - \nabla_n v_a \|_{L^2(\Sigma_T)} \right) \lesssim |T|^{\frac{1}{2}-}
\]
 for all $T \in [-1,0)$. Moreover, for any $\delta \in (0,1)$, the decay in Cartesian coordinates is\footnote{\label{note1}The $L^4_t L^4_x$-Strichartz norm in Theorem~\ref{maintheorem} and Theorem~\ref{smalldatatheorem} can be replaced by other Strichartz norms (see Remark~\ref{otherpq}).}
 \begin{align*}
  \| v - v_a \|_{L^4(t,2t)L^4(B_{(1-\delta)t})} \lesssim t^{-\frac{1}{2}+}
 \end{align*}
 as $t\to\infty$.
 \end{theorem}

\begin{remark}
 As with the blowup result, the normalizing factor on the left-hand side reflects the behavior of the solution $v_a$ in the respective norm, i.e., we have
\[ \| v_a \|_{H^1(\Sigma_T)} + \| \nabla_n v_a \|_{L^2(\Sigma_T)} \simeq |T|^{-\frac{1}{2}} \]
 for all $T\in [-1,0)$.
\end{remark}

\begin{remark}
 Contrary to the blowup result, the instability is now ``real'' in the sense that generic evolutions will either disperse (i.e., decay faster, see below) or blow up in finite time. This is easily understood by noting that solutions in the blowup picture with blowup time \emph{larger} than $0$ correspond to dispersive solutions in the decay picture. On the other hand, solutions in the blowup picture with blowup time \emph{less} than $0$ correspond to finite-time blowup in the decay picture. Only those solutions in the blowup picture that blow up precisely at time $T=0$ lead to slow decay in the decay picture. 
\end{remark}

\begin{figure}[h]
\centering
\def\svgwidth{\columnwidth}
\hspace*{1cm} 
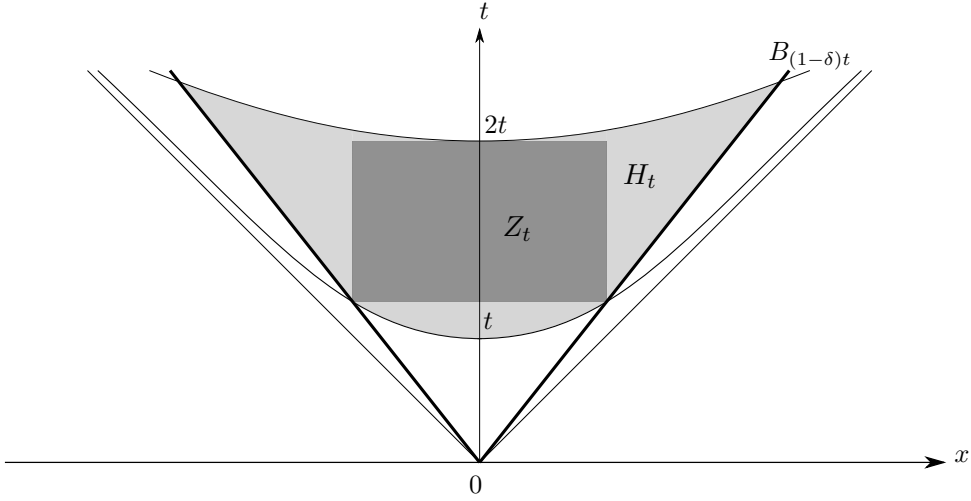
\caption{The Strichartz norms in Cartesian coordinates $(t,x)$ are estimated on cylindrical regions $Z_t := (t,2t) \times B_{(1-\delta)t}$, for any $\delta \in (0,1)$. A hyperboloidal domain covering this cylindrical region is, for example, given by the region $H_t$ between the two hyperboloidal slices $\Sigma_{- \frac{1}{t(2\delta - \delta^2)}}$ and $\Sigma_{- \frac{1}{2t}}$.}
\label{fig2}
\end{figure}

For quantitative comparison with the small data evolution we finally note the following result.

\begin{theorem}[Small data dispersive decay] \label{smalldatatheorem}
 There exists an $\varepsilon > 0$ such that the hyperboloidal initial value problem
\begin{align*}
 \Box v(t,x) + v(t,x)^3 &= 0, \nonumber \\
 v |_{\Sigma_{-1}} &= f, \\
 \nabla_n v |_{\Sigma_{-1}} &=  g. \nonumber
\end{align*}
 for initial data $(f,g) \in H^1(\Sigma_{-1})\times L^2(\Sigma_{-1})$ with $\| (f,g) \|_{H^1\times L^2(\Sigma_{-1})} < \varepsilon$ has a unique global solution $v$ (in the Duhamel sense) in the future domain of dependence $D^+(\Sigma_{-1})$ which, for any $\delta\in (0,1)$, satisfies\cref{note1}
\[
 \| v \|_{L^4(t,2t) L^4(B_{(1-\delta)t})} \lesssim t^{-\frac{1}{2}}.
\]
\end{theorem}

\begin{remark}
 Theorem~\ref{smalldatatheorem} and the fact that
\[
 \| v_a \|_{L^4 (t,2t) L^4 (B_{(1-\delta)t})} \simeq 1
\]
 show that the decay of the solutions constructed in Theorem~\ref{maintheorem} is slower than the decay of generic small data evolutions. In particular, the solutions from Theorem~\ref{maintheorem} are nondispersive.
\end{remark}

\subsection*{Additional remarks}
 Our main Theorem~\ref{maintheorem} rigorously establishes the codimension-1 stability of $v_0(t,x) = \sqrt{2}/t$ without symmetry assumptions, which was numerically observed in~\cite{BZ} in the case of spherical symmetry. A direct precursor of Theorem~\ref{maintheorem} is \cite{DZ}, where the codimension-4 stability of slow decay was established. The additional three unstable directions are caused by the Lorentz symmetry. In the present paper we use modulation theory to deal with this issue. In its present form, our result crucially relies on the conformal invariance of equation~\eqref{cwe} which necessitates the cubic power in three spatial dimensions. In general dimensions $d$, the conformally invariant wave equation is
\[ \Box v(t,x)\pm v(t,x)|v(t,x)|^{\frac{4}{d-1}}=0 \]
 with the explicit solution $v_0(t,x)=c_d \, t^{-\frac{d-1}{2}}$ ($c_d$ is a suitable constant). Our methods can be generalized to this situation. It is an interesting open question whether similar results can be proved for equations that are not conformally invariant.  
 We also hope that our paper is interesting from the general perspective of hyperboloidal methods which receive increasing attention (see, for example, \cite{DR,ES,F,Kl,LM,LM1,LS,St,W,Z}).

\subsection*{Outline}
 An overview of the methods and organization of the paper follows. We go through the proof of Theorem~\ref{maintheorem}, the adjustments for the proof of Theorem~\ref{smalldatatheorem} are explained and carried out in Section~\ref{ssec4.6}. The proof of Theorem~\ref{maintheorem-blowup} is part of that of Theorem~\ref{maintheorem}, see also Remark~\ref{rem:transformation}.
 
 Via the Kelvin transform we have seen that equation~\eqref{cwe} is exactly the same focusing cubic wave equation~\eqref{cwe-u} in hyperboloidal coordinates, however, instead of treating the asymptotics of global solutions in the future light cone we are led to look at solutions in the past light cone of the origin.
 In the preliminary Section~\ref{sec1} we further transform the cubic wave equation~\eqref{cwe-u} in hyperboloidal coordinates $(T,X)$ to similarity coordinates $(\tau,\xi)$.
 Similarity coordinates are a natural choice of coordinates for the selfsimilar solution $v_0$ which simply transforms to the constant solution $\psi_0 = \sqrt{2}$. Moreover, we introduce the Lorentz boosts of $v_0$, that is, the solutions $v_a$ of \eqref{cwe}.
 
 In Section~\ref{sec2} we rewrite the cubic wave equation as an evolution system of the form
 \[
 \partial_\tau \Psi(\tau) = \L \Psi(\tau) + \N (\Psi(\tau)),
\]
 where $\L$ is a linear operator and $\N$ is nonlinear. To account for the Lorentz symmetry we use a modulation ansatz
 \[
  \Psi(\tau) = \Psi_{a(\tau)} + \Phi(\tau)
 \]
 around the Lorentz transformations $\Psi_a$ of the selfsimilar solution $\Psi_0$. We allow for the (unknown) rapidity $a$ to depend on $\tau$, set $a(0)=0$ initially and assume (and later verify) that $a_\infty := \lim_{\tau \to \infty} a(\tau)$. This ansatz leads to an equivalent description as an evolution system for the perturbation term $\Phi$, i.e.,
\begin{align}\label{Phisystem}
 \partial_\tau \Phi(\tau) - \L \Phi(\tau) - \L_{a_\infty}' \Phi(\tau) = [\L_{a(\tau)}' - \L_{a_\infty}'] \Phi(\tau) + \N_{a(\tau)} (\Phi(\tau)) - \partial_\tau \Psi_{a(\tau)},
\end{align}
 where $\L'_a$ denotes the linearized part of the nonlinearity $\N$ and $\N_a$ the remaining full nonlinearity.
 
 In Section~\ref{sec3} we study the linearized part of the system \eqref{Phisystem}, that is
\begin{align}\label{Philinear}
 \partial_\tau \Phi(\tau) = [ \L + \L_{a_\infty}'] \Phi(\tau),
\end{align}
 and control the asymptotics of the solutions. To this end we employ semigroup theory and spectral theory. More precisely, the operator $\L$ generates a strongly continuous semigroup $\S$, and, since $\L'_{a_\infty}$ is bounded, there also exists a semigroup $\S_{a_\infty}$ generated by $\L_{a_\infty} := \L + \L_{a_\infty}'$. A careful analysis of the spectrum of $\L_{a_\infty}$ yields decay estimates for the linearized evolution \eqref{Philinear}.
 
 Finally, the nonlinear terms are controlled by standard Sobolev embedding, and the full nonlinear equation is solved by several fixed point arguments in Section~\ref{sec4}. For this purpose we first rewrite equation~\eqref{Phisystem} with $\Phi(0)=u$ as a weak integral equation
 \begin{align}\label{inteqK}
  \Phi(\tau) = \S_{a_\infty}(\tau) u - \int_0^\tau \S_{a_\infty} (\tau-\sigma) \left[ (\L_{a(\sigma)}'-\L_{a_\infty}') \Phi(\sigma) + \N_{a(\sigma)} (\Phi(\sigma)) - \partial_\sigma \Psi_{a(\sigma)} \right]  d\sigma 
 \end{align}
 using Duhamel's principle. The terms in the integrand are shown to be small and Lipschitz continuous with respect to $a$ and $\Phi$. The instabilities arising from the Lorentz symmetry of the cubic wave equation are suppressed by choosing the rapidity $a$ in a suitable way. In contrast, the time-translation instability is isolated by adding a correction term $\C_u(\Phi,a)$ and first solving a modified weak equation of the form
 \[
  \Phi(\tau) = \K_u(\Phi,a)(\tau) - \S_{a_\infty}(\tau) \C_u(\Phi,a), 
 \]
 where $\K_u(\Phi,a)(\tau)$ denotes the right hand side of \eqref{inteqK}, by means of contraction arguments. Solutions to this modified equation with vanishing correction term thus satisfy the original equation~\eqref{inteqK}. The condition $\C_u(\Phi,a) =0$ is shown to describe a codimension-1 manifold $\M$ of initial data, which we locally represent as a graph of a Lipschitz function.

\subsection*{Notation}
 By $\BB$ we denote the open unit ball and by $\SS$ the unit sphere in $\RR^3$. We write $B_r$ for the open ball with radius $r$ around the origin. The domain of an (unbounded) operator $\T$ is written as $\D(\T)$. The spectrum of a linear operator $\L$ is denoted by $\sigma(\L)$, the point spectrum by $\sigma_p(\L)$. Its resolvent is the operator $\R_\L$, i.e., $\R_\L(\lambda) = (\lambda - \L)^{-1}$ for $\lambda$ in the resolvent set $\rho(\L) = \CC \setminus \sigma(\L)$. We assume that $\delta,\varepsilon>0$ are generic and small, however implicit constants may depend on them. Einstein's summation convention is used throughout the manuscript. This means that if an index appears twice in a summation term (once as subscript and once as superscript), then we automatically sum over all values of that index, e.g., $x^i e_i$ is written instead of $\sum_{i=1}^3 x^i e_i$. Finally, the notation $a \lesssim b$ indicates that there exists a constant $C>0$ (possibly depending on a parameter) such that $a \leq C b$. If 
$a \lesssim b$ and $b \lesssim a$ holds, then we simply write $a \simeq b$. The decay estimate $A(t)\lesssim t^{-\frac{1}{2}+}$ of Theorem~\ref{maintheorem} means that for any $\varepsilon>0$ there exists a constant $C_\varepsilon>0$ such that $A(t)\leq C_\varepsilon t^{-\frac{1}{2}+\varepsilon}$.


\section{Preliminaries}
\label{sec1}

 We have reformulated the original cubic wave equation~\eqref{cwe} on the future light cone $\triangledown_0$ of the origin stated in Cartesian coordinates $(t,x)$ as a problem in hyperboloidal coordinates $(T,X)$ on the past light cone $\vartriangle_0$. Since we are interested in solutions close to the selfsimilar solution $v_0(t,x) = \sqrt{2} / t$, we further employ selfsimilar coordinates $(\tau,\xi)$, and we obtain an equivalent second order equation. In Section~\ref{sec2} this equation is then further transformed to an evolution problem of first order in $\tau$.

\subsection{The equation in similarity coordinates}
\label{ssec:sim}

 Since $u_0(T) = - \sqrt{2} / T$ is a selfsimilar solution, it is natural to employ the similarity coordinates
\begin{align}\label{coord2}
 \tau := - \log (-T) = -\log \frac{t}{t^2 - |x|^2}, \qquad \xi := \frac{X}{-T} = \frac{x}{t},
\end{align}
with $\tau \geq 0$ and $|\xi| < 1$, see Figure~\ref{fig2b}. The inverse transforms read
\begin{align*}
 T = - e^{-\tau}, \qquad X = e^{-\tau} \xi
\end{align*}
and 
\begin{align*}
 t = \frac{e^\tau}{1-|\xi|^2}, \qquad x = \frac{e^\tau \xi}{1-|\xi|^2},
\end{align*}
 respectively.
 The cubic wave equation transforms accordingly.

\begin{figure}[h]
\centering
\hspace*{.5cm}
\def\svgwidth{.5\columnwidth}
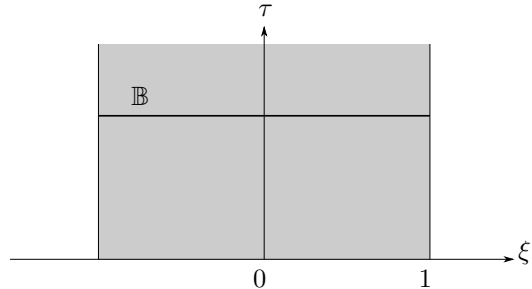
\caption{Domain of dependence in similarity coordinates $(\tau,\xi)$.}
\label{fig2b}
\end{figure}

\begin{lemma}\label{lem:transform2}
 The rescaled function
\[
 \psi(\tau,\xi) := e^{-\tau} u (-e^{-\tau},e^{-\tau} \xi)
\]
 solves the equation
\begin{align}\label{cwe-psi}
 \left[ \partial_\tau^2 + 3 \partial_\tau + 2 \xi^j \partial_{\xi^j} \partial_\tau - (\delta^{jk} - \xi^j \xi^k) \partial_{\xi^j} \partial_{\xi^k} + 4 \xi^j \partial_{\xi^j} + 2 \right] \psi(\tau,\xi) = \psi(\tau,\xi)^3
\end{align}
 on the domain $\{ (\tau,\xi) \in \RR \times \RR^3 \, | \, \tau \geq 0, |\xi| < 1 \}$ if and only if $u$ solves the cubic wave equation \eqref{cwe-u} on the past light cone $\vartriangle_0 \cap \{ T \geq -1 \}$ of the origin.
\end{lemma}

\begin{proof}
 The Jacobi matrix of the transformation from $(T,X)$ to $(\tau,\xi)$ reads, for $i,j \in \{1,2,3\}$,
\begin{align*}
 \begin{pmatrix}
  \frac{\partial \tau}{\partial T} & \frac{\partial \tau}{\partial X^i} \\
  \frac{\partial \xi^j}{\partial T} & \frac{\partial \xi^j}{\partial X^i}
 \end{pmatrix}
 = \begin{pmatrix}
    e^\tau & 0 \\
   \xi^j e^\tau & \delta^{ij} e^\tau
   \end{pmatrix}.
\end{align*}
 This implies that
\begin{align}\label{eq:dT}
 \partial_T u(T,X) = e^{2\tau} \left(1 + \partial_\tau + \xi^j \partial_{\xi^j} \right)  \psi(\tau,\xi)
\end{align}
 and
\[
 \partial_T^2 u(T,X) = e^{3 \tau} \left[ 2 + 3 \partial_\tau + 4 \xi^j \partial_{\xi^j}
+ \partial_\tau^2 + 2 \xi^j \partial_\tau \partial_{\xi^j} + \xi^j \xi^k \partial_{\xi^j} \partial_{\xi^k} \right] \psi(\tau,\xi),
\]
 as well as
\[
 \partial_{X^i} u(T,X) = e^{2\tau} \partial_{\xi^i} \psi(\tau,\xi), \qquad \partial_{X^i}^2 u(T,X) = e^{3\tau} \partial^2_{\xi^i} \psi(\tau,\xi).
\]
 Consequently, $u$ being a solution of $\Box u(T,X) + u(T,X)^3 =0$ is equivalent to $\psi(\tau,\xi)$ being a solution to the equation
\[
 e^{3 \tau} \left[ 2 + 3 \partial_\tau + 4 \xi^j \partial_{\xi^j}
+ \partial_\tau^2 + 2 \xi^j \partial_\tau \partial_{\xi^j} + (\xi^j \xi^k - \delta^{jk} ) \partial_{\xi^j} \partial_{\xi^k} \right] \psi(\tau,\xi) = e^{3\tau} \psi(\tau,\xi)^3,
\]
 which is \eqref{cwe-psi}.
\end{proof}

 The fundamental selfsimilar solution of \eqref{cwe-psi} is the constant solution $\psi_0(\tau,\xi) = \sqrt{2}$.

\subsection{The 3-parameter family of the selfsimilar solutions}
\label{ssec:lorentz}

 The cubic wave equation~\eqref{cwe-u} is invariant under time translations and Lorentz transformations. The time translation symmetry yields the one-parameter family $u_{T_0} (T,X) = \sqrt{2}/(T_0-T)$ of solutions. Moving the blowup time of $u_0 = - \sqrt{2}/T$ away from $T_0=0$ leads to either finite-time blowup ($T_0<0$) or dispersion ($T_0>0$). On the other hand, the blowup surface $\{ T \equiv 0 \}$ of $u_0$ is invariant under spatial translations and rotations. We fix the origin and allow for hyperbolic rotations of $u_0$ by applying Lorentz boosts for each direction. For the rapidity $a = (a^1,a^2,a^3) \in \RR^3$ the Lorentz transformation $\Lambda(a) = \Lambda^3(a^3)\Lambda^2(a^2)\Lambda^1(a^1)$ is given by
\begin{align*}
 \Lambda^j(a^j) \colon \RR \times \RR^3 &\to \RR \times \RR^3, \qquad j \in \{ 1,2,3 \}, \\
\begin{pmatrix}
 T \\ X^k
\end{pmatrix} &\mapsto
\begin{pmatrix}
 T \cosh a^j + X^j \sinh a^j \\
 X^k + \delta^{jk} \left( T \sinh a^j + X^j \cosh a^j - X^j \right) 
\end{pmatrix}.
\end{align*}
 Note that $\Lambda(a)$ maps the past light cone $\vartriangle_0$ into itself.
 Applied to the solution $u_0(T,X) = - \sqrt{2} / T$ of \eqref{cwe-u}, the Lorentz transformation $\Lambda(a)$ generates a 3-parameter family
\begin{align}\label{ua}
 u_{a} (T,X) := u_0(\Lambda(a)(T,X)) = \frac{\sqrt{2}}{A_0(a)(-T) - A_j(a) X^j}, \quad a \in \RR^3,
\end{align}
 of explicit blowup solutions, given by
\begin{align*}
 A_0(a) &= \cosh a^1 \cosh a^2 \cosh a^3, \\
 A_1(a) &= \sinh a^1 \cosh a^2 \cosh a^3, \\
 A_2(a) &= \sinh a^2 \cosh a^3, \\
 A_3(a) &= \sinh a^3.
\end{align*}
 The Lorentz transformations of the fundamental selfsimilar solution of \eqref{cwe-psi} are
\begin{align}\label{psi_a}
 \psi_{a}(\xi) := \psi_{a}(\tau,\xi) = \frac{\sqrt{2}}{A_0(a)-A_j(a)\xi^j}.
\end{align}

 In the original coordinates $(t,x)$ the Lorentz boosts applied to $v_0(t) = \sqrt{2}/t$ yield
\begin{align}\label{va}
 v_a(t,x) = \frac{1}{t^2 - |x|^2} u_a \left( -\frac{t}{t^2-|x|^2}, \frac{x}{t^2-|x|^2} \right)
          = \frac{\sqrt{2}}{A_0(a) t - A_j(a) x^j}.
\end{align}

\begin{remark}[$H^1\times L^2$-norm of selfsimilar solution on $\Sigma_T$]\label{rem:norm}
 The conserved energy for fixed time $t$,
\[
 E(v(t,.),\partial_t v(t,.)) := \frac{1}{2} \|(v(t,.),\partial_t v(t,.))\|^2_{\dot H^1 \times L^2 (\RR^3)} - \frac{1}{4} \|v(t,.)\|^4_{L^4(\RR^3)},
\]
 is infinite for the selfsimilar solution $v_0(t,x) = \sqrt{2} /t $. The $H^1 \times L^2$-norm of $v_0$ and its normal derivative $\nabla_n v_0$ on a hyperboloidal slice $\Sigma_T$ (introduced in Definition~\ref{def:nabla_n}) grows like
\begin{align*}
 \| v_0 \|_{H^1(\Sigma_T)} + \| \nabla_n v_0 \|_{L^2 (\Sigma_T)} &= \frac{1}{|T|} \| v_0 \|_{L^2(\Sigma_T)} + \| \nabla_n v_0 \|_{L^2(\Sigma_T)} \\
&= \frac{1}{|T|} \sqrt{\frac{2}{|T|^2} \vol (B_{|T|})} + \sqrt{\frac{2}{|T|^4} \vol (B_{|T|})} \simeq |T|^{-\frac{1}{2}}.
\end{align*}
 The same growth rate holds for the Lorentz transformations $v_a$ of $v_0$ which is why we normalize the decay estimate of $v-v_a$ in Theorem~\ref{maintheorem} by the factor $|T|^\frac{1}{2}$.
\end{remark}

\begin{remark}[$L^4_t L^4_x$-Strichartz norm of selfsimilar solution]
The selfsimilar solution $v_0(t,x) = \sqrt{2} / t$ satisfies
\begin{align*}
  \| v_0 \|_{L^4 (t,2t) L^4 (B_{(1-\delta)t})}^4 &= \int_t^{2t} \int_0^{(1-\delta)s} \frac{4}{s^4} r^2 dr ds = \int_t^{2t} \frac{4}{s^4} \frac{(1-\delta)^3 s^3}{3} ds \\
&= \frac{(1-\delta)^3}{3} 4 \log 2 \simeq 1,
 \end{align*}
and therefore also $\| v_a \|_{L^4 (t,2t) L^4 (B_{(1-\delta)t})}^4 \simeq 1$.
\end{remark}


\section{Modulation ansatz}
\label{sec2}

 We now rewrite \eqref{cwe} as a first-order evolution problem of the form
\begin{align}\label{cwe-Psi}
 \partial_\tau \Psi (\tau) = \L \Psi (\tau) + \N (\Psi(\tau)).
\end{align}
 and then insert the modulation ansatz $\Psi(\tau) = \Psi_{a(\tau)} + \Phi(\tau)$ corresponding to the $3$-parameter family of solutions $\psi_a$, $a \in \RR^3$, to study its stability in similarity coordinates $(\tau,\xi)$ introduced in Section~\ref{ssec:sim}. This leads to an evolution equation for the residual term $\Phi$.

\subsection{Evolution problem}
 Let $\Psi = (\psi_1,\psi_2)$ be defined by
\begin{equation}\label{Psi}
\begin{aligned}
 \psi_1 &= \psi, \\
 \psi_2 &= \partial_0 \psi + \xi^j \partial_j \psi + \psi,
\end{aligned}
\end{equation}
 where $\psi_2$ is chosen in a way to compensate for the $T$-derivative $\partial_T u(T,X) = e^{2\tau} [1+\partial_\tau + \xi^j \partial_{\xi^j}] \psi(\tau,\xi)$ (cf.\ equation~\eqref{eq:dT} in the proof of Lemma~\ref{lem:transform2}).
 Then \eqref{cwe-psi} is equivalent to the system
\begin{equation}\label{eq-psi}
 \begin{aligned}
 \partial_0 \psi_1 &= - \xi^j \partial_j \psi_1 - \psi_1 + \psi_2, \\
 \partial_0 \psi_2 &= \partial_j \partial^j \psi_1 - \xi^j \partial_j \psi_2 -2 \psi_2 + \psi_1^3.
 \end{aligned}
\end{equation}
which admits the family of static solutions 
\begin{equation}\label{Psi_a}
\begin{aligned}
 \psi_{a,1}(\xi) &= \psi_{a}(\xi) = 
                    \frac{\sqrt{2}}{A_0(a)-A_j(a)\xi^j}, \\
 \psi_{a,2}(\xi) &= \frac{\sqrt{2}A_0(a)}{(A_0(a)-A_j(a)\xi^j)^2},
\end{aligned}
\end{equation}
 being obtained from $\psi_a$ in \eqref{psi_a}. Thus let $\Psi_{a} := (\psi_{a,1},\psi_{a,2})$ be the term corresponding to the solution $\psi_0(\xi) = \sqrt{2}$ and its Lorentz boosts $\psi_a$. From \eqref{eq-psi} we read off the formal linear differential operator
\begin{align}\label{Ltilde}
 \wL u (\xi) = \begin{pmatrix}
                   - \xi^j \partial_j u_1(\xi) - u_1(\xi) + u_2(\xi) \\
                   \partial_j \partial^j u_1(\xi) - \xi^j \partial_j u_2(\xi) -2 u_2 (\xi)
                  \end{pmatrix}
\end{align}
and nonlinear term
\begin{align}\label{N}
 \N (u)(\xi) = \begin{pmatrix}
                    0 \\
                    u_1^3(\xi)
                   \end{pmatrix}.
\end{align}

\subsection{The free evolution}

 For the operator $\wL$ we define a suitable domain in a Hilbert space $\H$ and show how to obtain the linear operator $\L$ from $\wL$. Since the system
\[
 \frac{d}{d\tau} \Psi(\tau) = \wL \Psi(\tau)
\]
 equals the system for the residual term corresponding to the free wave equation (4-4) in \cite{DZ}[p.\ 474] we proceed analogously. Let $\BB$ denote the open unit ball in $\RR^3$ and the Hilbert space $\H$ be the completion of $(H^1(\BB) \cap \cC^1(\BB)) \times (L^2(\BB) \cap \cC(\BB))$ with respect to the induced norm $\| \cdot \|$ of the inner product
\begin{align}\label{H-norm}
 \langle u , v \rangle := \int_\BB \partial_j u_1 (\xi) \overline{\partial^j v_1(\xi)} \, d\xi
                          + \int_{\partial \BB} u_1(\omega) \overline{v_1(\omega)} \, d\sigma(\omega)
                          + \int_\BB u_2(\xi) \overline{v_2(\xi)} \, d\xi.
\end{align}
 By \cite{DZ}[Lemma 3.1] the first two terms are equivalent to the standard $H^1$-norm, i.e.,
\[
 \| f \|_{H^1(\BB)} \simeq \| f \|_{\dot H^1(\BB)} + \| f \|_{L^2(\partial \BB)},
\]
 hence the space $\H$ is equivalent to $H^1(\BB) \times L^2 (\BB)$ as a Banach space. The domain of $\wL$ is the subspace $\D(\wL) := (H^2(\BB) \cap \cC^2 (\overline{\BB} \setminus \{ 0 \})) \times ( H^1(\BB) \cap \cC^1 (\overline{\BB} \setminus \{ 0 \} )$. $\D(\wL)$ is dense in $\H$ since $\cC^\infty(\overline{\BB})$ is dense in both. From the semigroup approach carried out in \cite{DZ}[Prop.\ 4.1] we obtain the following result.

\begin{proposition}\label{prop:L}
 The operator $\wL \colon \D(\wL) \subseteq \H \to \H$ is densely defined and closable. Its closure $\L$ generates a strongly continuous semigroup $\S \colon [0,\infty) \to \cB(\H)$ which satisfies
\[
  \| \S(\tau) \| \leq e^{-\frac{\tau}{2}}, \quad \tau \in [0,\infty).
\]
 The spectrum $\sigma(\L)$ of $\L$ is contained in the shifted half plane $\{ z \in \CC \, | \, \Re z \leq - \frac{1}{2} \}$. \qed
\end{proposition}

 This result implies a bound on the resolvent of $\L$ \cite{EN}[Theorem II.1.10].

\begin{corollary}\label{cor:resolvent}
 The resolvent $\R_\L(\lambda) = (\lambda - \L)^{-1}$ of $\L$ is a bounded operator for all $\lambda \in \rho(\L) = \CC \setminus \sigma(\L)$ that satisfies
\[
\pushQED{\qed} 
 \| \R_\L(\lambda) \| \leq \frac{1}{\Re \lambda + \frac{1}{2}}, \quad \Re \lambda > - \tfrac{1}{2}. \qedhere
\popQED
\]
\end{corollary}

\subsection{Modulation ansatz}\label{ssec:mod}

 It is the aim of this section to write the system \eqref{cwe-psi} in the abstract form
\begin{align}\label{evoeq}
 \partial_\tau \Psi (\tau) = \L \Psi (\tau) + \N (\Psi(\tau))
\end{align}
 for a function $\Psi \colon [0,\infty) \to \H$ and to study the stability of the 3-parameter family $\Psi_a$ derived in \eqref{Psi_a}. This involves the modulation ansatz
\begin{align}\label{Psi1}
 \Psi(\tau) = \Psi_{a(\tau)} + \Phi(\tau)
\end{align}
 for a function $\Phi \colon [0, \infty) \to \H$, where we allow for the rapidity $a \in \RR^3$ to depend on $\tau$. The Lorentz boosts $\Psi_a$ are static solutions of \eqref{evoeq}, thus we know that $\Psi_{a(\tau)}$ satisfy $\L \Psi_{a(\tau)} + \N (\Psi_{a(\tau)}) = 0$. Plugging the ansatz \eqref{Psi1} back into \eqref{evoeq} thus yields the equation
\begin{align}\label{eq1}
 \partial_\tau \Phi(\tau) - \L \Phi(\tau) - \L'_{a(\tau)} \Phi(\tau)  = \N_{a(\tau)}(\Phi(\tau)) - \partial_\tau \Psi_{a(\tau)},
\end{align}
 where $\L'_{a(\tau)}$ is the linearized part of the nonlinearity $\N(\Psi(\tau))$ at $\Psi_{a(\tau)}$,
\begin{align}\label{Laprime}
 \L'_{a(\tau)} \Phi(\tau) := \begin{pmatrix}
                             0 \\
                             3 \psi^2_{a(\tau),1} \phi_1(\tau)
                            \end{pmatrix},
\end{align}
 and $\N_{a(\tau)}$ is the remaining nonlinearity,
\begin{align}\label{Na}
 \N_{a(\tau)} (\Phi(\tau)) := \; & \N (\Psi_{a(\tau)} + \Phi(\tau)) - \N(\Psi_{a(\tau)}) - \L'_{a(\tau)} \Phi(\tau) \nonumber \\
= \; & \begin{pmatrix}
        0 \\ 3 \psi_{a(\tau),1} \phi_1(\tau)^2 + \phi_1(\tau)^3
       \end{pmatrix}.
\end{align}
 In the formulation \eqref{eq1} the operator $\L'_{a(\tau)}$ on the left hand side of the equation depends on $\tau$. To avoid this $\tau$-dependence, we assume that $a_\infty := \lim_{\tau \to \infty} a(\tau)$ exists and only include the limiting operator $\L'_{a_\infty}$ on the left hand side (the existence of this limit is verified later). The modulation ansatz therefore leads to the equation
\begin{align}\label{modeq}
 \partial_\tau \Phi(\tau) - \L \Phi(\tau) - \L'_{a_\infty} \Phi(\tau)  = [ \L'_{a(\tau)} - \L'_{a_\infty} ] \Phi(\tau) + \N_{a(\tau)}(\Phi(\tau)) - \partial_\tau \Psi_{a(\tau)}
\end{align}
 for the perturbation term $\Phi$ of the solutions $\Psi_a$. The advantage of this formulation is that the left hand side of \eqref{modeq} consists besides $\partial_\tau \Phi$ only of linear and $\tau$-independent operations on $\Phi$, whereas the right hand side is expected to be small for large $\tau$.


\section{The linearized evolution}
\label{sec3}

 The aim is to first understand the linearized part of equation \eqref{modeq}, i.e., the ordinary differential system
\[
 \partial_\tau \Phi(\tau) = [ \L + \L'_{a_\infty} ] \Phi(\tau).
\]
 Since $a_\infty$ is fixed, we simply consider a vector $a \in \RR^3$ and define the linear operator
\begin{align}\label{La}
 \L_a :=  \L + \L'_{a},
\end{align}
 which is explicitly given as the closure of
\[
 \wL_a \Phi (\tau) = \begin{pmatrix}
                     - \xi^j \partial_j \phi_{1} (\tau) - \phi_{1} (\tau) + \phi_{2} (\tau) \\
                     \partial_j \partial^j \phi_1 (\tau) - \xi^j \partial_j \phi_2 (\tau) - 2 \phi_{2}(\xi) + 3 \psi^2_{1,a} \phi_1(\tau)
                    \end{pmatrix}.
\]
 Semigroup theory implies that the perturbation $\L_a$ of $\L$ generates also a strongly continuous one-parameter semigroup because $\L'_{a}$ is a bounded (even compact) operator (cf.\ \cite{DZ}[Lemma 4.3] for the case $a=0$).

\subsection{Spectral analysis for $\L_0$}

 For $a=0$ the Lorentz boost $\Lambda(0)$ is the identity and the solution considered is $u_0 (T,X) = - \sqrt{2} / T$. The analysis for this ground case has been carried out in \cite{DZ} and we already know that the spectrum of $\L_0$ satisfies (see \cite{DZ}[Prop.\ 4.6])
\begin{align}\label{sp0}
 \sigma (\L_0) \subseteq \{ z \in \CC \, | \, \Re z \leq - \tfrac{1}{2} \} \cup \{0,1\}.
\end{align}
 with $\{0,1\} \subseteq \sigma_p (\L_0)$ and $3$- and $1$-dimensional eigenspaces for $0$ and $1$, respectively. The 3-dimensional eigenspace corresponding to the eigenvalue $0$ is spanned by the eigenvectors
\[
 \q_{0,j} (\xi) = \begin{pmatrix}
                   \xi^j \\ 2 \xi^j
                  \end{pmatrix}, \quad j \in \{1,2,3\},
\]
the eigenspace corresponding to $1$ by the eigenvector
\[
 \p_{0} (\xi) = \begin{pmatrix}
               1 \\ 2
              \end{pmatrix}.
\]
 The analysis of the corresponding spectral projections shows that the algebraic multiplicities of the eigenvalues $0$ and $1$ are equal to their respective geometric multiplicities \cite{DZ}[Lemma 4.8].

\subsection{Spectral analysis for $\L_a$}
\label{ssec3.2}

Assuming now that $a \neq 0$ is small, we will observe that the spectrum of $\L_a$ is not so different from $\sigma(\L_0)$. More precisely, we work towards proving the following result.

\begin{proposition}\label{La-spectrum}
 Fix $\eps \in ( 0, \frac{1}{2})$. For $a \in \RR^3$ sufficiently small the spectrum of $\L_a$ satisfies
\[
 \sigma(\L_a) \subseteq \{ z \in \CC \, | \, \Re z < \eetilde \} \cup \{ 0,1 \}
\]
with $\{ 0,1 \} \subseteq \sigma_p (\L_a)$.
\end{proposition}

Let us first note some properties about the operator $\L'_a$.

\begin{lemma}
 For $a \in \RR^3$ sufficiently small the operator $\L'_a$ is compact and $\L_a = \L + \L'_a$ is a closed operator which generates a strongly-continuous one-parameter semigroup $\S_a \colon [0, \infty) \to \cB(\H)$.
\end{lemma}

 \begin{proof}
 The denominator of $\psi_{a,1}$ \eqref{Psi_a} consists of hyperbolic functions and is equal to $1$ for $a = 0$ and any $\xi \in \RR^3$. Hence $\psi_{a,1}^2 \in \cC^\infty(\overline{\BB})$ for $a$ sufficiently small. The first component of $\L'_a$ is zero, and the second component is a map
\[
 H^1(\BB) \to L^2(\BB)
\]
 and therefore compact due to the compact embedding $H^1 \subset\subset L^2$ (Rellich--Kondrakov Theorem).
 Since $\L$ is a closed operator, stability implies that $\L_a$ is also closed \cite{K}[Theorem IV.1.3].
 The bounded perturbation theorem of semigroup theory \cite{EN}[Theorem III.1.3] yields the existence of the semigroup $\S_a$.
\end{proof}

 To conclude that the spectrum of $\L_a$ can only change slightly as compared to that of $\L_0$ we analyze how near these operators are for small $a$. To this end we utilize a notion of nearness on the space of closed operators based on their graphs:

 Suppose $M,N$ are closed subspaces of a Banach space and $\SS_M,\SS_N$ are the unit spheres in $M$ and $N$, respectively. If both $M$ and $N$ are nonzero (or both are zero) we define
\[
 \delta_\mathrm{mf}(M,N) := \sup_{u \in \SS_M} \operatorname{dist} (u, \SS_N),
\]
 otherwise
\[
 \delta_\mathrm{mf}(0,N) := 0 \quad \text{for any } N, \quad \delta_\mathrm{mf}(M,0) := 1 \quad \text{for } M \neq 0.
\]
 The \emph{gap} $\hat\delta$ between the manifolds $M$ and $N$ is the symmetrization of $\delta$, i.e.,
\[
 \hat\delta_\mathrm{mf}(M,N) := \max \left( \delta_\mathrm{mf} (M,N), \delta_\mathrm{mf}(N,M) \right).
\]
 For any two closed operators $\T$ and $\S$ from the Banach spaces $X$ to $Y$ the \emph{operator gap} $\hat \delta$ between $\T$ and $\S$ is now defined as
\[
 \hat \delta (\T,\S) := \hat \delta_\mathrm{mf} ( \graph(\T), \graph(\S) ),
\]
 where $\hat \delta_\mathrm{mf}$ measures the gap between the closed linear manifolds $\graph(\T), \graph(\S)$ of the product space $X \times Y$. See \cite{K}[Sec.\ IV.2.1 and IV.2.4] for the details on the construction of $\hat \delta_\mathrm{mf}$ and $\hat \delta$.

\begin{lemma}\label{Laconv1}
 The closed operators $\L_a$ and $\L_b$ satisfy
\begin{align*}
 \hat \delta (\L_a,\L_b) &\leq \| \L'_a - \L'_b \|,
\end{align*}
 globally and for $a,b$ sufficiently small moreover
\[
 \| \L_a - \L_b \| = \| \L'_a - \L'_b \| \lesssim |a-b|.
\]
 In particular, $\L_a$ converges to $\L_0$ in the generalized sense of \cite{K}[Sec.\ IV.2.4].
\end{lemma}

\begin{proof}
 The difference $\L'_a - \L'_b$ of the two compact operators is compact, in particular bounded. Therefore it follows immediately by \cite{K}[Theorem IV.2.14] that
\[
 \hat \delta (\L_a,\L_b) = \hat \delta (\L_a,\L_a + (\L'_b - \L'_a)) \leq \| \L'_a - \L'_b \|.
\]
 The second inequality is due to the specific form of $\L'_a$ involving $a$ only in the square of the corresponding Lorentz boosted solution $\psi_a$ \eqref{psi_a}, which basically involves hyperbolic functions. Hence the operator $\L'_a$ is locally Lipschitz continuous in $a$ with respect to the operator norm.
\end{proof}

 Since the perturbation $\L'_a-\L'_0$ does not commute with the operator $\L_0$ we cannot expect continuous dependence of the spectrum, however, upper semicontinuity still holds in a weak sense~\cite{K}[Theorem IV.3.1]. In other words, the spectrum of $\L_a$ cannot extend to suddenly compared to that of $\L_0$, but it may well shrink suddenly. Moreover, it is true that each separated part of the spectrum $\sigma(\L_0)$ is upper semicontinuous \cite{K}[Theorem IV.3.16] and each finite system of eigenvalues changes continuously \cite{K}[Sec.\ IV.3.5] just as in the finite-dimensional case. It therefore remains to control the resolvent set on the unbounded domain away from the shifted half plane and the eigenvalues.

 Fix $\eps \in ( 0, \frac{1}{2})$ and $\lambda_0>1$. To simplify the notation in the following proofs we introduce the unbounded region
\begin{align} \Omega^\eps := \{ z \in \CC \, | \, \Re z \geq \eetilde \} \end{align}
 and the bounded region
\begin{align}\label{om:lambda0}
 \Omega^\eps_{\lambda_0} := \{ \lambda \in \CC \, | \, \Re \lambda \geq \eetilde \text{ and } |\lambda| \leq \lambda_0 \},
\end{align}
 which are sketched in Figure~\ref{fig3}.

\begin{figure}[h]
\centering
\def\svgwidth{.9\columnwidth}
\hspace*{3.5cm} 
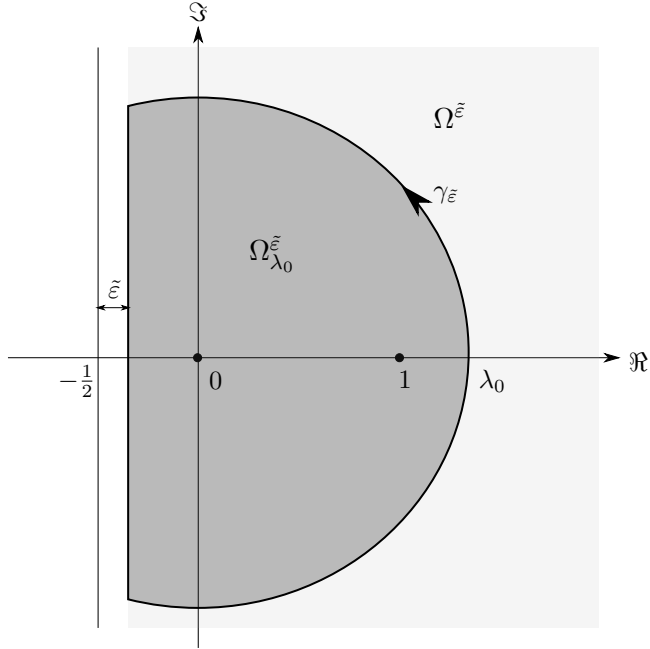
\caption{The domains $\Omega^\eps$ and $\Omega^\eps_{\lambda_0}$ for fixed $\eps \in (0,\frac{1}{2})$ and $\lambda_0>1$. $\Omega^\eps$ is an unbounded region away from the boundary $-\frac{1}{2}$ of the essential spectrum. The compact region $\Omega^\eps_{\lambda_0}$ includes the eigenvalues $0$ and $1$.}
\label{fig3}
\end{figure}

\begin{lemma} \label{Loresolvent}
 Fix $\eps \in (0,\frac{1}{2})$. There exist constants $c>0, \lambda_0 >1$ such that the resolvent $\R_{\L_0}(\lambda)$ is uniformly bounded by
\[
 \| \R_{\L_0}(\lambda) \| \leq c
\]
 for all $\lambda \in \Omega^\eps \setminus \Omega^\eps_{\lambda_0}$.
\end{lemma}

\begin{proof}
 Consider the shifted half plane $\Omega^\eps \subseteq \rho (\L) \cap \rho (\L_0)$. Corollary~\ref{cor:resolvent} implies that the free resolvent $\R_\L(\lambda)$ is a bounded operator which satisfies
\begin{align}\label{eq2}
 \| \R_\L(\lambda) \| \leq \frac{1}{\Re \lambda + \frac{1}{2}} \leq \frac{1}{\eps},
\end{align}
 \emph{uniformly\/} for $\lambda \in \Omega^\eps$. The resolvent $\R_{\L_0}(\lambda)$ can be related to the free resolvent $\R_\L(\lambda)$ by the second resolvent identity
\[
 \lambda-\L_0 = [1-\L_0'\R_\L(\lambda)](\lambda-\L).
\]
 Therefore, if 
\begin{align}\label{eq3}
 \| \L_0' \R_\L (\lambda) \| < 1,
\end{align}
 the Neumann series $(1 - \L_0' \R_\L(\lambda))^{-1} = \sum_{k =0}^\infty (\L_0' \R_\L(\lambda))^k$ converges and the claim follows. We will show that \eqref{eq3} holds true for all $|\lambda|$ sufficiently large. For each $f \in \H$ the function $u_f := \R_\L(\lambda) f \in \D(\L)$ solves $(\lambda - \L) u_f = f$ by definition. The first component of this equation is
\begin{align}\label{stern}
 \lambda u_{f,1}(\xi) + \xi^j \partial_j u_{f,1}(\xi) + u_{f,1} (\xi) - u_{f,2}(\xi) = f_1(\xi).
\end{align}
 Therefore, for $\xi \in \BB$,
\begin{align*}
 |\lambda| \| u_{f,1} \|_{L^2 (\BB)} &\leq |\xi^j| \| \partial_j u_{f,1} \|_{L^2(\BB)} + \| u_{f,1} \|_{L^2(\BB)} + \| u_{f,2} \|_{L^2(\BB)} + \| f_1 \|_{L^2(\BB)} \\
&\lesssim \| u_{f,1} \|_{H^1(\BB)} + \| u_{f,2} \|_{L^2(\BB)} + \| f_1 \|_{L^2(\BB)} \\
&\lesssim \| u \|_{H^1 (\BB) \times L^2 (\BB)} + \|f\|_{H^1 (\BB) \times L^2 (\BB)},
\end{align*}
 and by inserting the definition of $u_f$ and using the fact that $\H$ is equivalent to $H^1(\BB) \times L^2(\BB)$ as a Banach space \cite{DZ}[Lemma 3.1], we obtain for $\lambda \in \Omega^\eps$ with \eqref{eq2},
\[
 \| ( \R_\L (\lambda) f )_1 \|_{L^2(\BB)} = \| u_{f,1} \|_{L^2(\BB)} \lesssim |\lambda|^{-1} \left( \| \R_\L (\lambda) f \| + \| f \| \right) \lesssim |\lambda|^{-1} \|f\|.
\]
 By the definition of $\L_0'$ \eqref{Laprime} and the fact that $\psi_0 \equiv \sqrt{2}$ is constant,
\[
 \| \L_0' \R_\L (\lambda) f \| = \| 3 \psi^2_{0,1} ( \R_\L (\lambda) f )_1 \|_{L^2(\BB)} \leq 6 
\| ( \R_\L (\lambda) f )_1 \|_{L^2(\BB)} \lesssim |\lambda|^{-1} \| f \|,
\]
 which implies \eqref{eq3} for $|\lambda|$ sufficiently large.
\end{proof}

\begin{corollary}\label{Laresolvent}
 Fix $\eps \in (0,\frac{1}{2})$. For $a$ sufficiently small we have
\[
 \Omega^\eps \setminus \Omega^\eps_{\lambda_0} \subseteq \rho(\L_a) .
\]
\end{corollary}

\begin{proof}
 Let $\lambda \in \Omega^\eps \setminus \Omega^\eps_{\lambda_0} \subseteq \rho(\L_0)$ and consider the resolvent $\R_{\L_0} (\lambda) = (\lambda - \L_0)^{-1}$. Just as in the proof of Lemma~\ref{Laconv1} it can be shown that
\[
 \hat \delta (\lambda - \L_a,\lambda - \L_0) \leq \| \L'_a - \L'_0 \| \lesssim | a |, \qquad \lambda \in \Omega^\eps \setminus \Omega^\eps_{\lambda_0},
\]
 holds for $a$ sufficiently small. By Lemma~\ref{Loresolvent}, $\R_{\L_0}$ is uniformly bounded by $c$ on $\Omega^\eps \setminus \Omega^\eps_{\lambda_0}$. Therefore, for $a$ sufficiently small, 
\[
 \hat\delta (\lambda - \L_a, \lambda - \L_0) < \frac{1}{\sqrt{1+c^2}} \leq (1 + \| \R_{\L_0} (\lambda) \|^2)^{-\tfrac{1}{2}}, \qquad \lambda \in \Omega^\eps \setminus \Omega^\eps_{\lambda_0},
\]
 and the generalized stability theorem for bounded invertibility \cite{K}[Theorem IV.2.21] implies that $\lambda - \L_a$ is invertible and the inverse $\R_{\L_a}(\lambda)$ is bounded as well.
\end{proof}

 Lemma~\ref{Laconv1} and Lemma~\ref{Loresolvent} imply the existence of a uniform bound for all resolvents $\R_{\L_a}(\lambda)$ due to \cite{K}[Theorem IV.3.15].

\begin{corollary} \label{Laresolvent_bound}
 Fix $\eps \in (0,\frac{1}{2})$. There exist constants $c>0, \lambda_0>1$ such that for all $a \in \RR^3$ sufficiently small the resolvents $\R_{\L_a}(\lambda)$ are uniformly bounded by
\[
\pushQED{\qed}
 \| \R_{\L_a} (\lambda) \| \leq c, \qquad \lambda \in \Omega^\eps \setminus \Omega^\eps_{\lambda_0}. \qedhere
\popQED
\]
\end{corollary}

 Within the region $\Omega_{\lambda_0}^\eps$ it remains to control the eigenvalues $0$ and $1$. The eigenspaces of $\L_a$ can be computed explicitly by taking into account the time translation invariance of the cubic wave equation as well as its linearization with respect to $a$. Since $\Psi_a$ is a solution of the equivalent evolution equation $\partial_\tau \Psi(\tau) = \L \Psi(\tau) + \N(\Psi(\tau))$ with trivial left hand side, $\partial_a \Psi_a$ is a solution of the linearized equation $0 = \L \partial_a \Psi_a + D \N (\Psi_a) \partial_a \Psi_a = \L_a (\partial_a \Psi_a)$ around $\Psi_a$ by the chain rule. This is how we obtain the eigenfunctions associated to the eigenvalue $0$. On the other hand, the time translated solution $u_a^\theta (T,X) = u_a (T-\theta,X)$ is also a solution to \eqref{cwe-u}. Therefore, differentiation with respect to $\theta$ of $\psi_a^\theta(\tau,\xi) = \frac{\sqrt{2} e^{-\tau}}{A_0(a)(\theta+e^{-\tau} - A_j(a)e^{-\tau}\xi^j}$ (at $0$) in the coordinates $(\tau,\xi)$ yields 
the eigenfunction for the eigenvalue $1$. We obtain the following formulas.

\begin{lemma}\label{LaEV}
 The linear operator $\L_a$ has eigenvalues $0$ and $1$ with eigenspaces spanned by the functions
\[
 \q_{a,j}(\xi) = \partial_{a^j} \Psi_a(\xi), \quad j \in \{1,2,3\},
\]
 and
\[
 \p_a(\xi) = \frac{A_0(a)}{(A_0(a) - A_j(a)\xi^j)^2}
            \begin{pmatrix}
             1 \\ \frac{2A_0(a)}{A_0(a) - A_j(a)\xi^j}
            \end{pmatrix},
\]
 respectively. The corresponding geometric and algebraic multiplicities coincide and are $3$ and $1$.
\end{lemma}

\begin{proof}
 It can be easily verified that $\L_a \q_{a,j} = 0$, $j \in \{1,2,3\}$, and $\L_a \p_a (\xi) = \p_a(\xi)$ hold. For the latter equation we make use of the identity $A_\mu(a) A^\mu(a) = - 1$ of hyperbolic functions. 

 The total multiplicity of the eigenvalues of $\L_0$ are $4$, and by \cite{K}[Sec.\ IV.3.5] this cannot change for $\L_a$ with $a$ small. The geometric multiplicities of the eigenvalues $0$ and $1$ are $3$ and $1$, respectively. Therefore no further eigenvalues exist.
\end{proof}

\begin{proof}[Proof of Proposition~\ref{La-spectrum}]
 It suffices to summarize the results obtained in this section and apply general perturbation theory for linear operators. Note first that the spectrum of $\L_0$ can be separated into two parts by the simple closed curve $\gamma_\eps$ surrounding the region $\Omega^\eps_{\lambda_0}$, see Figure~\ref{fig3}.
 By Lemma~\ref{Laconv1}, $\hat \delta (\L_0,\L_a) \lesssim |a|$ and therefore \cite{K}[Theorem IV.3.16] implies that the spectrum of $\L_a$ is likewise separated into two parts by the curve $\gamma_\eps$ for all $a$ sufficiently small. Within the region $\Omega^\eps_{\lambda_0}$ the eigenvalues of $\L_a$ are just those derived in Lemma~\ref{LaEV} (this also makes use of the continuity of a finite spectrum of eigenvalues \cite {K}[Sec.\ IV.3.5]). In the unbounded region $ \Omega^\eps \setminus \Omega^\eps_{\lambda_0}$ the operator $\lambda - \L_a$ is invertible by Corollary~\ref{Laresolvent}. Thus
\[
 \sigma(\L_a) \subseteq \{ z \in \CC \, | \, \Re z < \eetilde \} \cup \{ 0,1 \}. \qedhere
\]
\end{proof}

\begin{remark}
 The key result in this section is Lemma~\ref{Loresolvent}, which shows that the resolvent of $\L_0$ is bounded \emph{uniformly} on the unbounded region $\Omega^\eps \setminus \Omega^\eps_{\lambda_0}$ away from the spectrum. Such a result is not true for general linear operators but holds here because $\lambda$ appears as a prefactor in the left hand side of \eqref{stern}. Consequently, this is a structural property of the equation, independent of $\L'_a$.
\end{remark}

\subsection{Growth estimates for the linearized evolution}

 We establish the growth bounds for the semigroup $\S_a$ by partitioning $\H$ into disjoint parts, each related to a separated part of the spectrum of $\L_a$. For $a$ sufficiently small we find that the semigroup $\S_a$ has a 4-dimensional unstable subspace which is spanned by the generators of the time translation and three Lorentz boosts. On the remaining infinite-dimensional subspace the semigroup decays exponentially.

\medskip
 To separate the eigenvalues from the remaining unbounded part of the spectrum we define spectral projections $\Q_{a,j}$, $j \in \{1,2,3\}$, and  $\P_a$ to the eigenvalues $0$ and $1$, respectively. The rectifiable, simple closed curves $\gamma_0$ and $\gamma_1$, defined by
\[
 \gamma_0 (t) = \tfrac{1-2\eps}{3} e^{2\pi \ii t}, \quad \gamma_1 (t) = 1 + \tfrac{1}{2} e^{2\pi \ii t}, \quad t \in [0,1],
\]
 enclose the eigenvalues $0$ and $1$ and do not intersect each other, see Figure~\ref{fig4}. The corresponding spectral projections of $\L_a$ are defined by
\[
 \Q_{a} := \frac{1}{2\pi \ii} \int_{\gamma_0} \R_{\L_a} (z) dz, \quad \P_a := \frac{1}{2\pi \ii} \int_{\gamma_1} \R_{\L_a} (z) dz
\]

\begin{figure}[h]
\centering
\def\svgwidth{.8\columnwidth}
\hspace*{1cm} 
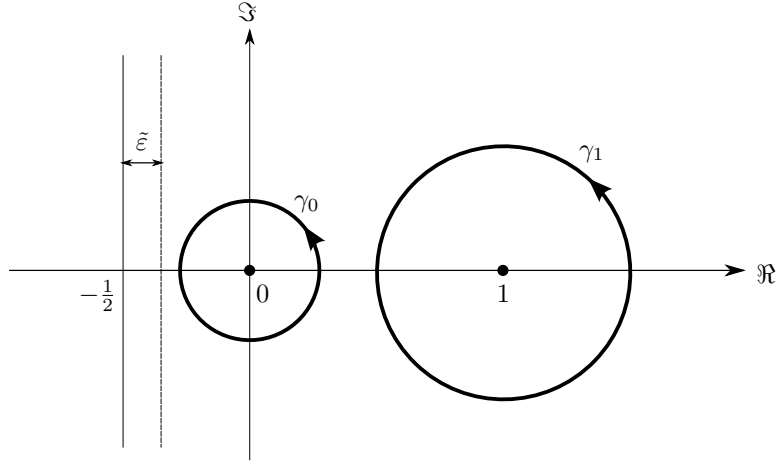
\caption{The spectral projections $\Q_a$ and $\P_a$ correspond to the eigenvalues $0$ and $1$, respectively.}
\label{fig4}
\end{figure}

 By Lemma~\ref{LaEV} we have that
\[
 \rg \Q_a = \langle \q_{a,1},\q_{a,2},\q_{a,3} \rangle, \quad \rg \P_a = \langle \p_a \rangle.
\]
 By $\Q_{a,j}$ we denote the projection onto the subspace generated by $\q_{a,j}$. That is, if $f \in \H$ and $\Q_a f = \sum_{k=1}^3 \alpha_k p_{a,k}$ is the unique representation of $f$ in the basis $\{ \q_{a,1},\q_{a,2},\q_{a,3} \}$, then $\Q_{a,j} f := \alpha_j \q_{a,j}$. In particular, $\Q_a = \sum_{j=1}^3 \Q_{a,j}$ and
\[
 \rg \Q_{a,j} = \langle \q_{a,j} \rangle.
\]

 The operator $\L_a$ restricted to the range of the projections yields bounded operators on the finite-dimensional Hilbert spaces $\rg \Q_a$ and $\rg \P_a$ with spectra equal to $\{0\}$ and $\{1\}$, respectively. The remaining projection
\[
 \wP_a := \I - \Q_a - \P_a = \I - \sum_{j=1}^3 \Q_{a,j} - \P_a,
\]
 defines the infinite-dimensional subspace $\rg \wP_a$ of $\H$ and the spectrum of $\L_a |_{\rg \wP_a \cap \D(\L_a)}$ is contained in the shifted half plane $\{ z \in \CC \, | \, \Re z < \eetilde \}$ (see \cite{K}[Theorem III.6.17]).

\begin{proposition}\label{prop:proj}
 Fix $\varepsilon \in ( 0, \frac{1}{2})$. The rank-one projections $\Q_{a,j}$, $j \in \{ 1,2,3 \}$, and $\P_a$ are bounded operators on $\H$, satisfy the transversality conditions
\[
 \Q_{a,j} \P_a = \P_a \Q_{a,j} = 0 \quad \text{ and } \quad \Q_{a,j} \Q_{a,k} = \delta_{jk} \Q_{a,j},
\]
as well as
\[
 \wP_a \Q_a = \Q_a \wP_a = \wP_a \P_a = \P_a \wP_a = 0,
\]
and commute with the semigroup $\S_a$ of $\L_a$,
\[
 [ \S_a (\tau), \Q_{a,j} ] = [ \S_a (\tau), \P_a ] = 0, \quad \tau \geq 0.
\]
Moreover, for $\tau \geq 0$,
\begin{align*}
 \S_a(\tau) \Q_{a,j} &= \Q_{a,j}, \quad j \in {1,2,3},\\
 \S_a(\tau) \P_a &= e^\tau \P_a, \\
 \| \S_a(\tau) \wP_a f \| &\lesssim e^{-\beep \tau} \| \wP_a f \|, \quad f \in \H.
\end{align*}
\end{proposition}

\begin{proof}
 Let $\eps \in (0,\varepsilon)$, and choose $\lambda_0>1$ and $a_0>0$ accordingly so that the results in Section~\ref{ssec3.2} hold for all $a \in \RR^3$ with $|a|<a_0$. Holomorphic functional calculus yields \[ \Q_a \P_a = \P_a \Q_a = 0.\] Hence the mutual transversality conditions follow directly from the definition of the projections $\Q_{a,j}$, and the transversality with $\wP_a$ from its definition. The semigroup $\S_a$ commutes with the resolvents of $\L_a$ (this follows from the integral representation of the resolvent, see \cite{EN}[Theorem II.1.10]), and therefore with the spectral projections, i.e.,
\[
 [ \S_a(\tau), \Q_a ] = [ \S_a(\tau), \P_a ] = 0.
\]
 By definition of $\Q_{a,j}$ and the fact that $\S_a(\tau) = e^{0 \cdot \tau} = \I$ on the range of $\Q_a$ (since $\L_a \Q_a = 0$) also
\[
 \Q_{a,j} \S_a(\tau) f = \Q_{a,j} \Q_a \S_a(\tau) f = \Q_{a,j} \S_a(\tau) \Q_a f = \Q_{a,j} \Q_a f = \Q_{a,j} f = \S_a(\tau) \Q_{a,j} f,
\]
 which is $[\S_a(\tau),\Q_{a,j}] =0$.

 It remains to establish a growth bound on the infinite-dimensional subspace $\rg \wP_a$. Since $\R_{\L_a}(z)$ is holomorphic away from the eigenvalues $\{0,1\}$, it is clear that $\R_{\L_a}(\lambda)\wP_a$ is bounded on the compact set $\Omega^\eps_{\lambda_0}$ for each $a$ \cite{K}[Theorem III.6.17]. Uniformity with respect to small $a$ follows once again from perturbation theory \cite{K}[Theorem IV.3.15].

 On the other hand, by Corollary~\ref{Laresolvent_bound} there exists $\lambda_0>0$ such that the resolvent $\R_{\L_a}(\lambda)$ of $\L_a$ is uniformly bounded by some $c>0$ for $\lambda \in \Omega^\eps \setminus \Omega^\eps_{\lambda_0} =  \{ z \in \CC \, | \, \Re z \geq \eetilde \text{ and } |z| > \lambda_0 \}$ and all $a$ sufficiently small.

 Together we therefore obtain that $\R_{\L_a}(\lambda)$ is bounded uniformly on the half space $\Omega^\eps = \{ z \in \CC \, | \, \Re z \geq \eetilde \}$, for all $a$ sufficiently small, i.e.,
\[
 \sup_{\lambda \in \Omega^\varepsilon} \| \R_{\L_a}(\lambda) \wP_a \| < \infty,
\]
 The Gearhart--Greiner--Pr\"uss  Theorem (see, for example, \cite{EN}[Theorem V.3.8]) thus implies that the strongly continuous semigroup $(\S_a(\tau)\wP_a)_{\tau \geq 0}$ is uniformly exponentially stable with growth bound $\eetilde$. In particular, for the given $\varepsilon>\eps$ there exists a constant $C_\varepsilon>0$ such that
\[
 \| \S_a(\tau) \wP_a f \| \leq C_\varepsilon e^{-\beep \tau} \| \wP_a f \|
\]
 for all $f \in \H$.
\end{proof}


\section{The nonlinear equation}
\label{sec4}

 In Section~\ref{sec3} we have established growth bounds for the linearized equation using semigroup theory. From now on we include the remaining nonlinear terms and hence consider the full nonlinear problem \eqref{eq-psi}, i.e.,
\begin{equation*}
 \begin{aligned}
 \partial_0 \psi_1 &= - \xi^j \partial_j \psi_1 - \psi_1 + \psi_2, \\
 \partial_0 \psi_2 &= \partial_j \partial^j \psi_1 - \xi^j \partial_j \psi_2 -2 \psi_2 + \psi_1^3,
 \end{aligned}
\end{equation*}
 with data prescribed at $\tau = 0$. Using a modulation ansatz, this formal Cauchy problem,
\begin{align*}
 \partial_\tau \Psi(\tau) &= \L \Psi(\tau) + \N (\Psi(\tau)), \\
 \Psi(0) &= \Psi_0 + u, 
\end{align*}
 has led to equation \eqref{modeq} for the perturbation $\Phi(\tau) = \Psi(\tau) - \Psi_{a(\tau)}$, i.e.,
\begin{align}\label{modeq2}
 \partial_\tau \Phi(\tau) - \L \Phi(\tau) - \L'_{a_\infty} \Phi(\tau)  = [ \L'_{a(\tau)} - \L'_{a_\infty} ] \Phi(\tau) + \N_{a(\tau)}(\Phi(\tau)) - \partial_\tau \Psi_{a(\tau)},
\end{align}
 where $\Psi_{a(\tau)}$ are the Lorentz boosts of the selfsimilar solution $\psi_0 \equiv \sqrt{2}$. Up to here we have only treated the linear part $\partial_\tau \Phi(\tau) - \L \Phi(\tau) - \L_{a_\infty}' \Phi(\tau)$ on the left hand side of this modulation equation \eqref{modeq2}. The inhomogeneous right hand side can be included by rewriting \eqref{modeq2} as an integral equation using Duhamel's principle. That is, for given initial data $\Phi(0) = u \in \H$ and the linear operator $\hL_a := \L_a'-\L_{a_\infty}'$, a weak solution of \eqref{modeq2} must satisfy
\begin{equation}\label{intmodeq}
\begin{aligned}
 \Phi(\tau) = \; & \S_{a_\infty}(\tau) u
 + \int_0^\tau \S_{a_\infty}(\tau-\sigma) \left[ \hL_{a(\sigma)} \Phi(\sigma) + \N_{a(\sigma)}(\Phi(\sigma)) - \partial_\sigma \Psi_{a(\sigma)} \right] d\sigma \\ =: \; & \K_u(\Phi,a)(\tau).
\end{aligned}
\end{equation}
 The difficulty when solving equation~\eqref{intmodeq} arises from the unstable subspaces $\rg \Q_{a_\infty}$ and $\rg \P_{a_\infty}$ of the linearized evolution described by the semigroup $\S_{a_\infty}(\tau)$ on $\H$. These instabilities are induced by Lorentz symmetry and time translation, respectively, and are encoded in the spectral projections $\Q_{a_\infty}$ and $\P_{a_\infty}$. While we are able to suppress the instability coming from the Lorentz symmetry by choosing the rapidity $a$ in a suitable way using the projections $\Q_{a_\infty,j}$, $j \in \{1,2,3\}$ (see Section~\ref{ssec4.2}), we treat the time-translation instability by adding a correction term $\C_u(\Phi,a)$ that filters out the influence of $\P_{a_\infty}$ (see Section~\ref{ssec4.3}). This approach leads us to a modified equation of the form
\begin{align}\label{modified}
 \Phi(\tau) = \wK_u(\Phi,a)(\tau),
\end{align}
 where $\wK_u(\Phi,a)(\tau) = \K_u(\Phi,a)(\tau) - \S_{a_\infty}(\tau)\C_u(\Phi,a)(\tau)$. The idea is to use a fixed point argument that relies on Lipschitz continuity of the linear theory, more precisely Lipschitz bounds for the semigroup $\S_a(\tau)$ and the projections $\Q_a,\P_a,\wP_a$, as well as the Lipschitz continuity of the nonlinearity (see Section~\ref{ssec4.1}). The correction term related to the modified equation~\eqref{modified} is finally used \emph{a posteriori\/} to define the correct set of initial data. More precisely, by projecting away from the time-translation instability we obtain in Section~\ref{ssec4.4} a codimension-1 manifold of initial data evolving into the global solutions of the main Theorem~\ref{maintheorem}. Finally, in Sections~\ref{ssec4.5} and \ref{ssec4.6} the asymptotic behavior of global solutions with codimension-1 and small initial data is derived, respectively. This proves the main results of this article, Theorem~\ref{smalldatatheorem} and Theorem~\ref{maintheorem}.

\subsection{Lipschitz estimates for the linear and nonlinear operators}
\label{ssec4.1}

It is the aim of this section to prove that the inhomogeneous terms in the integrand are small. The estimates below will later be used vastly in the contraction arguments of Section~\ref{ssec4.2} and \ref{ssec4.3}.

\begin{lemma}[Lipschitz estimates for the linear evolution]\label{lem:decay1}
 Fix $\varepsilon \in ( 0, \frac{1}{2})$ . The eigenvectors, spectral projections and semigroup satisfy Lipschitz bounds with respect to $a,b \in \RR^3$ small, i.e.,
\begin{align*}
 \| \q_{a,j} - \q_{b,j} \| + \| \p_a - \p_b \| &\lesssim |a-b|, \quad j \in \{1,2,3\}, \\
 \| \Q_a - \Q_b \| + \| \P_a - \P_b \| &\lesssim |a-b|, \\
 \| \S_a(\tau) \wP_a - \S_b(\tau) \wP_b \| &\lesssim |a-b| e^{-\beep \tau} , \quad \tau \geq 0.
\end{align*}
\end{lemma}

\begin{proof}
 Since the eigenvectors are made of hyperbolic functions (see Lemma~\ref{LaEV}), they are (locally) Lipschitz continuous for $a,b$ sufficiently small as long as the denominator is nonzero.

 The second estimate for the projections follows from the second resolvent identity, and the fact that the resolvents are uniformly (for small $a$) bounded on compact sets (this uses \cite{K}[Theorem IV.3.15] and Lemma~\ref{Laconv1}), i.e., for $\lambda \in \rg \gamma_0 \cup \rg \gamma_1 \subseteq \rho(\L_a) \cap \rho(\L_b)$,
\begin{align*}
 \| \R_{\L_a}(\lambda) - \R_{\L_b}(\lambda) \| &\leq \| \R_{\L_a}(\lambda) \| \| \R_{\L_b}(\lambda) \| \| \L_a - \L_b \| \\
& \lesssim \| \R_{\L_a}(\lambda) \| \| \R_{\L_b} (\lambda) \| | a-b| \lesssim |a-b|.
\end{align*}

 It remains to be proven that the third estimate for the projection $\wP_a$ holds, for which we fix $\eps \in (0,\varepsilon)$. Suppose that $u \in \cC^\infty(\overline \BB)^2$. By the proof of \cite{DS}[Lemma 4.9] the function
\[
 \Phi_{a,b}(\tau) := \frac{\S_a(\tau) \wP_a u - \S_b(\tau) \wP_b u}{|a-b|}
\]
satisfies the inhomogeneous equation
\begin{align}\label{Phiab}
 \partial_\tau \Phi_{a,b}(\tau) = \L_a \wP_a \Phi_{a,b}(\tau) + \frac{\L_a\wP_a - \L_b\wP_b}{|a-b|} \S_b(\tau) \wP_b u
\end{align}
 with initial value $\Phi_{a,b}(0) = \frac{\wP_a - \wP_b}{|a-b|} u$. Proposition~\ref{prop:proj} showed that $\L_a \Q_a = 0$ and $\L_a\P_a = \P_a$, thus by Lemma~\ref{Laconv1} and the second step of this proof, for $a,b$ small,
\[
 \| \L_a \wP_a - \L_b \wP_b \| = \| \L_a - \P_a -\L_b + \P_b \| = \| \L_a' - \L_b' - \P_a + \P_b \| \lesssim |a-b|.
\]
 Hence $\frac{\L_a\wP_a - \L_b\wP_b}{|a-b|}$ is in fact a bounded operator. Similarly, $\wP_a - \wP_b$ is bounded since
\[
 \| \wP_a - \wP_b \| \leq \| \Q_a - \Q_b \| + \| \P_a - \P_b \| \lesssim | a - b |.
\]
 The Duhamel formula applied to \eqref{Phiab} together with the growth estimate for $\S_a(\tau)\wP_a$ from Proposition~\ref{prop:proj} (applied to $\eps$ instead of $\varepsilon$) thus yields
\begin{align*}
 \| \Phi_{a,b}(\tau) \| &\leq \| \S_a(\tau) \wP_a \Phi_{a,b}(0) \| + \int_0^\tau \left\| \S_a(\tau-\sigma) \wP_a \frac{\L_a\wP_a - \L_b\wP_b}{|a-b|} \S_b(\sigma) \wP_b u \right\| d\sigma \\
&\lesssim e^{-\beeptilde \tau} \left\| \frac{\wP_a - \wP_b}{|a-b|} u \right\| + \int_0^\tau e^{-\beeptilde (\tau - \sigma)} \left\| \frac{\L_a\wP_a - \L_b\wP_b}{|a-b|} \right\| e^{-\beeptilde \sigma} \| u \| d\sigma \\
& \lesssim (1 + \tau) e^{-\beeptilde \tau} \| u \| \lesssim e^{-\beep \tau} \| u \|.
\end{align*}
For $u \in \H$ the argument follows from the density of $\cC^\infty(\overline \BB)^2$ in $\H$.
\end{proof}

 The right spaces in which to consider the evolution of the solution and rapidity are those which preserve the linear decay obtained in Lemma~\ref{lem:decay1}.

\begin{definition}\label{def:XA}
 Fix $\varepsilon \in ( 0, \frac{1}{2})$. Let $\X := \{ \Phi \in \cC ([0,\infty),\H) \, | \, \| \Phi \|_\X < \infty \}$ with norm
\[
 \| \Phi \|_\X := \sup_{\tau > 0} \left( e^{\beep \tau} \| \Phi(\tau) \| \right)
\]
 and $\A := \{ a \in \cC^1([0,\infty),\RR^3) \, | \, a(0) = 0, \| a \|_\A < \infty \}$ with norm
\[
 \| a \|_\A := \sup_{\tau>0} \left( e^{\beep \tau} |\dot a(\tau)| + |a(\tau)| \right).
\]
 For $\delta > 0$ we define the closed subsets $\Xd := \{ \Phi \in \X \, | \, \| \Phi \|_\X \leq \delta \}$ and $\Ad := \{ a \in \A \, | \, |\dot a(\tau) | \leq \delta e^{- \beep \tau} \}$ of $\X$ and $\A$, respectively.
\end{definition}

\begin{remark}[Properties of $\Ad$]\label{rem:a}
 Note that $a \in \Ad$ automatically implies that $\| a \|_\A \lesssim \delta$, since
\begin{align}\label{aest}
 |a(\tau)| \leq \int_0^\tau |\dot a (\sigma)| d\sigma \leq \frac{\delta}{\frac{1}{2}-\varepsilon} \left[ 1 - e^{-\beep\tau} \right] \leq \frac{\delta}{\frac{1}{2}-\varepsilon} 
\end{align}
 by the fundamental theorem of calculus. Moreover, the limit
\[
 a_\infty := \lim_{\tau \to \infty} a(\tau) = \int_0^\infty \dot a (\sigma) d\sigma
\]
 exists, because the partial integrals $\int_0^\tau |\dot a(\sigma)| d\sigma$ are bounded for all $\tau \geq 0$ by \eqref{aest}. In particular, the convergence rate is
\begin{align}\label{eq:ainfty}
 |a_\infty - a(\tau)| \leq \int_\tau^\infty |\dot a(\sigma)| d\sigma \leq \frac{\delta}{\frac{1}{2}-\varepsilon} e^{-\beep\tau} \lesssim \delta e^{-\beep\tau}
\end{align}
 and $|a_\infty| \leq \frac{\delta}{\eep}$.
\end{remark}

 Since $\H$ and $\RR^3$ are complete, both $\X$ and $\A$ are Banach spaces. The following estimates are needed to show that the integral terms on the right hand side of \eqref{intmodeq} are small.

\begin{lemma}\label{lem:decay3}
 Fix $\varepsilon \in ( 0, \frac{1}{2})$. For $\delta>0$ sufficiently small, $\Phi \in \Xd$, $a \in \Ad$ and all $\tau \geq 0$, the following exponential bounds are satisfied:
\begin{align*}
 \| \hL_{a(\tau)} \Phi(\tau) \| + \| \N_{a(\tau)} (\Phi(\tau)) \| &\lesssim \delta^2 e^{-\dbeep \tau}, \\
 \| \P_{a_\infty} \partial_\tau \Psi_{a(\tau)} \| + \| ( \I - \Q_{a_\infty}) \partial_\tau \Psi_{a(\tau)} \| &\lesssim \delta^2 e^{-\dbeep \tau}.
\end{align*}
\end{lemma}

\begin{proof}
 By Remark~\ref{rem:a}, $a_\infty$ exists and $|a_\infty| \leq \frac{\delta}{\eep}$ can be made arbitrarily small (choose $\delta$ small). By definition, $\hL_a = \L_a'-\L_{a_\infty}'$. Thus it follows from the Lipschitz regularity of Lemma~\ref{Laconv1} and from the definition of $\Xd$ that, for $\delta$ sufficiently small,
\[
 \| \hL_{a(\tau)} \Phi(\tau) \| \leq \| \L_{a(\tau)}' - \L_{a_\infty}' \| \| \Phi(\tau) \| \lesssim |a(\tau) - a_\infty | \| \Phi(\tau) \| \lesssim \delta^2 e^{-\dbeep\tau}.
\]
 The nonlinear term
\[
 \N_{a(\tau)}(\Phi(\tau)) = \N(\Psi_{a(\tau)} + \Phi(\tau)) - \N(\Psi_{a(\tau)}) - \L_{a(\tau)}' \Phi(\tau) =
\begin{pmatrix}
 0 \\ 3 \psi_{1,a(\tau)} \phi_1(\tau)^2 + \phi_1(\tau)^3
\end{pmatrix}
\]
 is a map $\H \to \H$ by the Sobolev embedding $H^1 (\BB) \hookrightarrow L^6(\BB) \hookrightarrow L^2(\BB)$, and we obtain that
\begin{align*}
 \| \N_{a(\tau)} (\Phi(\tau)) \| &\lesssim \|\Psi_{a(\tau)} \|_{L^\infty(\BB)} \| \Phi(\tau) \|^2 + \| \Phi(\tau) \|^3
  \lesssim \| \Phi(\tau) \|^2 + \| \Phi(\tau) \|^3 \lesssim \delta^2 e^{-\dbeep \tau}.
\end{align*}
 This proves the first estimate. It remains to control $\partial_\tau \Psi_{a(\tau)}$.

 Let $\hat \q_{a(\tau),j} := \q_{a(\tau),j} - \q_{a_\infty,j}$. By the chain rule and Lemma~\ref{LaEV},
\begin{align}\label{eq:partialPsi}
 \partial_\tau \Psi_{a(\tau)} = \dot a^j(\tau) \partial_{a^j} \Psi_{a(\tau)} = \dot a^j(\tau) \q_{a(\tau),j} = \dot a^j(\tau) \q_{a_\infty,j} + \dot a^j(\tau) \hat \q_{a(\tau),j}.
\end{align}
 Since $\q_{a_\infty,j} \in \rg \Q_{a_\infty}$, which is transversal to $\P_{a_\infty}$ by Proposition~\ref{prop:proj}, $\P_{a_\infty} (\dot a^j(\tau) \q_{a_\infty,j}) = 0$ and
\[
 \| \P_{a_\infty} \partial_\tau \Psi_{a(\tau)} \| = |\dot a^j(\tau)| \| \hat \q_{a(\tau),j} \| \lesssim \delta e^{-\beep \tau} |a(\tau) - a_\infty| \lesssim \delta^2 e^{-\dbeep \tau},
\]
 where we applied the first estimate of Lemma~\ref{lem:decay1} and \eqref{eq:ainfty}. Similarly, since also $(\I - \Q_{a_\infty}) \q_{a_\infty,j} = 0$ holds,
\[
 \| ( \I - \Q_{a_\infty} ) \partial_\tau \Psi_{a(\tau)} \| \lesssim \delta^2 e^{-\dbeep \tau}. \qedhere
\]
\end{proof}

 The operators $\hL_{a(\tau)}$, $\N_{a(\tau)}$ and the spectral projections are locally Lipschitz continuous in their argument as well as in $a(\tau)$. The following result is the key property which allows us to control the integral term in Duhamel's formula \eqref{intmodeq} and set up a contraction argument.

\begin{lemma}[Lipschitz estimates for the inhomogeneous terms]\label{lem:decay4}
 Fix $\varepsilon \in ( 0, \frac{1}{2})$. For $\delta>0$ sufficiently small, $\Phi,\Upsilon \in \Xd$, $a,b \in \Ad$, and $\tau \geq 0$, the following Lipschitz bounds are satisfied:
\begin{align*}
 \| \hL_{a(\tau)} \Phi(\tau) - \hL_{b(\tau)} \Upsilon(\tau) \| &\lesssim \delta e^{-\dbeep \tau} \left( \|\Phi - \Upsilon \|_\X + \| a - b \|_\A \right), \\
 \| \N_{a(\tau)}(\Phi(\tau)) - \N_{b(\tau)} (\Upsilon(\tau)) \| &\lesssim \delta e^{-\dbeep \tau} \left( \|\Phi - \Upsilon \|_\X + \| a - b \|_\A \right), \\
 \| \P_{a_\infty} \partial_\tau \Psi_{a(\tau)} - \P_{b_\infty} \partial_\tau \Psi_{b(\tau)} \| &\lesssim \delta e^{-\dbeep \tau} \| a - b \|_\A, \\
 \| (\I-\Q_{a_\infty}) \partial_\tau \Psi_{a(\tau)} - (\I - \Q_{b_\infty} ) \partial_\tau \Psi_{b(\tau)} \| &\lesssim \delta e^{-\dbeep \tau} \| a - b \|_\A.
\end{align*}
\end{lemma}

\begin{proof}

\textbf{Estimate for $\hL_{a(\tau)}$.}
 We split up the norm in two parts by adding and subtracting $\hL_{b(\tau)} \Phi(\tau)$ in a standard way, and obtain
\begin{align}\label{E1}
 \| \hL_{a(\tau)} \Phi(\tau) - \hL_{b(\tau)} \Upsilon (\tau) \| &\leq \| \hL_{a(\tau)} - \hL_{b(\tau)} \| \|  \Phi(\tau) \| + \| \hL_{b(\tau)} \| \| \Phi(\tau) - \Upsilon(\tau) \| \nonumber \\
&\leq \delta e^{-\beep \tau} \| \hL_{a(\tau)} - \hL_{b(\tau)} \| + e^{-\beep \tau} \| \hL_{b(\tau)} \| \| \Phi - \Upsilon \|_\X.
\end{align}
 By Lemma~\ref{Laconv1} and Remark~\ref{rem:a},
\begin{align}\label{E2}
 \| \hL_{b(\tau)} \| = \| \L_{b(\tau)}' - \L_{b_\infty}' \| \lesssim | b(\tau) - b_\infty | \lesssim \delta e^{-\beep \tau},
\end{align}
 which establishes the estimate for the second term of the inequality~\eqref{E1}.
 To estimate the first term we first recall that for $u \in \H$
\[
 \hL_{a(\tau)} u = (\L_{a(\tau)}' - \L_{a_\infty}') u = \begin{pmatrix}
                                                         0 \\ 3 (\psi_{a(\tau),1}^2 - \psi_{a_\infty,1}^2)u_1
                                                        \end{pmatrix},
\]
 which implies
\begin{align*}
 \hL_{a(\tau)} u (\xi) - \hL_{b(\tau)} u(\xi) = \; & ( \psi_{a(\tau),1}^2 (\xi) - \psi_{a_\infty,1}^2 (\xi) - \psi_{b(\tau),1}^2 (\xi) + \psi_{b_\infty,1}^2 (\xi) ) \begin{pmatrix}
                                                  0 \\ 3 u_1 (\xi)
                                                 \end{pmatrix} \\
             = \;& \int_\tau^\infty \left[ \partial_\sigma \psi_{b(\sigma),1}^2 (\xi) - \partial_\sigma \psi_{a(\sigma),1}^2 (\xi) \right] d\sigma
                                                 \begin{pmatrix}
                                                   0 \\ 3 u_1 (\xi)
                                                  \end{pmatrix} \\
             = \;& \int_\tau^\infty \left[ \dot b^j(\sigma) \partial_{b^j} \psi_{b(\sigma),1}^2 (\xi) - \dot a^j(\sigma) \partial_{a^j} \psi_{a(\sigma),1}^2 (\xi) \right] d\sigma
                                                 \begin{pmatrix}
                                                   0 \\ 3 u_1 (\xi)
                                                  \end{pmatrix} .
\end{align*}
 Thus, by $H^1(\BB) \hookrightarrow L^2(\BB)$ and the fact that $\psi_a$ is smooth in $a$ (hence $\partial_a \psi^2_{a,1}$ is Lipschitz continuous with respect to $a$ small), this implies for the norm
\begin{align*}
 \| \hL_{a(\tau)} & u - \hL_{b(\tau)} u \| \\
\lesssim \; & \| u_1 \|_{L^2(\BB)} \int_\tau^\infty \| \dot b^j(\sigma) \partial_{b^j} \psi_{b(\sigma),1}^2 - \dot a^j(\sigma) \partial_{a^j} \psi_{a(\sigma),1}^2 \|_{L^\infty(\BB)} d\sigma \\
\lesssim \; & \| u \| \int_\tau^\infty |\dot b(\sigma) - \dot a (\sigma) | \| \partial_{b^j} \psi_{b(\sigma),1}^2 \|_{L^\infty(\BB)} + |\dot a(\sigma)| \| \partial_{b^j} \psi_{b(\sigma),1}^2 - \partial_{a^j} \psi_{a(\sigma),1}^2 \|_{L^\infty(\BB)} d\sigma \\
\lesssim \; & \| u \| \int_\tau^\infty e^{-\beep\sigma} \| a - b \|_\A + \delta e^{-\beep\tau} | a(\sigma) - b(\sigma) | d\sigma \\
\lesssim \; & e^{-\beep\tau} \| a - b \|_\A \| u \|,
\end{align*}
 hence
\begin{align}\label{E3}
 \| \hL_{a(\tau)} - \hL_{b(\tau)} \| \lesssim e^{-\beep\tau} \| a - b \|_\A.
\end{align}
 Together, \eqref{E1}--\eqref{E3} yield the desired inequality.

\medskip
\textbf{Estimate for $\N_{a(\tau)}$.}
 Instead of $\Phi(\tau)$ and $\Upsilon(\tau)$ we first consider the difference for $u,v \in \H$ fixed. Again, we decompose the norm into two parts and use the explicit form \eqref{Na} of the nonlinearity to obtain
\begin{align}\label{N1}
 \| \N_{a(\tau)} (u) &- \N_{b(\tau)} (v) \| \nonumber \\
\leq \; & \| \N_{a(\tau)} (u) - \N_{b(\tau)} (u) \| + \| \N_{b(\tau)} (u) - \N_{b(\tau)} (v) \| \nonumber \\
 \leq \; &
3 \| ( \psi_{a(\tau),1} - \psi_{b(\tau),1} ) u_1^2 \|_{L^2(\BB)} \nonumber \\
& + 3 \| \psi_{a(\tau),1} ( u_1^2 - v_1^2) \|_{L^2(\BB)} + \| u_1^3 - v_1^3 \|_{L^2(\BB)}.
\end{align}
 The first term can simply be estimated by
\begin{align}\label{N2}
 \| ( \psi_{1,a(\tau)} - \psi_{1,b(\tau)} ) u_1^2 \|_{L^2(\BB)} \lesssim | a(\tau) - b(\tau) | \| u_1^2 \|_{L^2(\BB)} \leq \| a - b \|_\A \| u \|^2,
\end{align}
 where we used the Sobolev embedding $H^1(\BB) \hookrightarrow L^4(\BB) \hookrightarrow L^2(\BB)$ for $u \in \H$ at the end. The generalized H\"older inequality for products and the Sobolev embedding $H^1(\BB) \hookrightarrow L^p(\BB)$, $p \in [1,6]$, for $u,v \in \H$ furthermore imply that
\begin{align}\label{N3}
 \| \psi_{1,a(\tau)} ( u_1^2 - v_1^2) \|_{L^2(\BB)} &\lesssim \| u_1^2 - v_1^2 \|_{L^2(\BB)} \leq \| u_1 - v_1 \|_{L^4(\BB)} \| u_1 + v_1 \|_{L^4(\BB)} \nonumber \\
&\lesssim \| u_1 - v_1 \|_{H^1(\BB)} ( \| u_1 \|_{H^1(\BB)} + \| v_1 \|_{H^1(\BB)} ) \nonumber \\
& \leq \| u - v \| ( \| u \| + \| v \| ),
\end{align}
 and, in a similar fashion,
\begin{align}\label{N4}
 \| u_1^3 - v_1^3 \|_{L^2(\BB)} &\leq \| u_1 - v_1 \|_{L^6(\BB)} \| u_1^2 + u_1 v_1 + v_1^2 \|_{L^3(\BB)} \nonumber \\
&\lesssim \| u - v \| (\| u_1 \|^2_{L^6(\BB)} + \| u_1 \|_{L^6(\BB)} \| v_1 \|_{L^6(\BB)} + \| v_1 \|^2_{L^6(\BB)} ) \nonumber \\
&\lesssim \| u - v \| ( \| u \| + \| v \| )^2.
\end{align}
Combining \eqref{N1}--\eqref{N4} for $u=\Phi(\tau)$ and $v=\Upsilon(\tau)$ with the definition of $\Xd$ yields
\[
 \| \N_{a(\tau)} (\Phi(\tau)) - \N_{b(\tau)} (\Upsilon(\tau)) \| \lesssim \delta^2 e^{-\dbeep\tau} \| a - b \|_\A + \delta e^{-\dbeep \tau} \| \Phi(\tau) - \Upsilon (\tau) \|_\X.
\]

\medskip
\textbf{Estimates for $\partial_\tau \Psi_{a(\tau)}$.}
 Recall that by Proposition~\ref{prop:proj} the projections $\Q_a$ and $\P_a$ are transversal and that the eigenfunctions $\q_{a,j} = \partial_{a^j} \Psi_{a}$, $j \in \{ 1,2,3 \}$ of $\L_a$, derived in Lemma~\ref{LaEV}, span $\rg \Q_a$. Thus for $\hat \q_{a,j} := \q_{a,j} - \q_{a_\infty,j}$ we have by the chain rule that
\[
 \P_{a_\infty} \partial_\tau \Psi_{a(\tau)} = \P_{a_\infty} \dot a^j(\tau) \partial_{a^j} \Psi_{a(\tau)} = \dot a^j(\tau) \P_{a_\infty} \q_{a(\tau),j} = \dot a^j(\tau) \P_{a_\infty} \hat \q_{a(\tau),j}.
\]
 The triangle inequality thus once more implies
\begin{align}
 \| \P_{a_\infty} \partial_\tau &\Psi_{a(\tau)} - \P_{b_\infty} \partial_\tau \Psi_{b(\tau)} \| \nonumber \\
&\leq \| ( \dot a^j(\tau) \P_{a_\infty} - \dot b^j(\tau) \P_{b_\infty} ) \hat \q_{a(\tau),j} \| + \| \dot b^j(\tau) \P_{b_\infty} (\hat \q_{a(\tau),j} - \hat \q_{b(\tau),j} )\| \nonumber \\
&\leq |\dot a^j(\tau) - \dot b^j(\tau) | \| \hat \q_{a(\tau),j} \| + |\dot b^j(\tau)| \| \P_{a_\infty} - \P_{b_\infty} \| \|\hat \q_{a(\tau),j} \| \nonumber \\
&\quad + |\dot b^j(\tau)| \| \hat \q_{a(\tau),j} - \hat \q_{b(\tau),j} \|. \label{P1}
\end{align}
 We use the Lipschitz estimates already obtained in Lemma~\ref{lem:decay1}, together we the definition of $\hat \q_{a,j}$ and \eqref{eq:ainfty}, to derive the bounds
\begin{align}\label{qhat}
 \| \hat \q_{a(\tau),j} \| = \| \q_{a(\tau),j} - \q_{a_\infty,j} \| \lesssim |a(\tau) - a_\infty | \lesssim \delta e^{-\beep\tau}
\end{align}
 and
\begin{align}\label{qhatdiff}
 \| \hat \q_{a(\tau),j} - \hat \q_{b(\tau),j} \| 
 &\leq \int_\tau^\infty \| \partial_\sigma \hat \q_{b(\sigma),j} - \partial_\sigma \hat \q_{a(\sigma),j} \| d\sigma \nonumber \\
&\leq \int_\tau^\infty | \dot b^k(\sigma) - \dot a^k(\sigma) | \| \partial_{b^k} \hat \q_{b(\sigma),j} \| + |\dot a^k(\sigma)| \| \partial_{b^k} \hat \q_{b(\sigma),j} - \partial_{a^k} \hat \q_{a(\sigma),j} \| d\sigma \nonumber \\
&\lesssim \int_\tau^\infty e^{-\beep\sigma} \| a - b \|_\A + \delta e^{-\beep\sigma} | a(\sigma) - b(\sigma) | d\sigma \nonumber \\
&\lesssim e^{-\beep\tau} \| a - b \|_\A.
\end{align}
 Lemma~\ref{lem:decay1} and \eqref{P1} together with the definition of $\Ad$ thus yield
\begin{align*}
 \| \P_{a_\infty} \partial_\tau \Psi_{a(\tau)} - \P_{b_\infty} \partial_\tau \Psi_{b(\tau)} \| 
&\lesssim  \delta e^{-\dbeep\tau} \| a- b\|_\A + \delta^2 e^{-\dbeep\tau} |a_\infty-b_\infty| \\
& \quad + \delta e^{-\dbeep\tau} \| a- b\|_\A \\
&\lesssim \delta e^{-\dbeep\tau} \| a - b \|_\A.
\end{align*}
 In a similar fashion, since $\Q_a$ is transversal to $\I-\Q_a$ and hence
\[
 (\I - \Q_{a_\infty}) \partial_\tau \Psi_{a(\tau)} = \dot a^j(\tau) (\I - \Q_{a_\infty}) \hat \q_{a(\tau),j},
\]
 it follows that the last Lipschitz estimate holds.
\end{proof}

\begin{remark}
 The estimate~\eqref{N4} in the previous proof makes clear why the cubic nonlinearity can be controlled in the Banach space $\H \approx H^1(\BB) \times L^2(\BB)$, and why the same approach (via Sobolev embedding and generalized H\"older inequality) cannot work for $p>3$.
\end{remark}

\subsection{The modulation equation for $a$}
\label{ssec4.2}

 By choosing the rapidity $a(\tau)$ in a suitable way we will suppress the instability of $\Q_{a_\infty}$ that arises from the Lorentz symmetry of the equation. We follow the approach in \cite{DS}[Sec.\ 5.3 and 5.4], which assumes that $\Phi$ is given, to derive an equation for $a$ by applying the spectral projection $\Q_{a_\infty}$ corresponding to the Lorentz instability. Since $\S_a (\tau) \Q_{a,j} = \Q_{a,j}$ by Proposition~\ref{prop:proj}, multiplication of the integral equation \eqref{intmodeq} by $\Q_{a_\infty,j}$, $j \in \{1,2,3\}$, yields
\[
 \Q_{a_\infty,j} \Phi(\tau) = \Q_{a_\infty,j} u + \Q_{a_\infty,j} \int_0^\tau \left[ \hL_{a(\sigma)} \Phi(\sigma) + \N_{a(\sigma)}(\Phi(\sigma)) - \partial_\sigma \Psi_{a(\sigma)} \right] d \sigma
\]
 The right hand side will not vanish for $\tau = 0$ unless $\Q_{a_\infty} u = 0$. Since we are only interested in the long-term evolution it however suffices to assume that $\Q_{a_\infty} \Phi(\tau)$ vanishes for large $\tau$. For a smooth cut-off function $\chi$, which is $1$ on $[0,1]$, $0$ for $\tau\geq 4$ and derivative $|\dot\chi| \leq 1$ everywhere, and the ansatz $\Q_{a_\infty} \Phi(\tau) = \chi(\tau) v \in \rg \Q_{a_\infty}$, evaluation at $\tau=0$ yields $\Q_{a_\infty} \Phi(\tau) = \chi(\tau) \Q_{a_\infty} u$. Thus the modulation equation for $a$ is
\begin{equation}\label{Qmod}
 [1 - \chi(\tau)] \Q_{a_\infty,j} u + \Q_{a_\infty,j} \int_0^\tau \left[ \hL_{a(\sigma)} \Phi(\sigma) + \N_{a(\sigma)}(\Phi(\sigma)) - \partial_\sigma \Psi_{a(\sigma)} \right] d\sigma = 0.
\end{equation}
 More explicitly, using $\hat \q_{a(\tau),j} = \q_{a(\tau),j} - \q_{a_\infty,j}$ (see Lemma~\ref{LaEV} for the explicit form of the eigenvectors that has been used to replace the last term $\partial_\sigma \Psi_{a(\sigma)}$ in \eqref{Qmod}), we may write for $j\in\{1,2,3\}$
\begin{align}\label{modeq:a}
 a_j(\tau) \q_{a_\infty,j} = \; & [1-\chi(\tau)] \Q_{a_\infty,j} u + \Q_{a_\infty,j} \int_0^\tau \left[ \hL_{a(\sigma)} \Phi(\sigma) + \N_{a(\sigma)}(\Phi(\sigma)) \right] d\sigma \nonumber \\
& - \Q_{a_\infty,j} \int_0^\tau \dot a^k(\sigma) \hat \q_{a(\sigma),k} d\sigma \nonumber \\
= \; & - \int_0^\tau \dot\chi(\sigma) \Q_{a_\infty,j} u d\sigma
+ \int_0^\tau \Q_{a_\infty,j} \left[ \hL_{a(\sigma)} \Phi(\sigma) + \N_{a(\sigma)}(\Phi(\sigma)) \right] d\sigma \nonumber \\
 & -  \int_0^\tau \Q_{a_\infty,j} \dot a^k(\sigma) \hat \q_{a(\sigma),k} d\sigma . 
\end{align}
 For $\Phi$ sufficiently small, equation~\eqref{modeq:a} can be solved using the Banach fixed point theorem. The contraction argument is based on the Lipschitz bounds derived in Lemmas~\ref{lem:decay1} and \ref{lem:decay4} as well as the estimates from Lemma~\ref{lem:decay3}.

\begin{proposition}\label{lem:a}
 Fix $\varepsilon \in ( 0, \frac{1}{2})$. There exist $\delta>0$ and $c>0$ such that the modulation equations \eqref{modeq:a} with $\Phi \in \Xd$ and initial value $\| u \| \leq \frac{\delta}{c}$ have a unique solution $a \in \Ad$. Moreover, the solution operator $\Phi \mapsto a \colon \Xd \to \Ad$ is Lipschitz continuous.
\end{proposition}

\begin{proof}
 Taking equation~\eqref{modeq:a} in the inner product with $\q_{a_\infty,j}$, $j \in \{1,2,3\}$, inspires us to define $\G_u = (G_{u,1},G_{u,2},G_{u,3})$ such that
\begin{align}\label{aj}
 a_j(\tau) & = \int_0^\tau \| \q_{a_\infty,j} \|^{-2} \big\langle - \dot\chi(\sigma) \Q_{a_\infty,j} u + \Q_{a_\infty,j} \left[ \hL_{a(\sigma)} \Phi(\sigma) + \N_{a(\sigma)}(\Phi(\sigma)) \right] \nonumber \\
& \qquad \qquad \qquad \qquad - \Q_{a_\infty,j} \dot a^k(\sigma) \hat \q_{a(\sigma),k} , \q_{a_\infty,j} \big\rangle \, d\sigma \\
 & =: G_{u,j}(\Phi,a)(\tau). \nonumber
\end{align}
 It remains to be shown that $\G_u$ maps $\Ad$ to itself, $\G_u$ is a contraction and that $a$ depends Lipschitz continuously on $\Phi$.

\medskip
\textbf{Step 1. $\G_u$ maps $\Ad$ to itself.}
 Let $a \in \Ad$. We treat each component of the integrand of $\G_u$ individually. By definition, the derivative of $\chi$ is bounded and vanishes for large $\sigma$, thus trivially
\[
 \| \dot\chi(\sigma) \Q_{a_\infty,j} u \| \leq | \dot\chi(\sigma)| \| u \| \lesssim \frac{\delta}{c} e^{-\dbeep \sigma}.
\]
 Moreover, by Lemma~\ref{lem:decay3} we can estimate the term involving $\hL_{a(\sigma)}$ and $\N_{a(\sigma)}$, i.e.,
\[
 \left\| \Q_{a_\infty,j} \left[ \hL_{a(\sigma)} \Phi(\sigma) + \N_{a(\sigma)} (\Phi(\sigma)) \right] \right\| \lesssim \delta^2 e^{-\dbeep\sigma}.
\]
 The last term in \eqref{aj} is estimated using \eqref{qhat} and the definition of $\Ad$,
\[
 \| \Q_{a_\infty,j} \dot a^k(\sigma) \hat \q_{a(\sigma),k} \| \leq |\dot a^k(\sigma)| \| \hat \q_{a(\sigma),k} \| \lesssim \delta^2 e^{-\dbeep\sigma}.
\]
 Summing up and integrating we thus obtain by the Cauchy--Schwarz inequality and the fact that $\q_{a_\infty,j}$ is smooth and uniformly bounded from below for $a_\infty$ sufficiently small the estimate
\[
 | \dot G_{u,j}(\Phi,a)(\sigma) | \lesssim \| \q_{a_\infty,j} \|^{-1} \left( \frac{\delta}{c} + \delta^2 \right) e^{-\dbeep\sigma} \lesssim \frac{\delta}{c} e^{-\dbeep\sigma}.
\]
 In particular, $\G_u(\Phi,a) \in \Ad$ for $c$ large.

\medskip
\textbf{Step 2. $\G$ is a contraction on $\Ad$.}
 Suppose $a,b \in \Ad$. In a similar fashion as in Step 1 we employ the Lipschitz bound of Lemma~\ref{lem:decay1} to obtain
\begin{align*}
 \| \dot\chi(\sigma) \Q_{a_\infty,j} u - \dot\chi(\sigma) \Q_{b_\infty,j} u \| &\leq | \dot \chi(\sigma) | \| \Q_{a_\infty,j} - \Q_{b_\infty,j} \| \| u \| \\ & \lesssim \frac{\delta}{c} e^{-\dbeep\sigma} | a_\infty - b_\infty | \lesssim \delta e^{-\dbeep\sigma} \| a - b \|_\A.
\end{align*}
 Lemmas~\ref{lem:decay1}, \ref{lem:decay3} and \ref{lem:decay4} imply for the second term
\begin{align*}
 \Big\| \Q_{a_\infty,j} & \left[ \hL_{a(\sigma)} \Phi(\sigma) + \N_{a(\sigma)} (\Phi(\sigma)) \right] - \Q_{b_\infty,j} \left[ \hL_{b(\sigma)} \Phi(\sigma) + \N_{b(\sigma)} (\Phi(\sigma)) \right] \Big\| \\
\leq \; & \| ( \hL_{a(\sigma)} - \hL_{b(\sigma)}) \Phi(\sigma) \| + \| \N_{a(\sigma)} (\Phi(\sigma)) - \N_{b(\sigma)} (\Phi(\sigma)) \| \\
& + \| \Q_{a_\infty,j} - \Q_{b_\infty,j} \| \left( \| \hL_{a(\sigma)} \Phi(\sigma) \| + \| \N_{a(\sigma)} (\Phi(\sigma)) \| \right) \\
\lesssim \; & \delta e^{-\dbeep\sigma} \| a - b \|_\A + |a_\infty-b_\infty| \delta^2 e^{-\dbeep\sigma} \\
\lesssim \; & \delta e^{-\dbeep\sigma} \| a - b \|_\A.
\end{align*}
 Using \eqref{qhat}--\eqref{qhatdiff} and the second estimate in Lemma~\ref{lem:decay1} the third term in the integrand is estimated by
\begin{align*}
 \| \dot a^k(\sigma) &\Q_{a_\infty,j} \hat \q_{a(\sigma),k} - \dot b^k(\sigma) \Q_{b_\infty,j} \hat \q_{b(\sigma),k} \| \\
&\leq |\dot a^k(\sigma) - \dot b^k(\sigma) | \| \hat \q_{a(\sigma),k} \| + |\dot b^k(\sigma)| \| \Q_{a_\infty,j} \hat \q_{a(\sigma),k} - \Q_{b_\infty,j} \hat \q_{b(\sigma),k} \| \\
&\lesssim \delta e^{-\dbeep\sigma} \| a - b \|_\A + \delta e^{-\beep\sigma} \left( \| \Q_{a_\infty,j} (\hat \q_{a(\sigma),k} - \hat \q_{b(\sigma),k} )\| + \| \Q_{a_\infty,j} - \Q_{b_\infty,j} \| \| \hat \q_{b(\sigma),k} \| \right) \\
&\leq \delta e^{-\dbeep\sigma} \| a - b \|_\A + \delta e^{-\beep\sigma} \left( e^{-\beep\sigma} \| a- b\|_\A + |a_\infty-b_\infty| \delta e^{-\beep\sigma} \right) \\
&\lesssim \delta e^{-\dbeep\sigma} \| a - b \|_\A.
\end{align*}
 We are now in a position to estimate the integrand of \eqref{aj} using the Cauchy--Schwarz inequality. With the notation
\[
 H_{u,j}(\Phi,a)(\sigma) := - \dot\chi(\sigma) \Q_{a_\infty,j} u + \Q_{a_\infty,j} \left[ \hL_{a(\sigma)} \Phi(\sigma) + \N_{a(\sigma)}(\Phi(\sigma)) \right] - \Q_{a_\infty,j} \dot a^k(\sigma) \hat \q_{a(\sigma),k},
\]
 we just derived the estimate
\[
 \| H_{u,j}(\Phi,a)(\sigma) - H_{u,j}(\Phi,b)(\sigma) \| \lesssim \delta e^{-\dbeep\sigma} \| a- b \|_\A,
\]
 and in Step 1 the estimate
\[
 \| H_{u,j}(\Phi,a)(\sigma) \| \lesssim \delta e^{-\dbeep\sigma}.
\]
 Together this yields
\begin{align}\label{Gj1}
 | \dot G_{u,j}(\Phi,a)&(\sigma) - \dot G_{u,j}(\Phi,b)(\sigma) | \nonumber \\
= \; & \big| \| \q_{a_\infty,j} \|^{-2} \langle H_{u,j}(\Phi,a)(\sigma) - H_{u,j}(\Phi,b)(\sigma), \q_{a_\infty,j} \rangle  \nonumber \\
& + ( \| \q_{a_\infty,j} \|^{-2} - \| \q_{b_\infty,j} \|^{-2} ) \langle H_{u,j}(\Phi,b)(\sigma),\q_{a_\infty,j} \rangle \nonumber \\
& + \| \q_{b_\infty,j} \|^{-2} \langle H_{u,j}(\Phi,b)(\sigma) , \q_{a_\infty,j} - \q_{b_\infty,j} \rangle  \big| \nonumber \\
\lesssim \; & \| \q_{a_\infty,j} \|^{-1} \delta e^{-\dbeep\sigma} \|a-b\|_\A + |a_\infty - b_\infty| \delta e^{-\dbeep\sigma} \| \q_{a_\infty,j} \| \nonumber \\
&+ \delta e^{-\dbeep\sigma} \| \q_{a_\infty,j} -\q_{b_\infty,j}\| \nonumber \\
\lesssim \; & \delta e^{-\dbeep\sigma} \| a- b\|_\A,
\end{align}
 since $\q_{a_\infty,j}$ depends smoothly on $a_\infty$ and is therefore bounded uniformly for small $a_\infty$, and the difference $\q_{a_\infty,j} - \q_{b_\infty,j}$ is controlled by Lemma~\ref{lem:decay1}. Hence for all $a,b\in\Ad$,
\begin{align}\label{G1}
 \| \G_u(\Phi,a) &- \G_u(\Phi,b) \|_\A \nonumber \\
&\leq \sup_{\tau > 0}  \left( e^{\beep\tau}  | \dot\G_u(\Phi,a)(\tau) - \dot\G_u(\Phi,b)(\tau) | + | \G_u(\Phi,a)(\tau) - \G_u(\Phi,b)(\tau)| \right) \nonumber \\
&\leq \sup_{\tau \geq 0} \left( \delta e^{-\dbeep \tau} \|a-b\|_\A \right) +
\int_0^\infty | \dot \G_u(\Phi,a)(\sigma) - \dot \G_u(\Phi,b)(\sigma) | d\sigma \nonumber \\
&\lesssim \delta \| a- b\|_\A,
\end{align}
 which proves that $\G_u$ is a contraction for $\delta$ sufficiently small. The Banach fixed point theorem implies the existence and uniqueness of $a \in \Ad$ such that $a(\tau) = \G_u(\Phi,a)(\tau)$ for $\tau \geq 0$, which solves \eqref{aj} and therefore \eqref{modeq:a}.

\medskip
\textbf{Step 3. The solution operator is Lipschitz continuous.}
 Suppose $a = \G_u(\Phi,a)$ and $b = \G_u(\Upsilon,b)$ for $\Phi,\Upsilon \in \Xd$. By \eqref{aj} this implies, for all $j \in \{1,2,3\}$,
\begin{align*}
 | \dot a^j(\tau) - \dot b^j(\tau) | \leq \; & | \dot G_{u,j}(\Phi,a)(\tau) - \dot G_{u,j}(\Phi,b)(\tau) | + | \dot G_{u,j}(\Phi,b)(\tau) - \dot G_{u,j}(\Upsilon,b)(\tau) |.
\end{align*}
 The first term was already estimated in \eqref{Gj1} of Step 2. The second term can be estimated in a similar fashion using the Cauchy--Schwarz inequality and Lemma~\ref{lem:decay4}. More precisely, since $b$ is small,
\begin{align*}
 | \dot G_{u,j}(\Phi,b)(\tau) - \dot G_{u,j}(\Upsilon,b)(\tau) |
&\leq \| \q_{b_\infty,j} \|^{-1} \| H_{u,j}(\Phi,b)(\tau) - H_{u,j}(\Upsilon,b)(\tau) \| \\
&\lesssim \| \hL_{b(\tau)} \Phi(\tau) - \hL_{b(\tau)} \Upsilon(\tau) \| + \| \N_{b(\tau)} (\Phi(\tau)) + \N_{b(\tau)} (\Upsilon(\tau)) \| \\
&\lesssim \delta e^{-\dbeep\tau} \| \Phi - \Upsilon \|_\X.
\end{align*}
 Together with \eqref{Gj1} we thus obtain that
\[
 | \dot a(\tau) - \dot b(\tau) | \lesssim \delta e^{-\dbeep\tau} \left( \| a- b\|_\A + \| \Phi - \Upsilon \|_\X \right),
\]
 and by integration also
\[
| a(\tau) - b(\tau)| \lesssim \delta \left( \| a- b \|_\A + \| \Phi - \Upsilon \|_\X \right).
\]
 Hence the desired Lipschitz estimate is obtained for $\delta$ sufficiently small,
\[
 \| a - b \|_\A = \sup_{\tau > 0} \left( e^{\beep\tau} |\dot a(\tau) - \dot b(\tau) | + | a(\tau) - b(\tau)| \right) \lesssim \| \Phi - \Upsilon \|_\X. \qedhere
\]
\end{proof}

\subsection{A modified equation for the time-translation instability}
\label{ssec4.3}

 On the subspace $\rg \P_{a_\infty}$ we add a correction term to account for the time-translation instability. This is a version of the Lyapunov--Perron method (see, e.g., \cite[Sec.\ 3.3]{NS}), initially developed for dealing with instabilities arising in finite-dimensional dynamical systems. The choice of the correction term,
\begin{align}\label{C}
 \C_u(\Phi,a) := \P_{a_\infty} u + \P_{a_\infty} \int_0^\infty e^{-\sigma} \left[ \hL_{a(\sigma)} \Phi(\sigma) + \N_{a(\sigma)}(\Phi(\sigma)) - \partial_\sigma \Psi_{a(\sigma)} \right] d\sigma,
\end{align}
is motivated from applying $\P_{a_\infty}$ to the Duhamel formula \eqref{intmodeq}, which together with the growth rate $\S_a(\tau) \P_a = e^\tau \P_a$ obtained in Proposition~\ref{prop:proj} yields
\begin{align*}
 \P_{a_\infty} \Phi(\tau) = \; & e^\tau \P_{a_\infty} u + e^\tau \P_{a_\infty} \int_0^\tau e^{-\sigma} \left[ \hL_{a(\sigma)} \Phi(\sigma) + \N_{a(\sigma)}(\Phi(\sigma)) - \partial_\sigma \Psi_{a(\sigma)} \right] d\sigma \\
 = \; & \S_{a_\infty} (\tau) \left( \P_{a_\infty} u + \P_{a_\infty} \int_0^\tau e^{-\sigma} \left[ \hL_{a(\sigma)} \Phi(\sigma) + \N_{a(\sigma)}(\Phi(\sigma)) - \partial_\sigma \Psi_{a(\sigma)} \right] d\sigma \right).
\end{align*}
Thus we consider the modified equation
\begin{equation}\label{Peq}
\begin{aligned}
 \Phi(\tau) = \; & \S_{a_\infty} (\tau) [ u - \C_u(\Phi,a) ]\\
             & + \int_0^\tau \S_{a_\infty}(\tau-\sigma) \left[ \hL_{a(\sigma)} \Phi(\sigma) + \N_{a(\sigma)}(\Phi(\sigma)) - \partial_\sigma \Psi_{a(\sigma)} \right] d\sigma.
\end{aligned}
\end{equation}
For simplicity, we denote the right hand side of \eqref{Peq} by $\wK_u(\Phi,a)(\tau)$, so that
\[
 \K_u(\Phi,a)(\tau) = \wK_u(\Phi,a)(\tau) + \S_{a_\infty}(\tau) \C_u(\Phi,a),
\]
 where $\K_u(\Phi,a)(\tau)$ is the right hand side of the unmodified weak equation \eqref{intmodeq}. We will see that $\wK_u(\Phi,a)$ is a contraction, which again allows us to apply the Banach fixed point theorem to solve~\eqref{Peq}.

\begin{proposition}\label{prop:modsol}
 Fix $\varepsilon \in ( 0, \frac{1}{2})$. There exist $c >0$ sufficiently large and $\delta>0$ sufficiently small so that $\| u \| \leq \frac{\delta}{c}$ implies the existence and uniqueness of functions $\Phi \in \Xd$, $a \in \Ad$ such that the modified equation \eqref{Peq} holds for all $\tau \geq 0$. Moreover, the solution operator $u \mapsto (\Phi,a) \colon B_{\delta/c} := \{ u \in \H \, | \, \| u \| \leq \frac{\delta}{c} \} \subseteq \H \to \Xd \times \Ad$ is Lipschitz continuous.
\end{proposition}

Note that $a = a_u(\Phi)$ in the above proposition due to the discussion on the Lorentz symmetry in Proposition~\ref{lem:a} in the previous section.

\begin{proof}
 The existence proof consists of two parts. First, it is shown that for sufficiently small $\delta>0$, $\Phi \in \Xd$ implies $\wK_u(\Phi,a) \in \Xd$ for a unique $a=a_u(\Phi) \in \Ad$. Subsequently we prove that $\wK_u$ is a contraction, i.e.,
\begin{align}\label{Ku:contraction}
 \| \wK_u(\Phi,a) - \wK_u (\Upsilon,b) \|_\X \lesssim \delta \| \Phi - \Upsilon \|_\X
\end{align}
holds for $\Phi,\Upsilon \in \Xd$ and $a=a_u(\Phi), b=b_u(\Phi) \in \Ad$ as derived in Proposition~\ref{lem:a}. For $\delta$ sufficiently small \eqref{Ku:contraction} thus is a contraction and the existence and uniqueness of a fixed point $\Phi_u = \wK_u (\Phi_u)$ follows from the contraction mapping principle.

\medskip
{\bf Step 1. $\wK_u$ maps $\Xd$ into itself.}
 Let $\delta,c>0$ as in Proposition~\ref{lem:a}, $\Phi \in \Xd$ and $u$ be an initial value satisfying $\| u \| \leq \frac{\delta}{c}$. Then by Proposition~\ref{lem:a} there exists a unique function $a := a_u (\Phi) \in \Ad$ which satisfies the modulation equation~\eqref{modeq:a}. We prove that $\wK_u(\Phi,a) \in \Xd$ by considering its projections along $\Q_{a_\infty}$, $\P_{a_\infty}$ and $\wP_{a_\infty}$.

 Recall that $\S_{a_\infty}(\tau) \Q_{a_\infty} = \Q_{a_\infty} \S_{a_\infty}(\tau) = \Q_{a_\infty}$ by Proposition~\ref{prop:proj} and that the integral term of $\Q_{a_\infty} \wK_u(\Phi,a)$ projected along $\Q_{a_\infty,j}$, $j \in \{ 1,2,3\}$, is precisely given through equation \eqref{Qmod} using a smooth cut-off function $\chi \colon [0, \infty) \to [0,1]$. Together with the fact that $\C_u(\Phi,a) \in \rg \P_{a_\infty}$ and the transversality of $\P_{a_\infty}$ and $\Q_{a_\infty}$ by Proposition~\ref{prop:proj}, this yields
\begin{align}
 \Q_{a_\infty} \wK_u(\Phi,a)(\tau) = \; & \Q_{a_\infty} [ u - \C_u(\Phi,a) ] \nonumber \\
& + \Q_{a_\infty} \int_0^\tau \left[ \hL_{a(\sigma)} \Phi(\sigma) + \N_{a(\sigma)}(\Phi(\sigma)) - \partial_\sigma \Psi_{a(\sigma)} \right] d\sigma \nonumber \\
= \; & \Q_{a_\infty} u - \Q_{a_\infty} \C_u(\Phi,a) - [1-\chi(\tau)] \Q_{a_\infty} u \nonumber \\
= \; & \chi(\tau) \Q_{a_\infty} u - \Q_{a_\infty} \C_u(\Phi,a) \nonumber \\
= \; & \chi(\tau) \Q_{a_\infty} u. \label{Qsimplified}
\end{align}
 Hence by the initial assumption on $u$, and the fact that $\chi$ is defined such that $\| \chi \| \leq 1 \lesssim e^{-\dbeep \tau}$ on the interval $[0,4]$ and zero outside, we obtain
\[
 \| \Q_{a_\infty} \wK_u(\Phi,a)(\tau) \| \lesssim \frac{\delta}{c} e^{- \dbeep \tau}.
\]
This implies
\[
 \| \Q_{a_\infty} \wK_u(\Phi,a) \|_\X \lesssim \frac{\delta}{c},
\]
and for $c$ sufficiently large thus $\Q_{a_\infty} \wK_u(\Phi,a) \in \mathcal{X}^\varepsilon_{\delta / 3}$.

Since $\P_{a_\infty}^2 = \P_{a_\infty}$ is a projection that commutes with the semigroup and satisfies $\S_{a_\infty}(\tau) \P_{a_\infty} = e^\tau \P_{a_\infty}$ by Proposition~\ref{prop:proj}, and by definition of $\C_u(\Phi,a)$ in \eqref{C}, we have
\begin{align}
 \P_{a_\infty} \wK_u(\Phi,a)(\tau) = \; & e^\tau \P_{a_\infty} \left[ u - \C_u(\Phi,a) \right] \nonumber \\
& + \int_0^\tau e^{\tau-\sigma} \P_{a_\infty} \left[ \hL_{a(\sigma)} \Phi(\sigma) + \N_{a(\sigma)}(\Phi(\sigma)) - \partial_\sigma \Psi_{a(\sigma)} \right] d\sigma \nonumber \\
= \; & - \int_\tau^\infty e^{\tau-\sigma} \P_{a_\infty} \left[ \hL_{a(\sigma)} \Phi(\sigma) + \N_{a(\sigma)}(\Phi(\sigma)) - \partial_\sigma \Psi_{a(\sigma)} \right] d\sigma. \label{Psimplified}
\end{align}
The estimates on the integrand from Lemma~\ref{lem:decay3} thus imply the decay
\begin{align*}
 \| \P_{a_\infty} \wK_u(\Phi,a)(\tau) \| &\lesssim \int_\tau^\infty e^{\tau-\sigma} \delta^2 e^{-\dbeep\sigma} d\sigma \lesssim \delta^2 e^{-\dbeep\tau}.
\end{align*}
By definition of $\X$,
\[
 \| \P_{a_\infty} \wK_u(\Phi,a) \|_\X \lesssim \delta^2
\]
and thus $\P_{a_\infty} \wK_u(\Phi,a) \in \mathcal{X}^\varepsilon_{\delta/3}$ for $\delta$ sufficiently small.

Finally, we consider $\wP_{a_\infty} \wK_u(\Phi,a) = [ \I - \Q_{a_\infty} - \P_{a_\infty}] \wK_u(\Phi,a)$, which is
\begin{align}\label{wPKu}
 \wP_{a_\infty} & \wK_u(\Phi,a)(\tau) \nonumber \\
 = \; & \wP_{a_\infty} \S_{a_\infty} (\tau) [ u - \C_u(\Phi,a) ] \nonumber \\
& +  \int_0^\tau \wP_{a_\infty} \S_{a_\infty} (\tau - \sigma) \left[ \hL_{a(\sigma)} \Phi(\sigma) + \N_{a(\sigma)}(\Phi(\sigma)) - \partial_\sigma \Psi_{a(\sigma)} \right] d\sigma \nonumber \\
= \; & \S_{a_\infty} (\tau) \wP_{a_\infty} u + \int_0^\tau \wP_{a_\infty} \S_{a_\infty} (\tau - \sigma) \left[ \hL_{a(\sigma)} \Phi(\sigma) + \N_{a(\sigma)}(\Phi(\sigma)) - \partial_\sigma \Psi_{a(\sigma)} \right] d\sigma,
\end{align}
 again by transversality of $\wP_{a_\infty}$ and $\P_{a_\infty}$ from Proposition~\ref{prop:proj}. Thus by the growth bounds on $\S_{a_\infty} \wP_{a_\infty}$ obtained in Proposition~\ref{prop:proj} as well as Lemma~\ref{lem:decay3}, this yields
\begin{align*}
 \| \wP_{a_\infty} \wK_u(\Phi,a)(\tau) \| & \lesssim e^{-\beep \tau} \| u \| +  \int_0^\tau e^{-\beep (\tau-\sigma)} \delta^2 e^{- \dbeep \sigma} d\sigma \lesssim e^{-\beep \tau} \left( \frac{\delta}{c} + \delta^2 \right).
\end{align*}
 Hence
\[
 \| \wP_{a_\infty} \wK_u(\Phi,a)(\tau) \|_\X \lesssim \frac{\delta}{c} + \delta^2,
\]
 and $\wP_{a_\infty} \wK_u(\Phi,a) \in \mathcal{X}^\varepsilon_{\delta / 3}$ for $c$ sufficiently large and $\delta$ sufficiently small.

 If we combine these three results, that is, $\Q_{a_\infty} \wK_u(\Phi,a), \P_{a_\infty} \wK_u(\Phi,a), \wP_{a_\infty} \wK_u(\Phi,a) \in \mathcal{X}^\varepsilon_{\delta /3}$, then by the triangle inequality, $\wK_u(\Phi,a) = [\Q_{a_\infty} + \P_{a_\infty} + \wP_{a_\infty}] \wK_u(\Phi,a) \in \Xd$ for $\Phi \in \Xd$.

\medskip
{\bf Step 2. $\wK_u$ is a contraction on $\Xd$.}
 As in the first step of this proof we consider the three projections of $\wK_u$ separately, and combine them at the end. Let $\Phi,\Upsilon \in \Xd$ and let $a,b \in \Ad$ be the unique rapidities associated to $\Phi,\Upsilon$ via Proposition~\ref{lem:a}, respectively.

 From \eqref{Qsimplified} and Lemma~\ref{lem:decay1} it follows that
\begin{align*}
 \| \Q_{a_\infty} \wK_u(\Phi,a)(\tau) - \Q_{b_\infty} \wK_u(\Upsilon,b)(\tau) \| &= | \chi(\tau) | \| \Q_{a_\infty} u - \Q_{b_\infty} u \| \\
& \lesssim \frac{\delta}{c} | a_\infty - b_\infty | \lesssim \frac{\delta}{c} e^{-\beep \tau} \| a- b \|_\A,
\end{align*}
 hence by the Lipschitz regularity of the solution $a=a_u(\Phi), b=b_u(\Phi)$ obtained in Proposition~\ref{lem:a} also
\begin{align}\label{K2}
 \| \Q_{a_\infty} \wK_u(\Phi,a) - \Q_{b_\infty} \wK_u(\Upsilon,b) \|_\X \lesssim \delta \| \Phi -\Upsilon \|_\X.
\end{align}

 The simplification \eqref{Psimplified} together with the Lipschitz estimates on the integrand from Lemma~\ref{lem:decay4} and the Lipschitz regularity of $a$ with respect to $\Phi$ from Proposition~\ref{lem:a} also imply that
\begin{align*}
 \| \P_{a_\infty} \wK_u(\Phi,a)(\tau) &- \P_{b_\infty} \wK_u(\Upsilon,b)(\tau) \| \\
&\leq \int_\tau^\infty e^{\tau-\sigma} \Big\| \P_{a_\infty} \left[ \hL_{a(\sigma)} \Phi(\sigma) + \N_{a(\sigma)}(\Phi(\sigma)) - \partial_\sigma \Psi_{a(\sigma)} \right] \\
& \qquad \qquad \quad - \P_{b_\infty} \left[ \hL_{b(\sigma)} \Upsilon(\sigma) + \N_{b(\sigma)}(\Upsilon(\sigma)) - \partial_\sigma \Psi_{b(\sigma)} \right] \Big\| \, d\sigma \\
& \lesssim \int_\tau^\infty e^{\tau-\sigma} \delta e^{-\dbeep \sigma} \left( \| \Phi - \Upsilon \|_\X + \| a - b \|_\A \right) d\sigma \\
& \lesssim \delta e^{-\dbeep \tau} \| \Phi - \Upsilon \|_\X,
\end{align*}
which, in particular, shows that
\begin{align}\label{K1}
 \| \P_{a_\infty} \wK_u(\Phi,a) - \P_{b_\infty} \wK_u(\Upsilon,b) \|_\X \lesssim \delta \| \Phi -\Upsilon \|_\X.
\end{align}

 To treat the remaining component $\wP_{a_\infty} \wK_u(\Phi,a)$ let
\begin{align}\label{Z}
 \Z{\Phi}{a}(\sigma) := \hL_{a(\sigma)} \Phi(\sigma) + \N_{a(\sigma)}(\Phi(\sigma)) - \partial_\sigma \Psi_{a(\sigma)}.
\end{align}
 It follows from \eqref{wPKu} in Step 1 that
\begin{align*}
 \| \wP_{a_\infty} \wK_u(\Phi,a)(\tau) & - \wP_{b_\infty} \wK_u(\Upsilon,b)(\tau) \| \\
&\leq \| \S_{a_\infty} (\tau) \wP_{a_\infty} u - \S_{b_\infty}(\tau) \wP_{b_\infty} u \| \\
& \quad + \int_0^\tau \left\| \wP_{a_\infty} \S_{a_\infty} (\tau - \sigma) \Z{\Phi}{a}(\sigma)
- \wP_{b_\infty} \S_{b_\infty} (\tau - \sigma) \Z{\Upsilon}{b}(\sigma) \right\| d\sigma.
\end{align*}
Since
\begin{align*}
 \| \wP_{a_\infty} \S_{a_\infty} (\tau - \sigma) & \Z{\Phi}{a}(\sigma) - \wP_{b_\infty} \S_{b_\infty} (\tau - \sigma) \Z{\Upsilon}{b}(\sigma) \| \\
 \leq \; & \| \wP_{a_\infty} \S_{a_\infty} (\tau - \sigma) \left[ \Z{\Phi}{a}(\sigma) - \Z{\Upsilon}{b}(\sigma) \right] \| \\
\; & + \|  \left[ \wP_{a_\infty} \S_{a_\infty} (\tau - \sigma) - \wP_{b_\infty} \S_{b_\infty} (\tau - \sigma) \right] \Z{\Upsilon}{b}(\sigma) \|,
\end{align*}
it follows from Proposition~\ref{prop:proj}, Lemmas~\ref{lem:decay1}, \ref{lem:decay3} and \ref{lem:decay4} that
\begin{align*}
\| \wP_{a_\infty} \wK_u(\Phi,a)(\tau) & - \wP_{b_\infty} \wK_u(\Upsilon,b)(\tau) \| \\
&\lesssim e^{-\beep \tau} \frac{\delta}{c} |a_\infty - b_\infty| + \int_0^\tau e^{-\beep (\tau - \sigma)} \delta e^{-\dbeep \sigma} \left( \| \Phi - \Upsilon \|_\X + \| a - b \|_\A \right)  \\
& \qquad \qquad \qquad \qquad \qquad \qquad + |a_\infty-b_\infty| e^{-\beep (\tau - \sigma)} \delta^2 e^{-\dbeep \sigma} d\sigma \\
&\leq e^{-\beep \tau} \frac{\delta}{c} \|a-b\|_\A + e^{-\beep \tau} \delta \left( \| \Phi - \Upsilon \|_\X + \| a - b \|_\A \right) \\
&\lesssim e^{-\beep \tau} \delta \| \Phi - \Upsilon \|_\X.
\end{align*}
Thus we have that
\begin{align}\label{K3}
 \|  \wP_{a_\infty} \wK_u(\Phi,a) - \wP_{b_\infty} \wK_u(\Upsilon,b) \|_\X \lesssim \delta \| \Phi -\Upsilon\|_\X.
\end{align}

 Hence finally it follows from \eqref{K1}--\eqref{K3} with $\delta$ sufficiently small that
\begin{align*}
 \| \wK_u(\Phi,a) - \wK_u(\Upsilon,b) \|_\X \leq \;& \| \Q_{a_\infty} \wK_u(\Phi,a) - \Q_{b_\infty} \wK_u(\Upsilon,b) \|_\X \\ 
&+ \| \P_{a_\infty} \wK_u(\Phi,a) - \P_{b_\infty} \wK_u(\Upsilon,b) \|_\X \\
&+ \|  \wP_{a_\infty} \wK_u(\Phi,a) - \wP_{b_\infty} \wK_u(\Upsilon,b) \|_\X \\
\leq \; &\tfrac{1}{2} \| \Phi -\Upsilon \|_\X,
\end{align*}
 therefore $\wK_u$ is a contraction.

\medskip
\textbf{Step 3. The solution operator is Lipschitz continuous.}
 Suppose $\Phi = \wK_u(\Phi,a)$ and $\Upsilon=\wK_v(\Upsilon,b)$ for $u,v \in B_{\delta/c} := \{ u \in \H \, | \, \| u \| \leq \frac{\delta}{c} \}$. We already know from Proposition~\ref{lem:a} that
\[
 \| a - b \|_\A \lesssim \| \Phi - \Upsilon \|_\X,
\]
thus it remains to estimate $\Phi-\Upsilon$. 

 Proposition~\ref{prop:proj} and the modulation equation~\eqref{Qmod} imply that
\begin{align*}
 \wK_u(\Phi,a) &= \S_{a_\infty}(\tau) [u-\C_u(\Phi,a)] + \int_0^\tau \S_{a_\infty}(\tau-\sigma) \Z{\Phi}{a}(\sigma) d\sigma \\
&= \Q_{a_\infty} u + e^\tau \P_{a_\infty} u + \S_{a_\infty}(\tau) \wP_{a_\infty} u - e^\tau \P_{a_\infty} u - \P_{a_\infty} \int_0^\infty e^{\tau-\sigma} \Z{\Phi}{a}(\sigma) d\sigma \\
& \quad + \Q_{a_\infty} \int_0^\tau \Z{\Phi}{a}(\sigma)d\sigma + \P_{a_\infty} \int_0^\tau e^{\tau-\sigma} \Z{\Phi}{a}(\sigma) d\sigma \\
& \quad + \wP_{a_\infty} \int_0^\tau \S_{a_\infty} (\tau-\sigma) \Z{\Phi}{a}(\sigma) d\sigma \\
&= \chi(\tau) \Q_{a_\infty} u - \P_{a_\infty} \int_\tau^\infty e^{\tau-\sigma} \Z{\Phi}{a}(\sigma) d\sigma + \S_{a_\infty}(\tau) \wP_{a_\infty} u \\
& \quad + \wP_{a_\infty} \int_0^\tau \S_{a_\infty}(\tau-\sigma) \Z{\Phi}{a}(\sigma) d\sigma. 
\end{align*}
 Thus we have to estimate the terms arising in 
\begin{equation}\label{eq:Kusum}
\begin{aligned}
 \| \wK_u&(\Phi,a) - \wK_v(\Upsilon,b) \| \\
&\leq |\chi(\tau)| \|\Q_{a_\infty} u - \Q_{b_\infty} v \| + \| \S_{a_\infty}(\tau) \wP_{a_\infty} u - \S_{b_\infty}(\tau) \wP_{b_\infty} v \|  \\
&\quad + \left\| \P_{a_\infty} \int_\tau^\infty e^{\tau-\sigma} \Z{\Phi}{a}(\sigma) d\sigma - \P_{b_\infty} \int_\tau^\infty e^{\tau-\sigma} \Z{\Upsilon}{b}(\sigma) d\sigma \right\| \\
&\quad + \left\| \wP_{a_\infty} \int_0^\tau \S_{a_\infty} (\tau-\sigma) \Z{\Phi}{a}(\sigma) d\sigma - \wP_{b_\infty} \int_0^\tau \S_{b_\infty}(\tau-\sigma) \Z{\Upsilon}{b}(\sigma) d\sigma \right\|.
\end{aligned}
\end{equation}
 The estimate of the first terms follows from the definition of $\chi$, Lemma~\ref{lem:decay1} and Proposition~\ref{lem:a},
\begin{align}\label{term1}
 |\chi(\tau)| \| \Q_{a_\infty} u - \Q_{b_\infty} v \| &\lesssim e^{-\beep \tau} \left( \| u - v \| + |a_\infty - b_\infty| \| v \| \right) \nonumber \\
&\leq e^{-\beep \tau} \left( \| u - v \| + \|a-b\|_\A \frac{\delta}{c} \right) \nonumber \\
&\lesssim e^{-\beep \tau} \left( \| u - v \| + \delta \| \Phi-\Upsilon \|_\X \right)
\end{align}
 In a similar fashion we obtain
\begin{align}\label{term2}
 \| \S_{a_\infty}(\tau) \wP_{a_\infty} u - \S_{b_\infty}(\tau) \wP_{b_\infty} v \|
&\leq \| \S_{a_\infty} \wP_{a_\infty} (u-v) \| + \| ( \S_{a_\infty}(\tau) \wP_{a_\infty} - \S_{b_\infty}(\tau) \wP_{b_\infty} ) v \| \nonumber \\
&\lesssim e^{-\beep \tau} \| u- v\| + |a_\infty-b_\infty| e^{-\beep \tau} \|v\| \nonumber \\
&\lesssim e^{-\beep \tau} ( \| u- v\| + \delta \| \Phi-\Upsilon \|_\X ),
\end{align}
 where we additionally used Proposition~\ref{prop:proj} in the first line.
 From the Lipschitz regularity of $\P_{a_\infty}$ from Lemma~\ref{lem:decay1} and the estimates for the integrand terms from Lemmas~\ref{lem:decay3} and \ref{lem:decay4} it follows with Proposition~\ref{lem:a} that the third term in \eqref{eq:Kusum} is
\begin{align}\label{term3}
 \Big\| \P_{a_\infty} & \int_\tau^\infty e^{\tau-\sigma} \Z{\Phi}{a}(\sigma) d\sigma - \P_{b_\infty} \int_\tau^\infty e^{\tau-\sigma} \Z{\Upsilon}{b}(\sigma) d\sigma \Big\| \nonumber \\
&\leq \int_\tau^\infty e^{\tau-\sigma} \Big[ \| \hat \L_{a(\sigma)} \Phi(\sigma) - \hat \L_{b(\sigma)} \Upsilon(\sigma) \| + |a_\infty-b_\infty| \| \hat \L_{b(\sigma)} \Upsilon(\sigma) \| \nonumber \\
& \qquad\qquad\qquad + \| \N_{a(\sigma)}(\Phi(\sigma)) - \N_{b(\sigma)} (\Upsilon(\sigma)) \| + |a_\infty-b_\infty| \| \N_{b(\sigma)}(\Upsilon(\sigma)) \| \nonumber \\
& \qquad\qquad\qquad + \| \P_{a_\infty} \partial_\sigma \Psi_{a(\sigma)} - \P_{b_\infty} \partial_\sigma \Psi_{b(\sigma)} \| \Big] d\sigma \nonumber \\
&\lesssim \int_\tau^\infty e^{\tau-\sigma} \left[ \delta e^{-\dbeep \sigma} (\|\Phi-\Upsilon\|_\X + \|a-b\|_\A) + \|a-b\|_\A \delta^2 e^{-\dbeep\sigma} \right] d\sigma \nonumber \\
&\lesssim \delta e^\tau \|  \Phi -\Upsilon \|_\X \int_\tau^\infty e^{-\sigma} e^{-\dbeep \sigma} d\sigma \nonumber \\
&\lesssim \delta e^{-\dbeep \tau} \| \Phi - \Upsilon \|_\X.
\end{align}
 Finally, the last term in \eqref{eq:Kusum} can be controlled by the semigroup estimates of Proposition~\ref{prop:proj} and the Lipschitz regularity in Lemmas~\ref{lem:decay1}, \ref{lem:decay3} and \ref{lem:decay4}. More precisely,
\begin{align}\label{term4}
 \Big\| \wP_{a_\infty} & \int_0^\tau \S_{a_\infty} (\tau-\sigma) \Z{\Phi}{a}(\sigma) d\sigma - \wP_{b_\infty} \int_0^\tau \S_{b_\infty}(\tau-\sigma) \Z{\Upsilon}{b}(\sigma) d\sigma \Big\| \nonumber \\
&\leq \int_0^\tau \Big\| \S_{a_\infty}(\tau-\sigma) \wP_{a_\infty} \left[ \hat\L_{a(\sigma)}\Phi(\sigma) - \hat\L_{b(\sigma)}\Upsilon(\sigma) + \N_{a(\sigma)}(\Phi(\sigma)) - \N_{b(\sigma)}(\Upsilon(\sigma)) \right] \Big\| \nonumber \\
&\qquad\quad +  \| \S_{a_\infty} (\tau-\sigma) \wP_{a_\infty} - \S_{b_\infty} (\tau-\sigma) \wP_{b_\infty} \| \| \hat \L_{b(\sigma)}\Upsilon(\sigma) + \N_{b(\sigma)}(\Upsilon(\sigma)) \| \nonumber \\
&\qquad\quad + \| \S_{a_\infty} (\tau-\sigma) \wP_{a_\infty} \partial_\sigma \Psi_{a(\sigma)} - \S_{b_\infty} (\tau-\sigma) \wP_{b_\infty} \partial_\sigma \Psi_{b(\sigma)} \| d\sigma \nonumber \\
&\lesssim \int_0^\tau e^{-\beep (\tau-\sigma)} \delta e^{-\dbeep \sigma} \| \Phi - \Upsilon \|_\X + e^{-\beep (\tau-\sigma)} \|\Phi-\Upsilon\|_\X \delta^2 e^{-\dbeep \sigma} \nonumber \\
&\quad\quad + e^{-\beep (\tau-\sigma)} \delta e^{-\dbeep \sigma} \| \Phi-\Upsilon\|_\X d\sigma \nonumber \\
&\lesssim \delta e^{-\beep \tau} \| \Phi-\Upsilon \|_\X,
\end{align}
 where we have used 
\begin{align*}
 \| \S_{a_\infty} (\tau-\sigma) & \wP_{a_\infty} \partial_\sigma \Psi_{a(\sigma)} - \S_{b_\infty} (\tau-\sigma) \wP_{b_\infty} \partial_\sigma \Psi_{b(\sigma)} \| \\
&\leq \| \S_{a_\infty}(\tau-\sigma) \wP_{a_\infty} - \S_{b_\infty} (\tau-\sigma) \wP_{b_\infty} \| \| \wP_{a_\infty} \partial_\sigma \Psi_{a(\sigma)} \| \\
&\quad + \| \S_{b_\infty}(\tau-\sigma) \wP_{b_\infty} \| \| \wP_{a_\infty} \partial_\sigma \Psi_{a(\sigma)} - \wP_{b_\infty} \partial_\sigma \Psi_{b(\sigma)} \| \\
&\lesssim |a_\infty - b_\infty| e^{-\beep(\tau-\sigma)} \delta^2 e^{-\dbeep \sigma} + e^{-\beep (\tau-\sigma)} \delta e^{-\dbeep \sigma} \|a-b\|_\A \\
&\lesssim e^{-\beep \tau} e^{-\beep \sigma} \delta \| \Phi-\Upsilon\|_\X.
\end{align*}
 Together \eqref{eq:Kusum}--\eqref{term4} yield
\begin{align*}
 \| \Phi - \Upsilon \|_\X
&= \sup_{\tau > 0} \left( e^{\beep\tau} \| \Phi(\tau) - \Upsilon(\tau) \| \right) \lesssim \delta \| \Phi - \Upsilon \|_\X + \| u - v \|,
\end{align*}
so that we obtain the desired Lipschitz regularity for $\delta$ sufficiently small, i.e.,
\[
 \| a - b \|_\A + \| \Phi - \Upsilon \|_\X \lesssim \| u - v \|. \qedhere
\]
\end{proof}

\subsection{The subspace of stable and center initial data}
\label{ssec4.4}

 In Proposition~\ref{prop:modsol} of the above Section~\ref{ssec4.3} we solved the modified equation $\Phi(\tau) = \wK_u(\Phi,a)(\tau)$ by isolating the instability related to time translations. Using this existence and uniqueness result we are now in a position to identify stable initial data $u$ that lead to unique global solutions $\Phi_u$ of the \emph{original\/} problem $\Phi(\tau) = \K_u(\Phi,a)(\tau)$ stated explicitly in equation \eqref{intmodeq}. As a result, we obtain the set of initial data near the selfsimilar solution $\psi_0 \equiv \sqrt{2}$ suitable for Theorem~\ref{maintheorem}, that is, a stable codimension-$1$ Lipschitz manifold of hyperboloidal initial data leading to unique global solutions of the cubic wave equation \eqref{cwe} with nondispersive decay.

\medskip
 To see how the modification \eqref{Peq} of the original equation is related to a modification of the initial data, we first compare the modification of the right hand side $\K_u(\Phi,a)$ of the original equation \eqref{intmodeq} to the right hand side $\wK_u(\Phi,a)$ of the the modified equation \eqref{Peq}. This difference is made up of the correction term defined in \eqref{C}, i.e.,
\[
 \wK_u(\Phi,a) - \K_u(\Phi,a) = \S_{a_\infty}(\tau) \C_u(\Phi,a).
\]
 Suppose $\Phi_u,a_u$ denotes the unique solution of \eqref{Peq} corresponding to the initial value $\|u\| \leq \frac{\delta}{c}$ as shown in Proposition~\ref{prop:modsol}. By introducing the function
\begin{align}\label{F}
 \F(u) := \P_{a_{u,\infty}} \int_0^\infty e^{-\sigma} \left[ \hL_{a_u(\sigma)} \Phi_u(\sigma) + \N_{a_u(\sigma)}(\Phi_u(\sigma)) - \partial_\sigma \Psi_{a_u(\sigma)} \right] d\sigma,
\end{align}
we can write the correction term in the form
\[
 \C(u) := \C_u(\Phi_u,a_u) = \P_{a_{u,\infty}} u + \F(u).
\]
 The correction term is defined by via the projection $\P_{a_\infty}$ onto the unstable subspace $\rg \P_{a_\infty}$, thus the modification in the equation \eqref{Peq} should be viewed as adding a term from the image of a rank-1 projection to the initial data. More precisely, we obtain the following manifold.

\begin{theorem}\label{thm:manifold}
 Fix $\varepsilon \in (0,\frac{1}{2})$. Let $\delta >0$ be sufficiently small and $\tilde c > c>0$ sufficiently large. There exists a codimension-1 Lipschitz manifold $\M \subseteq \H$ with $0 \in \M$, defined as the graph of a Lipschitz function $\T \colon \ker \P_0 \to \rg \P_0$, that is,
\begin{align*}
 \M := \; & \{ v + \T(v) \, | \, v \in \ker \P_0, \| v \| \leq \tfrac{\delta}{\tilde c} \} \\
\subseteq \; & \{ u \, | \, \C(u) = 0, \| u \| \leq \tfrac{\delta}{c} \} \subseteq \ker \P_0 \oplus \rg \P_0 = \H,
\end{align*}
 such that for any $u \in \M$ there exists a unique solution $\Phi_u \in \Xd$ and $a_u \in \Ad$ to equation \eqref{intmodeq}, that is to
\[
 \Phi(\tau) =  \S_{a_\infty}(\tau) u
 - \int_0^\tau \S_{a_\infty}(\tau-\sigma) \left[ \hL_{a(\sigma)} \Phi(\sigma) + \N_{a(\sigma)}(\Phi(\sigma)) - \partial_\sigma \Psi_{a(\sigma)} \right] d\sigma, 
\]
 with initial value $\Phi_u (0) = u$. Moreover, the sum
\[ \Psi_u (\tau) :=  \Psi_{a_u(\tau)} + \Phi_u (\tau), \qquad \Psi_u \in \cC ([0,\infty),\H) \]
 with initial value $\Psi_u(0) = \Psi_0 + u$ is a solution in the Duhamel sense of the abstract evolution equation 
\begin{align}\label{evoeq1}
 \partial_\tau \Psi(\tau) = \L \Psi(\tau) + \N (\Psi(\tau)), \quad \tau \geq 0,
\end{align}
 formulated in Section~\ref{sec2}.
\end{theorem}

\begin{proof}
 \textbf{Step 1. $\C(u) = 0$ is equivalent to $\P_0 \C(u) = 0$.}
 One direction is just the restriction to $\rg \P_0$ and obviously satisfied. On the other hand, assume that $\P_0 \C(u) = 0$ and recall that $\C(u) \in \rg \P_{a_{u,\infty}}$ by definition. Lemma~\ref{lem:decay1} implies
\begin{align*}
 \| \C(u) \| &= \| \P_0 \C(u) +  \left( \P_{a_{u,\infty}} - \P_0 \right) \C(u) \| \\
&\leq \| \P_{a_{u,\infty}} - \P_0 \| \| \C(u) \| \lesssim | a_{u,\infty} | \| \C(u) \|.
\end{align*}
 By Remark~\ref{rem:a}, $| a_{u,\infty} | \lesssim \delta$ can be made arbitrarily small (for $\delta$ sufficiently small), which shows that $\C(u) = 0$.

\medskip
 \textbf{Step 2. $\C(u)$ as graph of $\T$ on $\ker \P_0$.}
 Since $\P_0$ is a projection, we can write $\H$ as a direct sum $\H = \ker \P_0 \oplus \rg \P_0$. Given that $u = v + w \in \ker \P_0 \oplus \rg \P_0$ should satisfy $\C(u)=0$, we define for every $v \in \ker \P_0$ the map
\begin{align*}
 \wC_v \colon \rg \P_0 &\to \rg \P_0 \\
  w &\mapsto \wC_v (w) := \P_0 \C(v+w).
\end{align*}
 and want to find a unique $w \in \rg \P_0$ such that $\wC_v(w) = 0$. This is equivalent to the inverse (if it exists) satisfying $\wC_v^{-1}(0) = w$. The invertibility is established in the next step, so that we can define
\begin{align*}
 \T \colon \ker \P_0 &\to \rg \P_0 \\
v &\mapsto \T(v) := \wC_v^{-1}(0).
\end{align*}

\medskip
 \textbf{Step 3. $\wC_v$ is invertible for small $v \in \ker \P_0$.}
 To be able to apply the Banach fixed point theorem, we rewrite the equation $\wC_v(w) = 0$ in a suitable way as an implicit equation $\Gamma_v(w) = w$. To this end note that for $v + w \in \ker \P_0 \oplus \rg \P_0$ we generally have
\begin{align*}
 \wC_v(w) = \; & \P_0 \C(v+w) \\
= \; & \P_0 \P_{a_{v+w,\infty}} (v+w) + \P_0 \F(v+w) \\
= \; & \P_0 \P_{a_{v+w,\infty}} w + \P_0 \P_{a_{v+w,\infty}} v + \P_0 \F(v+w) \\
= \; & w + \P_0 (\P_{a_{v+w,\infty}} - \P_0) (v+w) + \P_0 \F(v+w). 
\end{align*}
 Equivalent to showing that $\wC_v(w) = 0$ for a unique $w \in \rg\P_0$ is to solve the implicit equation
\[
 w = \Gamma_v(w) := \P_0 (\P_0 - \P_{a_{v+w,\infty}}) (v+w) - \P_0 \F(v+w).
\]
 We prove that $\Gamma_v$ is a contraction for every $v \in \ker \P_0$ that is sufficiently small.

 First we observe that with $\Gamma_v$, for $\| v \| \leq \frac{\delta}{2c}$, we land in the same subset 
\[ \widetilde B_{\delta/c}(v) := \{ w \in \rg \P_0 \, | \, \| v+w \| \leq \tfrac{\delta}{c} \}. \]
 For $u = v+w$ with $\| u \| \leq \frac{\delta}{c}$ there exist unique $\Phi_{v+w} \in \Xd,$ $a_{v+w} \in \Ad$ that solve equation~\eqref{Peq} by Proposition~\ref{prop:modsol}. Thus it follows from Lemmas~\ref{lem:decay1} and \ref{lem:decay3} that
\begin{align*}
 \| \Gamma_v(w) \| &\leq \| \P_{a_{v+w,\infty}} - \P_0 \| \| v + w \| + \| \F(v+w) \| \\
&\lesssim | a_{v+w,\infty} | \frac{\delta}{c} + \int_0^\infty e^{-\sigma} \delta^2 e^{-\dbeep\sigma} d\sigma \\
&\lesssim \frac{\delta}{\beep} \frac{\delta}{c} + \delta^2
\end{align*}
 where the last inequality follows again from the fact that $|a_{v+w,\infty}| \leq \frac{\delta}{\eep}$ by the proof of Lemma~\ref{lem:decay3}. Thus for sufficiently small $\delta$ and $\| v \| \leq \frac{\delta}{2c}$ we observe that $\| v + \Gamma_v(w)\| \leq \frac{\delta}{c}$, hence $\Gamma_v(w) \in \widetilde B_{\delta/c}(v)$.

 It remains to be shown that $\Gamma_v$ is a contraction for every $v \in \ker \P_0$ with $\| v \| \leq \frac{\delta}{2c}$. For $w,\widetilde w \in \widetilde B_{\delta/c}(v)$ we obtain by Proposition~\ref{prop:modsol} corresponding solutions $(\Phi,a)$ and $(\Upsilon,b)$ in $\Xd \times \Ad$ to the modified equation \eqref{Peq} for $v+w$ and $v+\widetilde w$, respectively. This yields
\begin{align}\label{eq:G1}
 \| \Gamma_v(w) - \Gamma_v(\widetilde{w}) \|
\leq \; & \| \P_0 (\P_0 - \P_{a_\infty})(v+w) - \P_0 (\P_0 - \P_{b_\infty} ) (v+\widetilde w) \| \nonumber\\
& + \| \P_0 \F(v+w) - \P_0 \F(v+\widetilde w) \|.
\end{align}
 Since $v \in \ker \P_0$ the first term can be simplified, i.e.,
\begin{align*}
 \P_0 (\P_0 - \P_{a_\infty})(v+w) &- \P_0 (\P_0 - \P_{b_\infty} ) (v+\widetilde w) \\
&= \P_0 \P_0 (w-\widetilde w) - \P_0 \left[ \P_{a_\infty} (v+w) - \P_{b_\infty} (v+\widetilde w) \right] \\
&= \P_0 \P_0 (w-\widetilde w) - \P_0 (\P_{a_\infty} - \P_{b_\infty}) (v+\widetilde w) - \P_0 \P_{a_\infty} (w-\widetilde w) \\
&= \P_0 (\P_0 - \P_{a_\infty}) (w-\widetilde w) - \P_0 (\P_{a_\infty} - \P_{b_\infty}) (v+\widetilde w),
\end{align*}
 and estimated using Lemma~\ref{lem:decay1}, Remark~\ref{rem:a} and the Lipschitz continuous dependence of the solutions, such that
\begin{align}\label{eq:G2}
 \|  \P_0 (\P_0 - \P_{a_\infty})(v+w) &- \P_0 (\P_0 - \P_{b_\infty} ) (v+\widetilde w) \| \nonumber \\
&\leq \| \P_0 (\P_0 - \P_{a_\infty}) (w-\widetilde w) \| + \| \P_0 (\P_{a_\infty} - \P_{b_\infty}) (v+\widetilde w) \| \nonumber\\
&\lesssim |a_\infty| \| w-\widetilde w \| + |a_\infty - b_\infty| \| v+\widetilde w\| \nonumber\\
&\lesssim \delta \| w- \widetilde w \| + \tfrac{\delta}{c} \| a - b \|_\A \nonumber\\
&\lesssim \delta \| w-\widetilde w\|.
\end{align}
 The estimate for the second term in \eqref{eq:G1} follows directly from Lemma~\ref{lem:decay4}, the Lipschitz bounds obtained in Lemmas~\ref{lem:decay1}, \ref{lem:decay3} and \ref{lem:decay4}, together with the Lipschitz bound of the solution from Proposition~\ref{prop:modsol}, i.e.,
\begin{align}\label{eq:G3}
 \| \P_0 \F (v+w) - \P_0 \F (v+\widetilde w) \|
&\lesssim \int_0^\infty e^{-\sigma} \Big[ \| \hat \L_{a(\sigma)} \Phi(\sigma) - \hat \L_{b(\sigma)} \Upsilon(\sigma) \| \nonumber \\
&\qquad\qquad\quad + \| \N_{a(\sigma)}(\Phi(\sigma)) - \N_{b(\sigma)}(\Upsilon(\sigma)) \| \nonumber \\
&\qquad\qquad\quad + |a_\infty-b_\infty| (\|\hat \L_{b(\sigma)} \Upsilon(\sigma)\| + \| \N_{b(\sigma)}(\Upsilon(\sigma)) \|) \nonumber \\
&\qquad\qquad\quad + \| \P_{a_\infty} \partial_\sigma \Psi_{a(\sigma)} - \P_{b_\infty} \partial_\sigma \Psi_{b(\sigma)} \| \Big] d\sigma \nonumber \\
&\lesssim \int_0^\infty e^{-\sigma} \delta e^{-\dbeep \sigma} \| w-\widetilde w\| d\sigma \nonumber \\
&\lesssim \delta \| w-\widetilde w\|.
\end{align}
 Together \eqref{eq:G1}--\eqref{eq:G3} yield
\[
 \| \Gamma_v(w) - \Gamma_v (\widetilde w) \| \lesssim \delta \| w - \widetilde w \|,
\]
 hence $\Gamma_v$ is a contraction for $\delta$ sufficiently small. Note that the constants contained in $\lesssim$ do \emph{not\/} depend on the specific $u = v+w$, since the bounds come from general results. In particular, $\delta>0$ can be chosen sufficiently small \emph{independent\/} of and therefore uniformly in $v$. Thus for each $v \in \ker \P_0$ with $\| v \| \leq \frac{\delta}{2c}$ the Banach fixed point theorem implies the existence and uniqueness of $w_v \in B_{\delta/c}(v) \subseteq \rg \P_0$ that satisfies
\begin{align}\label{eq:Cuzero}
 \C(v+w) = \wC_v(w) = 0.
\end{align}

\medskip
 \textbf{Step 4. $\T$ is Lipschitz continuous.}
 In order to prove Lipschitz regularity of the graph function $\T$ recall that it has been defined as the inverse function of $\wC$ at $0$, which has been rewritten as an implicit function with right hand side $\Gamma$. By the triangle inequality and for sufficiently small $\delta$ we know from Step 3 that for $v, \widetilde v \in \ker \P_0$ with $\|v\|,\|\widetilde v\| \leq \frac{\delta}{2c}$, the unique solutions $w \in \widetilde B_{\delta/c}(v), \widetilde w \in \widetilde B_{\delta/c}(\widetilde v) \subseteq \rg \P_0$ to equation~\eqref{eq:Cuzero}, respectively, satisfy
\begin{align}
 \| \T(v) - \T(\widetilde v) \| &= \| \wC_v^{-1} (0) - \wC_{\widetilde v}^{-1} (0) \| = \| w - \widetilde w \| \label{T1} \\
&= \| \Gamma_v(w) - \Gamma_{\widetilde v} (\widetilde w) \| \nonumber \\
&\leq \| \Gamma_v(w) - \Gamma_v(\widetilde w) \| + \| \Gamma_v(\widetilde w)  - \Gamma_{\widetilde v} (\widetilde w)\| \nonumber \\
&\leq \tfrac{1}{2} \| w - \widetilde w \| + \| \Gamma_v(\widetilde w)  - \Gamma_{\widetilde v} (\widetilde w)\| \label{T2}
\end{align}
 The second term in the last inequality is thus crucial and reads
\begin{align*}
 \Gamma_v(\widetilde w) - \Gamma_{\widetilde v}(\widetilde w)
&= \P_0 (\P_0 - \P_{a_{v+\widetilde w,\infty}})(v+\widetilde w) - \P_0 (\P_0 - \P_{a_{\widetilde v+\widetilde w,\infty}})(\widetilde v+\widetilde w) \\ & \quad - \P_0 \F(v+\widetilde w) + \P_0 \F(\widetilde v + \widetilde w)\\
&= - \P_0 \P_{a_{v+\widetilde w}} (v +\widetilde w) + \P_0 \P_{a_{\widetilde v+\widetilde w}} (\widetilde v +\widetilde w) - \P_0 \F(v+\widetilde w) + \P_0 \F(\widetilde v + \widetilde w) \\
&= \P_0 (\P_{a_{\widetilde v + \widetilde w,\infty}} - \P_{a_{v+\widetilde w,\infty}}) \widetilde w - \P_0 (\P_{a_{v+\widetilde w,\infty}} v - \P_{a_{\widetilde v + \widetilde w,\infty}} \widetilde v) \\ & \quad- \P_0 \F(v+\widetilde w) + \P_0 \F(\widetilde v + \widetilde w).
\end{align*}
 The norm of the difference $\P_0 (\P_{a_{\widetilde v + \widetilde w,\infty}} - \P_{a_{v+\widetilde w,\infty}}) \widetilde w - \P_0 (\P_{a_{v+\widetilde w,\infty}} v - \P_{a_{\widetilde v + \widetilde w,\infty}} \widetilde v)$ can be estimated using Lemma~\ref{lem:decay1} and the Lipschitz regularity of the solution as shown in Proposition~\ref{prop:modsol}, which yields
\begin{align*}
\| \P_{a_{\widetilde v + \widetilde w,\infty}} &- \P_{a_{v+\widetilde w,\infty}} \| \| \widetilde w \|
 + \| \P_{a_{v + \widetilde w,\infty}} v - \P_{a_{v+\widetilde w,\infty}} \widetilde v \| + \| \P_{a_{v+\widetilde w,\infty}} \widetilde v - \P_{a_{\widetilde v + \widetilde w,\infty}} \widetilde v \| \nonumber \\
&\lesssim | a_{\widetilde v + \widetilde w,\infty} - a_{v+\widetilde w,\infty} | \| \widetilde w \| + \| v - \widetilde v \| + |a_{v+\widetilde w,\infty} - a_{\widetilde v + \widetilde w,\infty} | \| \widetilde v \| \nonumber \\
&\lesssim \frac{\delta}{c} \| v - \widetilde v \| + \| v - \widetilde v \| + \frac{\delta}{2c} \| v - \widetilde v \| \nonumber \\
&\lesssim \| v - \widetilde v \|. \nonumber
\end{align*}
 The remaining terms have already been estimated in \eqref{eq:G3} of Step 3, i.e.,
\[
 \| \P_0 \F(v+\widetilde w) - \P_0 \F(\widetilde v + \widetilde w) \| \lesssim \delta \| v - \widetilde v \|,
\]
 so that
\begin{align}\label{T3}
 \| \Gamma_v(\widetilde w) &- \Gamma_{\widetilde v}(\widetilde w) \| \lesssim \| v - \widetilde v \|.
\end{align}
 Together, \eqref{T1}--\eqref{T3} imply the desired Lipschitz continuity of $\T$, i.e.,
\[
 \| \T(v) - \T(\widetilde v) \| = \| w -\widetilde w \| \leq 2 \| \Gamma_v(\widetilde w) - \Gamma_{\widetilde v}(\widetilde w) \|  \lesssim \| v - \widetilde v \|.
\]
 This proves that $\M$ is a Lipschitz manifold of codimension-1 that consists of initial data $u$ for which unique global solutions $(\Phi_u,a_u) \in \Xd \times \Ad$ to the original integral equation~\eqref{intmodeq} exist.

 Note that
\[
 \wC_0(0) = \P_0 \C(0) = \F(0) = 0,
\]
 since $\Phi_u = 0$ and $a_u = 0$ for $u=0$ by Proposition~\ref{prop:modsol}.
 In particular, since $w \in \rg \P_0$ with $\wC_0(w)=0$ is unique by the above,
\[
 \T(0)=0.
\]
 Thus $\| v + \T(v) \| \leq \frac{\delta}{c}$ is small for $\| v \|$ small. The assumption that $\T(v) \in \widetilde B_{\delta/c}(v)$ is therefore not necessary if we assume that $\|v\|\leq \frac{\delta}{\tilde c}$ for a sufficiently large(r) constant $\tilde c$ than $2c$.

\medskip
\textbf{Step 5. Solution to the evolution equation \eqref{evoeq1}.}
The final statement on $\Psi$ follows directly from the modulation ansatz of Section~\ref{ssec:mod} and the fact that $a(0) = 0$ has been assumed in the Definition~\ref{def:XA} of $\A$.
\end{proof}

\subsection{Evolution for initial data from codimension-1 manifold: Proof of Theorem~\ref{maintheorem}}
\label{ssec4.5}

We conclude by showing the decay for solutions to the cubic wave equation with initial data in the codimension-1 manifold $\M$ defined in Theorem~\ref{thm:manifold} is as stated in Theorem~\ref{maintheorem}.

\begin{proof}[Proof of Theorem~\ref{maintheorem}]
 In Sections~\ref{sec1} and \ref{sec2} the focusing cubic wave equation~\eqref{cwe} in $v$ has been rewritten as the equivalent evolution equation~\eqref{evoeq1} in $\Psi$. Thus Theorem~\ref{thm:manifold} implies that hyperboloidal initial data from a codimension-1 manifold $\M$ of $\H$ lead to unique global solutions $v \in \cC([0,\infty),\H)$ of \eqref{cwe} for given initial data $(f,g) \in \M$. Moreover, to each solution a unique $a = a_{(f,g)} \in \Ad$ with limit $a_\infty = \lim_{t\to\infty}a(\tau) \in \RR^3$ is associated. The remaining part of the proof is concerned with estimating the difference $v-v_{a_\infty}$, where $v_{a_\infty}$ is the Lorentz transformed selfsimilar solution $v_0(t,x) = \sqrt{2}/t$ with rapidity $a_\infty \in \RR^3$, as defined in \eqref{va} of Section~\ref{ssec:lorentz}. (Note that $a$ in the statement of the theorem is the limit $a_\infty \in \RR^3$ and not the function $a \in \Ad$ used in the proof below.)

\medskip
\textbf{Step 1. $v-v_{a_\infty}$ in energy norm.}
 Recall that the Hilbert space $\H$ is equivalent to $H^1(\BB) \times L^2 (\BB)$ as a Banach space \cite{DZ}[Prop.\ 4.1], which is why it suffices to control the norms $\| v - v_{a_\infty} \|_{H^1(\Sigma_T)}$ and $\| \nabla_n v - \nabla_n v_{a_\infty} \|_{L^2(\Sigma_T)}$, introduced in Definition~\ref{def:norms}, as $T \to 0$. In order to proceed we have to recall the modulation ansatz and the coordinate transforms. The $H^1$-norm of $v-v_{a_\infty}$ is obtained by integrating
\begin{align}\label{M1}
 u(T,X)-u_{a_\infty}(T,X) = \frac{v \circ \kappa_T(X) - v_{a_\infty} \circ \kappa_T(X)}{T^2 - |X|^2}
\end{align}
 and its spatial derivative on the hyperboloidal slice $\Sigma_T$ ($u$ solves the same cubic wave equation by Lemma~\ref{lem:transform1}). In Lemma~\ref{lem:transform2} and \eqref{cwe-Psi} this problem was translated to an equation for $\psi$, where
\begin{align}\label{M2}
 u(T,X) = \frac{1}{-T} \psi\left(-\log(-T),\frac{X}{-T}\right),
\end{align}
 and subsequently transformed to a first-order system by the modulation ansatz in Section~\ref{ssec:mod} and solved in Theorem~\ref{thm:manifold} above. The first component of the perturbation term $\Phi (\tau) = \Psi(\tau) - \Psi_{a(\tau)}$ reflects the difference of the solutions $\psi$ to a Lorentz transform $\psi_a$ of the fundamental selfsimilar solution $\psi_0(\xi) = \sqrt{2}$, i.e.,
\begin{align}\label{diff1}
 (\psi-\psi_{a(-\log(-T))})\left(-\log(-T),\frac{X}{-T}\right) = \phi_1\left(-\log(-T),\frac{X}{-T}\right).
\end{align}
 By Theorem~\ref{thm:manifold} the solution to an initial value from $\M$, for $\varepsilon \in (0,\frac{1}{2})$, satisfies $\Phi \in \mathcal{X}^\varepsilon_{\tilde \delta}$, $a \in \mathcal{A}^\varepsilon_{\tilde \delta}$ (for sufficiently small $\tilde \delta$). Thus by Definition~\ref{def:norms}, Definition~\ref{def:XA} of the spaces $\X \subseteq \H$ and $\A$ and Remark~\ref{rem:a} we obtain
\begin{align}\label{H1:2}
 |T|^{-1} \| v - v_{a_\infty} \|_{L^2(\Sigma_T)} &= |T|^{-2} \left\| (\psi-\psi_{a_\infty})\left(-\log(-T), \frac{.}{|T|}\right) \right\|_{L^2(B_{|T|})} \nonumber \\
&\leq |T|^{-\frac{1}{2}} \left( \left\| \phi_1 (-\log (-T), .) \right\|_{L^2(\BB)} + \| \psi_{a(-\log(-T))} - \psi_{a_\infty} \|_{L^2(\BB)} \right) \nonumber\\
&\lesssim |T|^{-\frac{1}{2}} \left( \| \Phi(-\log(-T)) \|
+ |a(-\log(-T)) - a_\infty| \right) \nonumber \\
&\lesssim |T|^{-\frac{1}{2}} |T|^{\eep} = |T|^{-\varepsilon}.
\end{align}
 The second term is obtained in a similar fashion from the derivative of \eqref{diff1}, i.e.,
\begin{align*}
  \partial_{X^j} ( \psi - \psi_{a(-\log(-T))} ) \left( -\log(-T),\frac{X}{-T} \right)
&= \partial_{X^j} \phi_1\left(-\log(-T),\frac{X}{-T}\right) \\
&= \frac{1}{-T} \partial_j \phi_1\left(-\log(-T),\frac{X}{-T}\right).
\end{align*}
 This yields the $\dot H^1(\Sigma_T)$-norm
\begin{align}\label{H1:1}
 \| v - v_{a_\infty} \|_{\dot H^1(\Sigma_T)} &= |T|^{-1} \left\| \nabla (\psi - \psi_{a_\infty}) \left(-\log(-T),\frac{.}{-T} \right) \right\|_{L^2(B_{|T|})} \nonumber \\
&\leq |T|^{-\frac{1}{2}} \left( \| \nabla \phi_1 (-\log(-T),.) \|_{L^2(\BB)} + \|\nabla (\psi_{a(-\log(-T))} - \psi_{a_\infty}) \|_{L^2(\BB)} \right) \nonumber \\
&\lesssim |T|^{-\frac{1}{2}} \left( \| \Phi (-\log(-T)) \| + |a(-\log(-T))-a_\infty| \right) \nonumber \\
&\lesssim |T|^{-\frac{1}{2}} |T|^{\eep} = |T|^{-\varepsilon},
\end{align}
 again by the definition of $\|.\|$ in \eqref{H-norm} and the definitions of $\X$ and $\A$.
 Together, \eqref{H1:2} and \eqref{H1:1} make up the $H^1$-norm on the leaf $\Sigma_T$ (see Definition~\ref{def:norms}),
\begin{align}\label{eest1}
 \| v - v_{a_\infty} \|_{H^1(\Sigma_T)} &= \sqrt{\| v - v_{a_\infty} \|^2_{\dot H^1(\Sigma_T)} + T^{-2}\| v - v_{a_\infty} \|^2_{L^2(\Sigma_T)}} \lesssim |T|^{-\varepsilon}.
\end{align}

 It remains to estimate the difference $\nabla_n v - \nabla_n v_{a_\infty}$ of the normal derivatives introduced in Definition~\ref{def:nabla_n}. The $T$-derivative can be rewritten using the expression \eqref{eq:dT} in the proof of Lemma~\ref{lem:transform2}, which makes up the second component of $\Psi$ (see \eqref{Psi}) in the formulation of the evolution problem as a first-order system \eqref{cwe-Psi}. The modulation ansatz of Section~\ref{ssec:mod} thus allows us to estimate the difference in normal derivatives in terms of the second component of the perturbation term $\Phi$ since
\begin{align*}
& \frac{\nabla_n v \circ \kappa_T(X) - \nabla_n v_{a(-\log(-T))} \circ \kappa_T(X)}{T^2-|X|^2} \\
&\qquad\qquad = \partial_T u(T,X) - \partial_T u_{a(-\log(-T))}(T,X) \\
&\qquad\qquad = T^{-2} \left(1 + \partial_0 + \frac{X^j}{-T} \partial_j \right)(\psi - \psi_{a(-\log(-T))}) \left(-\log(-T),\frac{X}{-T}\right) \\
&\qquad\qquad = T^{-2} \phi_2 \left(-\log(-T),\frac{X}{-T}\right).
\end{align*}
For the $L^2$-norm on $\Sigma_T$ we once more obtain the asymptotic estimate
\begin{align}\label{eest2}
 \| \nabla_n v - \nabla_n v_{a_\infty} \|_{L^2(\Sigma_T)}
&= T^{-2} \left\| (\psi_2-\psi_{a_\infty,2}) \left(-\log(-T),\frac{.}{-T}\right) \right\|_{L^2(B_{|T|})} \nonumber \\
&\leq |T|^{-\frac{1}{2}} \left( \| \phi_2(-\log(-T),.) \|_{L^2(\BB)} + \| \psi_{a(-\log(-T)),2} - \psi_{a_\infty,2} \|_{L^2(\BB)} \right) \nonumber\\
&\lesssim |T|^{-\frac{1}{2}} \left( \| \Phi(-\log(-T)) \| + |a(-\log(-T))-a_\infty| \right) \nonumber \\
&\lesssim |T|^{-\frac{1}{2}} |T|^{\eep} = |T|^{-\varepsilon}.
\end{align}

 Hence by \eqref{eest1}--\eqref{eest2} we have
\begin{align}\label{v-estimate}
 |T|^\frac{1}{2} \left( \| v - v_{a_\infty} \|_{H^1(\Sigma_T)} + \| \nabla_n v - \nabla_n v_{a_\infty} \|_{L^2(\Sigma_T)} \right) \lesssim |T|^{\frac{1}{2}-},
\end{align}
 where we have used the factor $|T|^\frac{1}{2}$ because the norm of $v_{a_\infty}$ is of order $|T|^{-\frac{1}{2}}$ (see Remark~\ref{rem:norm})

\medskip
\textbf{Step 2. $v-v_{a_\infty}$ in localized Strichartz norms.}
 To obtain a decay estimate in terms of $(t,x)$ we fix $\delta >0$ and integrate over the cylinder $Z_t := (t,2t) \times B_{(1-\delta)t} \subseteq \RR \times \RR^3$. The time component $T = - \frac{t}{t^2-|x|^2}$ of the inverse Kelvin transform $\kappa^{-1}$ is a strictly increasing function,  which is why the domain $Z_t$ amounts to considering hyperboloidal slices $\Sigma_S$ with index $S$ in the interval
\begin{align*}
 - \frac{t}{t^2-|x|^2} < S := - \frac{s}{s^2-|x|^2} < - \frac{2t}{4t^2-|x|^2}.
\end{align*}
 Since $t^2-|x|^2 \geq t^2 (2\delta-\delta^2)$ and $4t^2 \geq 4t^2 - |x|^2$ we may consider a larger time interval and estimate $S$ independent of $x$ by
\[ -\frac{1}{t(2\delta-\delta^2)} < S < -\frac{1}{2t}. \]
 Similarly, we have to translate the domain $x \in B_{(1-\delta)t}$ for $x$ to $X$ and $S$. For $s \in (t,2t)$ this yields
\[
 |X| = \frac{|x|}{s^2-|x|^2} < \frac{(1-\delta)t}{s^2-|x|^2} < (1-\delta) \frac{s}{s^2-|x|^2} = (1-\delta) |S|.
\]
 From the Jacobi matrix of the Kelvin transform $\kappa \colon (T,X) \mapsto (t,x)$ computed in the proof of Lemma~\ref{lem:transform1} we obtain the determinant
\[
 \det D \kappa (T,X) = (T^2 - |X|^2)^{-4}.
\]
 Therefore,
\begin{align*}
 & \left\| v - v_{a_\infty} \right\|_{L^p(t,2t) L^p (B_{(1-\delta)t})}^p \\
&\qquad = \int_t^{2t} \int_{B_{(1-\delta)t}} \left|(v - v_{a_\infty})(s,x)\right|^p \, dx \, ds \\
&\qquad \leq \int_{-\frac{1}{t(2\delta-\delta^2)}}^{-\frac{1}{2t}} \int_{B_{(1-\delta)|S|}} \frac{|v\circ \kappa_S(X) - v_{a_\infty} \circ \kappa_S(X)|^p}{(S^2-|X|^2)^4} \, dX \, dS \\
&\qquad = \int_{-\frac{1}{t(2\delta-\delta^2)}}^{-\frac{1}{2t}} \int_{B_{(1-\delta)|S|}} \left| (u - u_{a_\infty})(X,S) \right|^p (S^2-|X|^2)^{p-4} \, dX \, dS.
\end{align*}
 The reformulation \eqref{M2} together with the coordinate transformation $Y^i := \frac{X^i}{-S}$ with $dX = |S|^3 dY$ yields
\begin{align}\label{Lq:1}
 \| v &- v_{a_\infty} \|_{L^p(t,2t) L^p (B_{(1-\delta)t})}^p \nonumber \\
&\leq \int_{-\frac{1}{t(2\delta-\delta^2)}}^{-\frac{1}{2t}} |S|^{p-5} \left( \int_{B_{1-\delta}} 
(1-|Y|^2)^{p-4} \left| (\psi - \psi_{a_\infty}) \left( - \log (-S),Y \right) \right|^p dY \right) dS \nonumber \\
&\lesssim \int_{-\frac{1}{t(2\delta-\delta^2)}}^{-\frac{1}{2t}} |S|^{p-5} \left\| (\psi - \psi_{a_\infty}) (-\log(-S),.) \right\|^p_{L^p(\BB)} \, dS
\end{align}
 By the Sobolev embedding $H^1 (\BB) \hookrightarrow L^p (\BB)$ for $p \in[1, 6]$, together with the conclusion that $\Phi \in \mathcal{X}^\varepsilon_{\tilde \delta}$, $a \in \mathcal{A}^\varepsilon_{\tilde \delta}$ (for sufficiently small $\tilde \delta$) from Theorem~\ref{thm:manifold}, we can further deduce that
\begin{align}\label{Lq:2}
 \| (\psi - \psi_{a_\infty}) (-\log(-S),.) \|_{L^p(\BB)} 
&\leq \left\| \phi_1 (-\log(-S),.) \right\|_{L^p(\BB)} + \|\psi_{a(-\log(-S))}-\psi_{a_\infty} \|_{L^p(\BB)} \nonumber \\
&\lesssim  \left\| \Phi (-\log(-S)) \right\| + |a(-\log(-S)) - a_\infty| \nonumber \\
&\lesssim |S|^{\eep},
\end{align}
 for $\varepsilon \in (0,\frac{1}{2})$. Together, \eqref{Lq:1}--\eqref{Lq:2} imply that
\begin{align*}
 \| v - v_{a_\infty} \|_{L^p(t,2t) L^p (B_{(1-\delta)t})}^p
 &\lesssim \int_{-\frac{1}{t(2\delta-\delta^2)}}^{-\frac{1}{2t}} |S|^{p\left(1+\beep\right)-5} dS
 \lesssim t^{4-p(\frac{3}{2}-\varepsilon)}.
\end{align*}
 where we assumed, without loss of generality, that $\varepsilon > 0$ is sufficiently small so that the exponent satisfies $p (\frac{3}{2}-\varepsilon)-5 > -1$. Thus we obtain the asymptotic decay
\begin{align}\label{pq-estimate}
 \| v - v_{a_\infty} \|_{L^p(t,2t) L^p (B_{(1-\delta)t})} \lesssim t^{-\left( \frac{3}{2}-\frac{4}{p} \right) + \varepsilon} =: t^{-\left( \frac{3}{2}-\frac{4}{p} \right) +},
\end{align}
 which for $p = 4$ this is the desired estimate
\[
 \left\| v - v_{a_\infty} \right\|_{L^4(t,2t) L^4 (B_{(1-\delta)t})} \lesssim t^{-\frac{1}{2}+}. \qedhere
\]
\end{proof}

\begin{remark}[More general Strichartz estimates] \label{otherpq}
 As can be seen from the proof of Theorem~\ref{maintheorem} above, in particular \eqref{pq-estimate}, it is possible to choose other $p$. For $p \in (\frac{8}{3},6]$ one obtains the decay estimate
\[
 \| v - v_{a_\infty} \|_{L^p(t,2t) L^p (B_{(1-\delta)t})} \lesssim t^{-\left( \frac{3}{2}-\frac{4}{p} \right) +}.
\]
The upper bound for $p$ is necessary to be able to use the Sobolev embedding $H^1(\BB)\hookrightarrow L^p(\BB)$.
\end{remark}

\begin{remark}[Formulation for $u$]\label{rem:transformation}
 In terms of the hyperboloidal coordinates $(T,X)$ the codimension-1 decay estimate of Theorem~\ref{maintheorem} translates to a codimension-1 stability result of the blowup solution $u_0(T,X) = - \sqrt{2}/T$ for the cubic wave equation~\eqref{cwe-u} in the backward light cone $\vartriangle_0$. In particular, in the course of the proof of Theorem~\ref{maintheorem} we show that a solution $u$ of \eqref{cwe-u} corresponding to initial data in $\M$ which approximates the blowup solution $u_{a_\infty}$ satisfies
\begin{align*}
 |T|^{\frac{1}{2}} \| u(T,.) - u_{a_\infty}(T,.) \|_{\dot H^1(B_{|T|})} &\lesssim |T|^{\frac{1}{2}-}, \\
 |T|^{-\frac{1}{2}}  \| u(T,.) - u_{a_\infty}(T,.) \|_{L^2(B_{|T|})} & \lesssim |T|^{\frac{1}{2}-}, \\
 |T|^{\frac{1}{2}} \| \partial_T u(T,.) - \partial_T u_{a_\infty}(T,.) \|_{L^2(B_{|T|})} &\lesssim |T|^{\frac{1}{2}-},
\end{align*}
for $T \in [-1,0)$ (cf.\ estimate \eqref{v-estimate} for $v$). This is the content of Theorem~\ref{maintheorem-blowup}.
\end{remark}

\subsection{Evolution for small initial data: Proof of Theorem~\ref{smalldatatheorem}}
\label{ssec4.6}

In Theorem~\ref{maintheorem} we have shown that solutions $v$ with initial data from the codimension-1 manifold $\M$ remain $t^{-\frac{1}{2}+}$ close to Lorentz transformations $v_a$ of the selfsimilar solution $v_0(t,x)=\sqrt{2}/t$, which themselves decay slowly.
On the contrary, the long time behavior of solutions to nonlinear wave equations with small initial data (that is, around $0$) is rather different. These solutions decay with $t^{-\frac{1}{2}}$.

\begin{proof}[Proof of Theorem~\ref{smalldatatheorem}]
 The proof proceeds along the same lines as the proof of Theorem~\ref{maintheorem}, however we compare solutions with small initial data simply to the trivial solution $v\equiv$ of \eqref{cwe}. This leads to several simplifications: The operator $\L'$ defined analogously to $\L_a'$ in \eqref{Laprime} vanishes identically, and the linearized equation therefore is just the free equation. More precisely, $\L + \L' = \L$ and the associated semigroup is the free semigroup $\S$ obtained in Proposition~\ref{prop:L} with exact decay
\begin{align}\label{S1}
 \| \S(\tau) \| \leq e^{-\frac{\tau}{2}}, \quad \tau \in [0,\infty).
\end{align}
 Hence the nonlinear treatment of the equation in Section~\ref{sec4} is analogous to that for Theorem~\ref{maintheorem} with the only difference that we have to use the decay \eqref{S1} instead of the one obtained in Proposition~\ref{prop:proj} (which required an additional $\varepsilon>0$). Thus for solutions $v$ corresponding to small initial data we obtain the decay estimate
\[
 \| v \|_{L^4(t,2t) L^4 (B_{(1-\delta)t})} = \| v - 0 \|_{L^4(t,2t) L^4 (B_{(1-\delta)t})} \lesssim t^{-\frac{1}{2}},
\]
 and more generally for $p\in[1,6]$
\[
 \| v \|_{L^p(t,2t) L^p (B_{(1-\delta)t})} \lesssim t^{-\left( \frac{3}{2} - \frac{4}{p} \right)}. \qedhere
\]
\end{proof}


\section*{Acknowledgments}

The authors would like to thank the anonymous referees for useful suggestions.
Financial support by the Sofja Kovalevskaja award of the Humboldt Foundation endowed by the German Federal Ministry of Education and Research, held by the second author, is acknowledged. Partial support by SFB 1060 of the German Research Foundation (DFG) is also gratefully acknowledged.



\begin{thebibliography}{99}
\providecommand{\url}[1]{\texttt{#1}}
\providecommand{\urlprefix}{URL }

\bibitem{BCT}
P.\ Bizo\'n, T.\ Chmaj and Z.\ Tabor, On blowup for semilinear wave equations with a focusing nonlinearity, {\it Nonlinearity} {\bf 17} (2004), no.\ 6, 2187--2201.

\bibitem{BZ}
P.\ Bizo\'n and A. Zengino\u{g}lu, Universality of global dynamics for the cubic wave equation, {\it Nonlinearity} {\bf 22} (2009), no.\ 10,
2473--2485.

\bibitem{DR}
M.\ Dafermos and I.\ Rodnianski, A new physical-space approach to decay for the wave equation with applications to black hole spacetimes, in \textit{XVIth International Congress on Mathematical Physics}, P.\ Exner (ed.), World Scientific, London, 2009, pp. 421--433, related article available online at \url{http://arxiv.org/abs/0910.4957}.

\bibitem{DS12}
R.\ Donninger and B.\ Sch\"orkhuber, Stable self-similar blow up for energy subcritical wave equations, {\it Dyn.\  Partial Differ.\ Equ.} {\bf 9} (2012), no.\ 1, 63--87.

\bibitem{DS}
R.\ Donninger and B.\ Sch\"orkhuber, On blowup in supercritical wave equations, {\it Comm.\ Math.\ Phys.} {\bf 346} (2016), no.\ 3, 907--943.

\bibitem{DZ}
R.\ Donninger and A.\ Zengino\u{g}lu, Nondispersive decay for the cubic wave equation,
{\it Anal.\ PDE} {\bf 7} (2014), no.\ 2, 461--495.

\bibitem{ES}
D.M.\ Eardley and L.\ Smarr, Time functions in numerical relativity: Marginally bound dust collapse, {\it Phys.\ Rev.\ D (3)} {\bf 19} (1979), no.\ 8, 2239--2259.

\bibitem{EN}
K.-J.\ Engel and R.\ Nagel, {\it A short course in operator semigroups}, Universitext, Springer, New York, 2006.

\bibitem{F}
H.\ Friedrich, Cauchy problems for the conformal vacuum field equations in general relativitiy, {\it Comm.\ Math.\ Phys.} {\bf 91} (1983), no.\ 4, 445--472.

\bibitem{G}
R.T.\ Glassey, Blow-up theorems for nonlinear wave equations, {\it Math.\ Z.} {\bf 132} (1973), 183--203.

\bibitem{K}
T.\ Kato, {\it Perturbation theory for linear operators}, Classics in Mathematics, Springer Verlag, Berlin, 1995. Reprint of the 1980 edition.

\bibitem{Kl}
S.\ Klainerman, Global existence of small amplitude solutions to nonlinear Klein-Gordon equations in four space-time dimensions, {\it Comm.\ Pure Appl.\ Math.} \textbf{38} (1985), no.\ 5, 631--641.

\bibitem{LM}
P.G.\ LeFloch and Y.\ Ma, {\it The hyperboloidal foliation method}, Series in Applied and Computational Mathematics, 2. World Scientific Publishing Co. Pte. Ltd., Hackensack, NJ, 2014.

\bibitem{LM1}
P.G.\ LeFloch and Y.\ Ma, The global nonlinear stability of Minkowski space for self-gravitating massive fields, {\it Comm.\ Math.\ Phys.} \textbf{346} (2016), no.\ 2, 603--665.

\bibitem{L}
H.A.\ Levine, Instability and nonexistence of global solutions to nonlinear wave equations of the form $Pu_{tt} =  -Au + \mathcal{F}(u)$, {\it Trans.\ Amer.\ Math.\ Soc.}  {\bf 192} (1974), 1--21.

\bibitem{LS}
H.\ Lindblad and A.\ Soffer, Scattering for the Klein--Gordon equation with quadratic and variable coefficient cubic nonlinearities, {\it Trans.\ Amer.\ Math.\ Soc.} {\bf 367} (2015), no.\ 12, 8861--8909. 

\bibitem{MZ}
F.\ Merle and H.\ Zaag, Determination of the blow-up rate for a critical semilinear wave equation, {\it Math.\ Ann.} {\bf 331} (2005), no.\ 2, 395--416.

\bibitem{MZ1}
F.\ Merle and H.\ Zaag, On growth rate near the blowup surface for semilinear wave equations, {\it Int.\ Math.\ Res.\ Not.} (2005), no.\ 19, 1127–1155.

\bibitem{MZ2}
F.\ Merle and H.\ Zaag, On the stability of the notion of non-characteristic point and blow-up profile for semilinear wave equations, {\it Comm.\ Math.\ Phys.} {\bf 333} (2015), no.\ 3, 1529--1562.

\bibitem{MZ3}
F.\ Merle and H.\ Zaag, Dynamics near explicit stationary solutions in similarity variables for solutions of a semilinear wave equation in higher dimensions, {\it Trans.\ Amer.\ Math.\ Soc.} {\bf 368} (2016), no.\ 1, 27--87.

\bibitem{NS}
K.\ Nakanishi and W.\ Schlag, {\it Invariant manifolds and dispersive Hamiltonian evolution equations}, Zurich Lectures in Advanced Mathematics. European Mathematical Society (EMS), Z\"urich, 2011.

\bibitem{P}
H.\ Pecher, Scattering for semilinear wave equations with small data in three space dimensions, {\it Math.\ Z.} {\bf 198} (1988), no.\ 2, 277--289.

\bibitem{St}
J.\ Sterbenz, Dispersive decay for the 1D Klein--Gordon equation with variable coefficient nonlinearities,  {\it Trans.\ Amer.\ Math.\ Soc.} \textbf{368} (2016), no.\ 3, 2081--2113.

\bibitem{W}
Q.\ Wang, An intrinsic hyperboloid approach for Einstein Klein--Gordon equations, {\tt arxiv:1607.01466} (2016).

\bibitem{Z}
A. Zengino\u{g}lu, Hyperboloidal foliations and scri-fixing, {\it Classical Quantum Gravity} {\bf 25} (2008), no.\ 14,
145002, 19 p.

\end{thebibliography}
\end{document}